\documentclass[11pt]{amsart}

\usepackage{CJKutf8}

\usepackage{amssymb}
\usepackage{amsmath}
  \usepackage{paralist}
\usepackage{enumitem}
\usepackage{comment}

\usepackage[colorlinks,
          linkcolor = blue,		
          anchorcolor = blue,	
          urlcolor = red,		
          citecolor = red,		
          ]{hyperref}

\usepackage{tikz}
\usetikzlibrary{arrows, automata}
\usetikzlibrary{calc}
\usetikzlibrary{intersections}

\tikzset{
    right angle quadrant/.code={
        \pgfmathsetmacro\quadranta{{1,1,-1,-1}[#1-1]}     
        \pgfmathsetmacro\quadrantb{{1,-1,-1,1}[#1-1]}},
    right angle quadrant=1, 
    right angle length/.code={\def\rightanglelength{#1}},   
    right angle length=1ex, 
    right angle symbol/.style n args={3}{
        insert path={
            let \p0 = ($(#1)!(#3)!(#2)$) in     
                let \p1 = ($(\p0)!\quadranta*\rightanglelength!(#3)$), 
                \p2 = ($(\p0)!\quadrantb*\rightanglelength!(#2)$) in 
                let \p3 = ($(\p1)+(\p2)-(\p0)$) in  
            (\p1) -- (\p3) -- (\p2)
        }
    }
}

\newtheorem*{TheoremA'}{Theorem A'}

\newtheorem{thmm}{Theorem}

\newtheorem*{TheoremD'}{Theorem D'}

\newtheorem*{TheoremE'}{Theorem E'}

\newtheorem*{CorollaryC1}{Corollary C.1}

\newtheorem{theorem}{Theorem}[section]
\newtheorem{corollary}[theorem]{Corollary}

\newtheorem*{main*}{Main Theorem}
\newtheorem{lemma}[theorem]{Lemma}
\newtheorem{proposition}[theorem]{Proposition}
\newtheorem{conjecture}[theorem]{Conjecture}

\newtheorem{question}[theorem]{Question}
\theoremstyle{definition}
\newtheorem{definition}[theorem]{Definition}
\newtheorem{remark}[theorem]{Remark}

\def\F{{\mathcal F}}
\def\I{{\mathcal I}}
\def\O{{\mathcal O}}

\def\G{{\mathcal G}}
\def\CL{{\mathcal L}}

\def\p{{\partial}}
\def\M{{\mathcal M}}

\newcommand{\R}{\mathcal{R}}
\def\pX{{\partial X}}
\def\RR{{\mathbb R}}
\def\NN{{\mathbb N}}

\def\ZZ{{\mathbb Z}}

\def\inj{{\text{inj}}}
\def\lx{{\underline x}}
\def\ly{{\underline y}}
\def\lv{{\underline v}}
\def\lc{{\underline c}}
\def\lw{{\underline w}}
\def\lB{{\underline B}}
\def\lS{{\underline S}}
\def\lV{{\underline V}}
\def\lN{{\underline N}}
\def\lm{{\underline m}}
\def\pr{{\text{pr}}}

\def\diam{\mathop{\hbox{{\rm diam}}}}

\def\diam{\mathop{\hbox{{\rm diam}}}}
\def\pr{\mathop{\hbox{{\rm pr}}}}
\def\supp{\mathop{\hbox{{\rm supp}}}}

\def\loc{{\mathop{\hbox{\footnotesize  \rm loc}}}}
\def\supp{\mathop{\hbox{{\rm supp}}}}
\def\top{{\mathop{\hbox{\footnotesize \rm top}}}}

\def\Sing{\mathop{\hbox{{\rm Sing}}}}
\def\Rec{\mathop{\hbox{{\rm Rec}}}}

\def\a{\alpha}
\def\b{\beta}
\def\c{\gamma}   \def\C{\Gamma}
\def\d{\delta}   
   \def\L{\Lambda}
\def\ve{\varepsilon} \def\e{\epsilon}

\def\ae{\text{-a.e.}\ }
\def\bP{\textbf{P}}
\def\bF{\textbf{F}}

\title[Running heading with forty characters or less]
      {On ergodic properties of geodesic flows on uniform visibility manifolds without conjugate points}

\author[first-name1 last-name1 and first-name2 last-name2]{Weisheng Wu}


\address{School of Mathematical Sciences, Xiamen University, Xiamen, 361005, P.R. China}
\email{wuweisheng@xmu.edu.cn}

\begin{document}

\begin{abstract}
In this paper, we conduct a comprehensive study on ergodic properties of the geodesic flow on a $C^\infty$ uniform visibility manifold $M$ without conjugate points. If $M$ is a closed surface of genus $\mathfrak{g}\geq 2$ without conjugate points and with bounded asymptote, we study the geometric properties of singular geodesics and show that if all singular geodesics are closed, then there are at most finitely many isolated singular closed geodesics and finitely many generalized strips. In particular, the geodesic flow is ergodic with respect to Liouville measure under the above assumption. 

Let $(M,g)$ be a closed uniform visibility manifold without conjugate points and $X$ be its universal cover. Under the entropy gap assumption, the geodesic flow has a unique measure of maximal entropy (MME for short) by \cite[Theorem 1.2]{MR}. We develop a Patterson-Sullivan construction of this unique MME and show that it has local product structure, is fully supported, and has the Bernoulli property.

If we assume further that $M$ has continuous Green bundles and the geodesic flow has a hyperbolic periodic point, using the nonuniform hyperbolic structure on an open dense subset and the symbolic approach developed in \cite{LP}, we show that for any H\"{o}lder continuous function $\psi$ or of the form $\psi = q\psi^u$ for $q\in \RR$ where $\psi^u$ is the Jacobian of the geodesic flow along unstable Green bundle, the equilibrium state of $\psi$ is unique under the pressure gap condition.

Under the same conditions above, we apply the mixing properties of the MME to obtain the following two Margulis type asymptotic formulae:
\begin{equation*}
\begin{aligned}
\#P(t)&\sim\frac{e^{ht}}{ht}\\
b_t(x)&\sim c(x)\frac{e^{ht}}{h}.
\end{aligned}
\end{equation*}
where $P(t)$ is the set of free-homotopy classes containing a closed geodesic with length at most $t$, $b_t(x)$ is the Riemannian volume of the ball of radius $t>0$ around $x\in X$, $h$ denotes the topological entropy of the geodesic flow, and the notation $f_1\sim f_2$ means that $\lim_{t\to\infty}\frac{f_1(t)}{f_2(t)}=1$. The function $c: X\to \RR$ is continuous and called the Margulis function. We then obtain some rigidity results on $M$ involving the Margulis function. 

Finally, for a uniform visibility manifold $M=X/\C$ (not necessarily compact) without conjugate points, we show that if $\C$ is non-elementary and contains an expansive isometry, then the theorem of Hopf-Tsuji-Sullivan dichotomy holds.
\end{abstract}

\maketitle
\markboth{Ergodic properties of geodesic flows}
{W. Wu}
\renewcommand{\sectionmark}[1]{}

\tableofcontents

\section{Introduction}
The geodesic flow is an important topic lying in the crossing field of dynamical systems and differential geometry. The geodesic flow on manifolds of negative curvature and nonpositive curvature is the most natural example of hyperbolic dynamical systems and
chaotic phenomenon, and its ergodic properties are extensively studied since the beginning of the last century. In this article, we present several new progresses in the ergodic theory of geodesic flows on manifolds without conjugate points, including the ergodicity with respect to Liouville measure, uniqueness and Bernoulli properties of measure of maximal entropy (MME for short) and equilibirum states, counting closed geodesics and volume growth asymptotics, Margulis functions and related rigidity phenomenon, and the Hopf-Tsuji-Sullivan dichotomy with respect to the Bowen-Margulis measure.

Suppose that $(M,g)$ is a $C^{\infty}$ closed $n$-dimensional Riemannian manifold,
where $g$ is a Riemannian metric. The geodesic flow describes the motion of a free particle on $M$. Suppose that a particle is at a point $p\in M$ with unit velocity $v\in S_pM:=\{v\in TM: \|v\|=1\}$. Then the particle will move along the geodesic that is tangent to $v$ and proceeds length $t$ by time $t$. Let $\pi: SM\to M$ be the unit tangent bundle over $M$. For each $v\in S_pM$,
we always denote by $c_{v}: \RR\to M$ the unique geodesic on $M$ satisfying the initial conditions $c_v(0)=p$ and $\dot c_v(0)=v$. The geodesic $c_{v}$ satisfies the geodesic equation
$$\nabla_{\dot c_v(t)}\dot c_v(t)=0,$$
where $\nabla$ is the Riemannian connection on $M$.
The geodesic flow $\phi=(\phi^{t})_{t\in\mathbb{R}}$ (induced by the Riemannian metric $g$) on $SM$ is defined as:
\[
\phi^{t}: SM \rightarrow SM, \qquad (p,v) \mapsto
(c_{v}(t),\dot c_{v}(t)),\ \ \ \ \forall\ t\in \RR.
\]

A vector field $J(t)$ along a geodesic $c:\RR\to M$ is called a \emph{Jacobi field} if it satisfies the \emph{Jacobi equation}:
\[J''+R(J, \dot c)\dot c=0\]
where $R$ is the Riemannian curvature tensor and\ $'$\ denotes the covariant derivative along $c$.

\begin{definition}
Let $c$ be a geodesic on $(M,g)$.
\begin{enumerate}
  \item A pair of distinct points $p=c(t_{1})$ and $q=c(t_{2})$ are called \emph{focal} if there is a Jacobi field $J$ along $c$ such that $J(t_{1})=0$, $J'(t_{1})\neq 0$ and $\frac{d}{dt}\mid_{t=t_{2}}\| J(t)\|^{2}=0$;
  \item $p=c(t_{1})$ and $q=c(t_{2})$ are called \emph{conjugate} if there is a nontrivial Jacobi field $J$ along $c$ such that $J(t_{1})=0=J(t_{2})$.
\end{enumerate}
A compact Riemannian manifold $(M,g)$ is called a manifold \emph{without focal points/without conjugate points} if there is no focal points/conjugate points on any geodesic in $(M,g)$.
\end{definition}
By definition, if a manifold has no focal points then it has no conjugate points. All manifolds of nonpositive curvature always have no focal points. In brief, we have
\begin{equation*}
    \begin{aligned}
 \text{negative curvature}&\Rightarrow  \text{nonpositive curvature}  
 \Rightarrow \text{no focal points} \\&\Rightarrow \text{no conjugate points}.       
    \end{aligned}
\end{equation*}

There is a natural isomorphism between the tangent space $T_v(SM)$ and the space of perpendicular Jacobi fields along the geodesic $c_v$ where $v\in SM$. The \textit{stable/unstable Green bundles} $G^s(v)/G^u(v), v\in SM$ are subbundles of $TSM$ associated to the two distinguished families of Jacobi fields constructed by Green \cite{Gre1}, which we call \textit{stable/unstable Jacobi fields}. In general $G^s(v)/G^u(v)$ are only measurable in $v$, and they are not necessarily transverse to each other at $v$.

The statistical properties of geodesic flows on surfaces with negative curvature were first studied by Hadamard and Morse in the beginning of the twentieth century. In the 1960s, Anosov \cite{An, AnS} introduced the notion of the so-called \textit{Anosov flows} by abstracting the hyperbolic behaviour of the geodesic flows of negatively curved manifolds. Suppose that $M$ has negative curvature everywhere $K<0$. Then the geodesic flow $\phi^t$ is an \textit{Anosov flow}: there exist a $\phi^t$-invariant $T_vSM=Z(v)\oplus G^s(v)\oplus G^u(v)$, constants $0<\lambda<1$, $c>0$ such that
\begin{enumerate}
  \item $Z(v)$ is the vector field tangent to the orbit of the flow;
  \item $\|d\phi^t(\xi)\| \leq c\lambda^t$ for $t>0$, whenever $\xi \in G^s(v)$,
  \item $\|d\phi^{-t}(\eta)\| \leq c\lambda^t$ for $t>0$,  whenever $\eta \in G^u(v)$.
\end{enumerate}
$T_vSM=Z(v)\oplus G^s(v)\oplus G^u(v)$ is called the Anosov splitting associated with the Anosov flow. In this negative curvature case, the strong stable/unstable bundles of Anosov spliting $G^s(v)/G^u(v)$ are given exactly by Green bundles, which is proved to be H\"{o}lder continuous. Furthermore it is known that they are uniquely integrable to foliations $W^s$ and $W^u$ of $SM$, called the strong stable/unstable foliations of $\phi^t$. 

Geodesic flows on closed manifolds with negative curvature have very rich dynamics and broad applications. In the last century, this class of geodesic flows have always been attracting great interests of the mathematicians in dynamical systems and related areas. A lot of beautiful results on the dynamics of the geodesic flows have been exhibited. Among which, the ergodicity and mixing properties of Liouville measure, the uniqueness of measure of maximal entropy etc., have their special importance and receive extensive attentions.  

Geodesic flows on manifolds of nonpositive curvature $K\le 0$ have also been extensively studied since 1970's. The study for the geodesic flows become much more difficult due to the existence of parallel Jacobi fields.  
A Jacobi field $J(t)$ along a geodesic $c(t)$ is called \emph{parallel} if $J'(t)=0$ for all $t\in \RR$. The notion of \emph{rank} is defined as follows.
\begin{definition}
For each $v \in SM$, we define \text{rank}($v$) to be the dimension of the vector space of parallel Jacobi fields along the geodesic $c_{v}$, and \text{rank}($M$):=$\min\{$\text{rank}$(v): v \in SM\}$. 
\end{definition}

In fact by the higher rank rigidity theorem \cite{Bal, BBE, BBS, BS1, BS2}, any simply connected irreducible Riemannian manifold of nonpositive curvature and of rank greater than one is isometric to a globally symmetric space of noncompact type. This theorem was extended by Watkins \cite{Wat} to manifolds without focal points. On the other hand, for rank one manifolds of nonpositive curvature or more generally without focal points, the geodesic flows present certain \emph{non-uniformly hyperbolic} behaviors. Ergodic properties including the ergodicity with respect to Liouville measure and the uniqueness of MME attract many interests and fruitful results were achieved in the last several decades (see \cite{Kn1, Kn2, Ba, Wu, BCFT, GR0, Link1, LP, CKP1, CKP2, LWW, WLW, Wu1, Wu2} and many others). 

Even weaker hyperbolicity can be enjoyed by the geodesic flows on manifolds without conjugate points. The study of its ergodic properties 
faces a great challenge and relatively few results are known so far (see for example \cite{CKW1, CKW2, GR, MR, Mamani, Wu1, PW} etc.) Our goal in this paper is to conduct a comprehensive study on ergodic properties of geodesic flows on various classes of manifolds without conjugate points.

\section{Statement of main results}
For geodesic flows on manifolds without conjugate points, the very weak hyperbolicity prevents us to use rich results from hyperbolic dynamics.  On the other hand, in contrast with manifolds with negative or nonpositive curvature, strong geometric properties such as convexity and monotonicity, or the powerful tools such as analysis involving the curvature are not available. Nevertheless, Eberlein \cite{Eb1, EO} introduced the following uniform visibility axiom, which provides us a rather good geometric property of the manifold.
\begin{definition}(Cf. \cite[Definition 2.1]{CKW2})\label{vis}
A simply connected Riemannian manifold $X$ is a \emph{uniform visibility manifold} if for every $\e>0$ there exists $L=L(\e) > 0$ such that whenever a geodesic segment $c: [a, b]\to X$ stays at distance at least $L$ from some point $p \in X$,
then the angle sustained by $c$ at $p$ is less than $\e$, that is,
$$\angle_p(c) := \sup_{a\le s,t\le b}\angle_p(c(s), c(t)) < \e.$$
If $M$ is a Riemannian manifold without conjugate points whose universal cover $X$ is a uniform visibility manifold, then we say that $M$ is a \textit{uniform visibility} manifold.
\end{definition}
Let $M$ be a compact manifold with nonpositive curvature $K\le 0$. Then $M$ is a uniform visibility manifold if and only if the universal cover $X$ of $M$ contains no totally geodesic isometric imbedding of the plane $\RR^2$, see \cite{Eb1}.
 
Now we introduce various classes of  manifolds without conjugate points, coming from the Green bundles or Jacobi equations. Let $M$ be a Riemannian manifold without conjugate points. As mentioned above, in general, $G^s(v)/G^u(v)$ are only measurable in $v$, and they are not necessarily transverse to each other at $v$. 
Let us define 
\begin{equation}\label{set1}
\mathcal{R}_1:=\{v\in SM: G^s(v)\cap G^u(v)=\{0\}\}.
\end{equation}
We say that $M$ has \textit{continuous Green bundles} or \textit{continuous asymptote} if $G^s(v)/G^u(v)$ are continuous in $v$. In this case, it is clear that $\R_1$ is an open subset of $SM$.

We say that $M$ has \textit{bounded asymptote} if there exists $C>0$ such that $\|J^s(t)\|\le C\|J^s(0)\|$ for any $t>0$ and any stable Jacobi field $J^s$. We have (cf. \cite[5.3 Satz]{Kn}, \cite{GR})
\begin{equation*}
    \begin{aligned}
 \text{no focal points}\Rightarrow  \text{bounded asymptote}  \Rightarrow \text{continuous asymptote}.       
    \end{aligned}
\end{equation*}

In Section 3, we will recall the Busemann functions and the construction of stable/unstable horospherical foliations, denoted by $\mathcal{F}^s(v)/\mathcal{F}^u(v), v\in SM$, for manifolds without conjugate points. Let $X$ be a simply connected uniform visibility manifold without conjugate points. Then $\mathcal{F}^s(v), \mathcal{F}^u(v)$ are both continuous in $v\in SX$. Denote $\mathcal{I}(v)=\mathcal{F}^s(v)\cap\mathcal{F}^u(v)$. If $\mathcal{I}(v)=\{v\}$, then $v$ is called an \textit{expansive vector}. Denote the set of expansive vectors by
\begin{equation*}
\mathcal{R}_0:=\{v\in SM:\mathcal{I}(v)=\{v\}\}.    
\end{equation*}
These two sets $\R_0$ and $\R_1$ are central objects in our study of dynamical properties of the geodesic flow.

\subsection{Ergodicity of Liouville measure}
Let $\mathcal{M}_{\phi}(SM)$ be the set of $\phi^t$-invariant Borel probability measures on $SM$. We say that the geodesic flow is \textit{ergodic} with respect to $\mu\in \mathcal{M}_{\phi}(SM)$ or just $\mu$ is ergodic, if $\mu(\phi^t(A))=\mu(A)$ for any Borel measurable set $A\subset SM$ and any $t\in \RR$. Equivalently, by Birkhoff ergodic theorem, $\mu$ is ergodic if and only if for $\mu \ae v\in SM$
  $$\lim_{t\to \pm \infty}\frac{1}{t}\int_0^tf(\phi^s(v))ds=\int_{SM} f d \mu, \quad \forall f\in C(SM,\mathbb R).$$
Here $C(SM,\mathbb R)$ denotes the Banach space of all continuous real-valued functions on $SM$.
  
The most natural invariant measure for the geodesic flow is the Liouville measure $\nu$, which is the normalized volume measure induced by the Sasaki metric on $SM$. Hopf (cf. \cite{Ho0,Ho1}) proved the ergodicity of the geodesic flow with respect to the Liouville measure $\nu$ on $SM$ for compact surfaces of variable negative curvature and for compact manifolds of constant negative sectional curvature in any dimension. The general case for compact manifolds of variable negative curvature was established by Anosov and Sinai (cf. \cite{An, AnS}). The ergodicity was established based on the classical Hopf argument using stable/unstable foliations on $SM$ (cf., for example the appendix in \cite{BB}). 

In 1970's, by using his theory of nonuniform hyperbolicity, Pesin obtained a celebrated result on the ergodicity of the geodesic flows on manifolds without focal points satisfying the uniform visibility property (cf. Theorem 12.2.12 in \cite{BP}). Define the sets:
\begin{equation}\label{e:Pesin}
\begin{aligned}
\Delta^+&:=\{v\in SM: \chi(v,\xi)<0 \text{\ for any\ } \xi\in G^s(v)\},\\
\Delta^-&:=\{v\in SM: \chi(v,\xi)>0 \text{\ for any\ } \xi\in G^u(v)\},\\ 
\Delta&:=\Delta^+\cap \Delta^-,
\end{aligned}
\end{equation}
where $\chi$ denotes the Lyapunov exponents, and $G^s/G^u$ the stable/unstable Green bundles. $\Delta$ is called the \emph{regular set} with respect to the geodesic flow. Pesin proved the following theorem:
\begin{theorem}[Pesin, cf. \cite{BP}]\label{pesin} For the geodesic flow on a closed surface of genus $\mathfrak{g}\geq 2$ and without focal points, we have that
$\nu(\Delta)>0$, and $\phi^t|_\Delta$ is ergodic.
\end{theorem}
\vspace{.1cm}
However, even for surfaces of nonpositive curvature, the ergodicity of the geodesic flows with respect to the Liouville measure $\nu$ on $SM$ is widely open since 1980's, due to the existence of the following ``flat'' geodesics. 

Consider a closed surface $M$ of genus $\mathfrak{g} \geq 2$ and of nonpositive curvature. Let
\begin{equation*}\label{e:flat}
\begin{aligned}
\Lambda:=\{v\in SM: K(c_v(t))= 0, \ \forall t\in \mathbb{R}\},
\end{aligned}
\end{equation*}
where $K$ denotes the curvature of the point. We call $c_v$ a \emph{singular or flat geodesic} if $v\in \Lambda$. People still do not know if $\Lambda$ is small in measure ($\nu(\Lambda)=0$ or not), in general. However, from the dynamical point of view, Knieper \cite{Kn1} showed there is an entropy gap in arbitrary dimensions:
$$h_{\top}(\phi^1|_{\Lambda})<h_{\top}(\phi^1),$$
where $h_{\top}$ denotes the topological entropy, and $\Lambda$ denotes the irregular set of the geodesic flow which is a counterpart of the above defined set in arbitrary dimensions. Recently, Burns-Climenhaga-Fisher-Thompson \cite{BCFT} reproved the entropy gap using specification. This means that the geodesic flow restricted on $\Lambda$ has less complexity than the whole system. Knieper \cite{Kn1} (see also Burns and Gelfert \cite{BG}) proved that on rank $1$ surfaces of nonpositive curvature, the geodesic flow on $\Lambda$ has zero topological entropy. In higher dimensions, it is possible to have positive entropy on $\Lambda$; an example was given by Gromov (cf. \cite{Gro}).

For geodesic flows on rank $1$ surfaces of nonpositive curvature, the orbits inside $\Lambda$ are also believed to have simple behavior. In all the known examples, all the orbits in $\Lambda$ are closed. In a recent survey, Burns asked whether there exists a non-closed flat geodesic (cf. Question 6.2.1 in \cite{BM}). 
\begin{conjecture}[Cf. \cite{BM, Ha, RH} etc.]\label{conjecture}
Let $(M,g)$ be a smooth, connected and closed surface of genus $\mathfrak{g}\geq 2$, which has nonpositive curvature. Then all flat geodesics are closed and there are only finitely many homotopy classes of such geodesics. In particular, $\nu(\Lambda)=0$, and hence the geodesic flow on $SM$ is ergodic.
\end{conjecture}

The conjecture has been verified for surfaces of genus $\mathfrak{g}\geq 2$ with nonpositive curvature  \cite{Wu} and without focal points \cite{WLW}, assuming that the negative curvature part has finitely many connected components:
\begin{theorem}[Cf. \cite{Wu, WLW}]\label{wu}
Let $(M,g)$ be a smooth, connected and closed surface of genus $\mathfrak{g}\geq 2$, without focal points. Suppose that the set $\{x\in M: K(x)<0\}$ has finitely many connected components, then there are only finitely many homotopy classes of flat geodesics and thus $\nu(\Lambda)=0$. In particular, the geodesic flow is ergodic.
\end{theorem}

In this paper, we study the ergodicity of the geodesic flows on surfaces of genus $\mathfrak{g}\geq 2$ without conjugate points and with continuous Green bundles. We first show the relation between the regular set $\Delta$ and the set $\mathcal{R}_1$ of unit vectors with linearly independent stable and unstable Green bundles, see \eqref{set1}. We call the geodesic $c_v$ \textit{singular}, if $v\in \R_1^c$. To achieve the goal, we explore the properties of singular geodesics, which are also of independent interest. Our results are summarized in the following theorem. Here, we let $\text{Per}(\phi^t)$ denote the set of periodic points of the geodesic flow. 
\begin{thmm}\label{singulargeodesic}
Let $(M,g)$ be a smooth, connected and closed surface of genus $\mathfrak{g}\geq 2$ without conjugate points and with continuous Green bundles, then $\Delta$ agrees $\nu$-almost everywhere with the open dense set $\mathcal{R}_1$. If furthermore $M$ has bounded asymptote, then:
\begin{enumerate}
       \item Let $c_v$ be a lift of a closed geodesic and $v\in \R_1^c$. Then an ideal triangle with $c_v$ as an edge has infinite area;
   \item  If $\R_1^c \subset \text{Per}(\phi^t)$, then there is a finite decomposition of $\R_1^c$:
$$\R_1^c = \mathcal{O}_1 \cup \mathcal{O}_2 \cup \ldots \mathcal{O}_k \cup \mathcal{S}_1\cup \mathcal{S}_2 \cup \ldots \cup \mathcal{S}_l,$$
where each $\mathcal{O}_i, 1\leq i \leq k$, is an isolated periodic orbit and each $\mathcal{S}_j, 1\leq j \leq l$, consists of vectors tangent to a generalized strip. Here $k$ or $l$ are allowed to be $0$ if there is no isolated closed singular geodesic or no generalized strip. In particular, $\nu(\R_1^c)=0$, and hence the geodesic flow on $SM$ is ergodic.
   \end{enumerate}
\end{thmm}

We ask the following question for the future study.
\begin{question}
Let $(M,g)$ be a smooth, connected and closed surface of genus $\mathfrak{g}\geq 2$, without conjugate points, and with continuous Green bundles. Suppose that the set $(\pi(\mathcal{R}_1^c))^c$ has finitely many connected components. Does there exist a non-closed singular geodesic? If no, we have $\mathcal{R}_1^c\subset \text{Per}(\phi^t)$ and hence by Theorem \ref{singulargeodesic}, the geodesic flow is ergodic. 
\end{question}
One essential difficulty is that the flat strip lemma fails for surfaces without conjugate points \cite{Burns}, and there is lack of control on the geometric shape of the generalized strips, see for example the proof of Lemmas \ref{boundary} and \ref{flat closed geodesic} and Theorem \ref{A}. Moreover, we cannot effectively characterize the singular geodesics using curvatures through Riccati equations.

\subsection{Measure of maximal entropy and Bowen-Margulis measures}
Metric entropy introduced by Kolmogorov \cite{Kol} and Sinai, is one of the most important invariant describing the complexity of a dynamical system. Later, Adler-Konheim-McAndrew \cite{AKM} introduced the topological entropy for a topological dynamical system.

\begin{theorem}[Variational principle for entropy]\label{VPentropy}
Let $T:Y\to Y$ be a continuous self-map on a compact metric space $(Y,d)$. Then
\begin{equation*}
\begin{aligned}
h_{top}(T)&=\sup\{h_{\mu}(T): \mu\in \mathcal{M}_T(Y)\}\\
&=\sup\{h_{\mu}(T): \mu\in \mathcal{M}_T^e(Y)\}.
\end{aligned}
\end{equation*}
Here $\mathcal{M}_T(Y)$ (resp. $\mathcal{M}^e_T(Y)$) denotes the set of all $T$-invariant (resp. ergodic) probability measures on $Y$, $h_{top}(T)$ denotes the topological entropy of $T$, and $h_{\mu}(T)$ the metric entropy of $\mu$. 
\end{theorem}
A measure $m\in \mathcal{M}_T(Y)$ is called a \emph{measure of maximal entropy} (MME for short) if $h_m(T)=h_{top}(T)$. The definition applies to a flow by considering its time-one map. MME (when it exists) is another important class of distinct invariant measures, which reflects the complexity of the whole system.

When a closed Riemannian manifold $(M,g)$ has negative curvature, the geodesic flow on $SM$ is a prime example of Anosov flow (cf. \cite{An}, \cite[Section 17.6]{KH}). For Anosov flows, Margulis \cite{Mar2} in his thesis constructed an invariant measure under the geodesic flow. This measure is mixing, and the conditional measures on stable/unstable manifolds contract/expand with a uniform rate under the flow. Using different methods, Bowen constructed a MME in \cite{Bo1} for Axiom A flows. Later in \cite{Bo2}, Bowen proved that MME is unique. It turns out that the measures constructed by Bowen and Margulis are eventually the same one, which is called \emph{Bowen-Margulis measure}. Bowen-Margulis measure has the Bernoulli property \cite{Rat}, an ultimate mixing property for measurable dynamical systems. The uniqueness of MME becomes a classical problem in smooth ergodic theory. For uniformly hyperbolic systems, there are at least two typical methods, one using Markov partitions and symbolic dynamics (\cite{Bo4}), the other using topological properties such as expansiveness and specification (\cite{Bo3, Franco}).

In 1984, A.~Katok \cite{BuKa} conjectured that the geodesic flow on a closed rank one manifold of nonpositive curvature admits a unique MME. In 1998, Knieper \cite{Kn2} proved Katok's conjecture. He used Patterson-Sullivan measures on the boundary at infinity of the universal cover of $M$ to construct a measure (which is called Knieper measure), and showed that this measure is the unique MME. 

Beyond nonpositive curvature, manifolds without focal/conjugate points are studied extensively. Following Knieper's method, Liu, Wang and the author construct the Knieper measure for manifolds without focal points and showed it is
the unique MME \cite{LWW}. Gelfert and Ruggiero \cite{GR0} also proved the uniqueness of MME for surfaces without focal points
by constructing an expansive factor of the geodesic flow. 

For general manifolds without conjugate points, uniqueness of MME is far from being well understood. The first remarkable progress was made by Climenhaga, Knieper and War for a class of manifolds (including all surfaces of genus at least $2$) without conjugate points. In \cite{CKW1}, the authors proved the uniqueness of MME for the following manifolds in \emph{class $\mathcal{H}$}.
\begin{definition}\label{classH}
We say that compact manifolds without conjugate points satisfying the following conditions are in \emph{class $\mathcal{H}$}:
\begin{enumerate}
  \item There exists a Riemannian metric $g_0$ on $M$ for which all sectional curvatures are negative;
  \item $M$ has uniform visibility property;
  \item The fundamental group $\pi_1(M)$ is residually finite: the intersection of its finite index subgroups is trivial;
  \item There exists $h_0 < h$ such that any ergodic invariant Borel probability measure $\mu$ on $SM$ with entropy strictly greater than $h_0$ is almost expansive (cf. \cite[Definition 2.8]{CKW2}).
\end{enumerate}
\end{definition}
All compact surfaces of genus at least $2$ without conjugate points belong to the class $\mathcal{H}$. 

Very recently, another remarkable progress was made in a series of papers by Gelferet-Ruggiero \cite{GR}, Mamani \cite{Mamani} and Mamani-Ruggiero \cite{MR}:
\begin{theorem}(\cite[Theorem 1.2]{MR})\label{mmeuni}
Let $(M,g)$ be a closed $C^\infty$ uniform visibility manifold without conjugate points. Suppose that 
$$\sup\{h_\mu(\phi^1): \mu \text{\ is supported on\ } \R_0^c\}<h_{\top}(\phi^1).$$ Then the geodesic flow has a unique MME.
\end{theorem}

The entropy gap assumption above is equivalent to the condition (4) for manifolds in \emph{class $\mathcal{H}$}. It is satisfied automatically for surfaces:
\begin{theorem}(\cite[Theorem 1.2]{Mamani})
Let $(M,g)$ be a closed $C^\infty$ surface without conjugate points of genus greater than one. Then the geodesic flow has a unique MME.
\end{theorem}

We remark that Theorem \ref{mmeuni} applies to a class of manifolds larger than class $\mathcal{H}$, since conditions (1) and (3) for class $\mathcal{H}$ are dropped. The authors prove that there exists an expansive factor of the geodesic flow acting on a compact metric space which is topologically mixing and has local product structure and hence has unique MME.  On the other hand, Climenhaga, Knieper and War \cite{CKW1} used Morse lemma to get coarse estimates, and successfully applied Climenhaga-Thompson's criterion \cite{CT1} using a nonuniform version of specification and expansiveness. Climenhaga-Knieper-War \cite{CKW1} also construct the unique MME via Patterson-Sullivan measures and prove that the MME is mixing.

In this paper, for a closed uniform visibility manifold without conjugate points, we give an explicit construction of a Bowen-Margulis measure via Patterson-Sullivan measures on the boundary at infinity. Since the generalized strip does not necessarily possess a submanifold structure, there is no volume along a generalized strip. Instead, we use a measurable choice theorem to construct a Borel measure along a generalized strip. We show that a Bowen-Margulis measure constructed in this way has metric entropy equal to the topological entropy, and hence coincides with the unique MME obtained in Theorem \ref{mmeuni}. This shows that the unique MME has local product structure. In fact, our Patterson-Sullivan construction of Bowen-Margulis measure also works for noncompact uniform visibility manifold without conjugate points, which allows us to consider the Hopf-Tsuji-Sullivan dichotomy with respect to such Bowen-Margulis measure in the last section.

As an application of the local product structure of Bowen-Margulis measure, we show that it has the Bernoulli property, which is the ultimate mixing property for measurable dynamical systems. In nonpositive curvature, the Bowen-Margulis measure is proved to be mixing in \cite{Ba} and Bernoulli in \cite{CT}. For manifolds without focal points, the Bowen-Margulis measure is known to be mixing by \cite{LLW}, have the Kolmogorov property by \cite{CKP2} and the Bernoulli property by \cite{Wu2}. We apply the method of Call and Thompson \cite{CT} to prove the Kolmogorov property by showing that the product system of the expansive factors has a unique MME. To lift the Kolmogorov property to the Bernoulli, we adapt the classical argument in \cite{OW1, Rat} for Anosov flows. In fact, we follow closely \cite{CH}, which is for systems with nonuniformly hyperbolic structure.
\begin{thmm}\label{bernoulli}
Let $(M,g)$ be a closed $C^\infty$ uniform visibility manifold without conjugate points. Suppose that 
$$\sup\{h_\mu(\phi^1): \mu \text{\ is supported on\ } \R_0^c\}<h_{\top}(\phi^1).$$
Then the unique MME is fully supported and has the Bernoulli property.
\end{thmm}

\subsection{Uniqueness of equilibrium states}
First introduced by Ruelle \cite{R} in 1973 for an expansive dynamical system and later studied by Walters \cite{W2} for general systems, pressure and equilibrium states are natural generalization of entropy and MME, provided a continuous potential function on the phase space. They have played a fundamental role in ergodic theory, dynamical systems and dimension theory (\cite{Bo1,R1,W,Pesin}, etc.). The uniqueness of equilibrium states for H\"{o}lder continuous potentials was shown for uniformly hyperbolic systems in the classical works of Bowen \cite{Bo3, Bo4}. Moreover, the unique equilibrium state has the Bernoulli property.

Recently, Burns, Climenhaga, Fisher and Thompson \cite{BCFT} proved the uniqueness of equilibrium states for certain potentials in nonpositive curvature, which generalized Knieper's result. Extending the approach in \cite{BCFT}, Chen, Kao and Park proved the uniqueness of equilibrium states for certain potentials in \cite{CKP1, CKP2} for no focal points case. These results are obtained via Climenhaga-Thompson approach \cite{CT1} using nonuniform specification and expansiveness. However, there are essential technical gaps in this approach to study uniqueness of equilibrium states for no conjugate points case.

Recently, Lima-Poletti \cite{LP} reproved the uniqueness of equilibrium states in nonpositive curvature. They developed a countable symbolic coding of the systems and thus obtained the uniqueness of equilibrium states. Based on the nununiform hyperbolicity of the system, we are able to apply Lima-Poletti's approach in no conjugate points case:
\begin{thmm}\label{uniqueness}
Let $(M,g)$ be a closed $C^\infty$ uniform visibility manifold without conjugate points and with continuous Green bundles, and let $\psi:SM\to \RR$ be H\"{o}lder continuous or of the form $\psi = q\psi^u$ for $q\in \RR$. Suppose that the geodesic flow has a hyperbolic periodic point. If $P(\R_1^c, \psi) < P(\psi)$, then $\psi$ has a unique equilibrium state $\mu$, which is hyperbolic, fully supported and has the Bernoulli property.
\end{thmm}

In the above theorem, $\psi^u$ is the \textit{geometric potential} defined as 
\[\psi^u(v):=-\frac{d}{dt}\Big|_{t=0}\log \det (d\phi^t\mid_{G^u(v)}).\]
We also have the following corollary for MME,  which recovers \cite[Part 2 of Theorem 1.4]{MR}.
\begin{CorollaryC1}\label{uniqueness1}
Let $(M,g)$ be a closed $C^\infty$ uniform visibility manifold without conjugate points and with continuous Green bundles. Suppose that the geodesic flow has a hyperbolic periodic point. Then the geodesic flow has a unique MME, which is hyperbolic, fully supported and has the Bernoulli property.
\end{CorollaryC1}
Compare the above corollary with Theorem \ref{bernoulli}. In the setting of Theorem \ref{bernoulli}, we only know that $\R_0$ has full measure with respect to the unique MME. However, we do not necessarily have $\R_1\subset \R_0$. Thus we do not know whether the MME is hyperbolic and the symbolic approach \cite[Theorem 1.2]{ALP} is not applicable.

\subsection{Counting closed geodesics and volume asymptotics}
\subsubsection{Counting closed geodesics}
There is correspondence between closed geodesics on $M$ and periodic orbits under the geodesic flow, and they are important from both geometric and dynamical points of view. Let $P(t)$ be the set of free-homotopy classes containing a closed geodesic with length at most $t$. If $M$ has negative curvature, each free-homotopy class contains exactly one closed geodesic, so $\#P(t)$ is just the number of closed geodesics with length at most $t$. If $M$ has nonpositive curvature or without focal/conjugate points, a free-homotopy class may have infinitely many closed geodesics.

We obtain in this paper the following Margulis-type asymptotic estimates of $\#P(t)$.
\begin{thmm}\label{margulis}
Let $(M,g)$ be a closed $C^\infty$ uniform visibility manifold without conjugate points and with continuous Green bundles. Suppose that the geodesic flow has a hyperbolic periodic point. Then
\begin{equation}\label{e:mar1}
\begin{aligned}
\#P(t)\sim\frac{e^{ht}}{ht}
\end{aligned}
\end{equation}
where $h=h_{\text{top}}(g)$ denotes the topological entropy of the geodesic flow on $SM$, and the notation $f_1\sim f_2$ means that $\lim_{t\to\infty}\frac{f_1(t)}{f_2(t)}=1$.
\end{thmm}

In negative curvature, the formula \eqref{e:mar1} of asymptotic growth of the number of closed geodesics is given by the celebrated theorem of Margulis in his 1970 thesis \cite{Mar1, Mar2}. The main tool in the proof of Margulis' theorem is the ergodic theory of Bowen-Margulis measure, which is the unique MME for the Anosov flow. Margulis \cite{Mar2} gave an explicit construction of this measure, and showed that it is mixing, and the conditional measures on stable/unstable manifolds have the scaling property, (i.e., contract/expand with a uniform rate under the geodesic flow), and are invariant under unstable/stable holonomies. Margulis \cite{Mar2} then proved \eqref{e:mar1} using these ergodic properties of the Bowen-Margulis measure.

There are other approaches to formula \eqref{e:mar1} in various settings. In the 1950's Huber \cite{Hu} proved it for closed surfaces of constant negative curvature using Selberg's trace formula. Another alternative approach to \eqref{e:mar1} for Axiom A flows is given by Parry and Pollicott \cite{PP}, based on zeta functions and symbolic dynamics. See the survey by Sharp in \cite{Mar2} for further progress that has been made since Margulis' thesis.

It took a long time to establish formula  \eqref{e:mar1} for the geodesic flow on a closed rank one manifold of nonpositive curvature, which exhibits nonuniformly hyperbolic behavior. In 1982, Katok \cite{Ka} used a closing lemma in Pesin theory to show that for a smooth flow with topological entropy $h>0$ and no fixed points, the exponential growth rate of the number of periodic orbits is at least $h$. During the course of setting Katok's conjecture on the uniqueness of MME, Knieper constructed Patterson-Sullivan measure and used it to obtain the following asymptotic estimates \cite{Kn1, Kn3}: there exists $C>0$ such that for $t$ large enough
$$\frac{1}{C}\frac{e^{ht}}{t}\le \#P(t)\le C \frac{e^{ht}}{t}.$$
Since then, efforts are made to improve the above to the Margulis-type asymptotic estimates \eqref{e:mar1}. A preprint \cite{Gun}, though not published, contains many inspiring ideas to this problem. A breakthrough was made recently by Ricks \cite{Ri}: \eqref{e:mar1} is established for rank one locally CAT$(0)$ spaces, which include rank one manifolds with nonpositve curvature.

Beyond nonpositive curvature, remarkable progress was made by Climenhaga, Knieper and War \cite{CKW2} where \eqref{e:mar1} is obtained for manifolds without conjugate points in class $\mathcal{H}$. Weaver \cite{We} also obtained \eqref{e:mar1} for surfaces having negative curvature outside of a collection of radially symmetric caps. The author \cite{Wu2} obtained \eqref{e:mar1} for rank one manifolds without focal points.

The following is a result of equidistribution of periodic orbits with respect to the Bowen-Margulis measure, which is crucial to prove Theorem \ref{margulis}.

\begin{theorem}\label{equi}
Let $(M,g)$ be a closed $C^\infty$ uniform visibility manifold without conjugate points and with continuous Green bundles. Suppose that the geodesic flow has a hyperbolic periodic point, and $\e\in (0, \inj (M)/2)$ is fixed where $\inj (M)$ is the injectivity radius of $M$.
For $t > 0$, let $C(t)$ be any maximal set of pairwise non-free-homotopic
closed geodesics with lengths in $(t-\e, t]$, and define the measure
$$\nu_t:=\frac{1}{\#C(t)}\sum_{c\in C(t)}\frac{Leb_c}{t}$$
where $Leb_c$ is the Lebesgue measure along the curve $\dot c$ in the unit
tangent bundle $SM$.

Then the measures $\nu_t$ converge in the
weak$^*$ topology to the unique MME as $t\to \infty$.
\end{theorem}

We follow Margulis' approach and combine ideas in \cite{Ri, CKW2}. For example, we construct a flow box in a geometric way which helps us to apply the local product structure of the Bowen-Margulis measure. We use mixing properties of the MME to count the number of self-intersections of the box under the geodesic flow.  To produce closed geodesics from such self-intersections, a type of closing lemma is of particular importance, which is established by the uniform visibility property. Moreover in \cite{CKW2}, by Morse Lemma the action induced by a deck transformation on the boundary at infinity is topologically conjugate to the analogous one under a metric of negative curvature (\cite[Remark 2.6]{CKW2}). In our setting, this property is absent. We assume that the visibility manifold $M$ has continuous Green bundles, so that we can choose the flow box inside $\R_1$ to avoid generalized strips. See Remarks \ref{cyclic} and \ref{axisissue}.

\subsubsection{Volume estimates}
A twin problem is the asymptotics of volume growth of Riemannian balls in the universal cover. Margulis in his thesis also obtained the following result: Consider a smooth closed manifold $M$ of negative curvature with universal cover $X$, then
\begin{equation}\label{e:mar2}
b_t(x)/\frac{e^{ht}}{h}\sim c(x),
\end{equation}
where $b_t(x)$ is the Riemannian volume of the ball of radius $t>0$ around $x\in X$, $h$ the topological entropy of the geodesic flow, and $c: X\to \RR$ is a continuous function, which is called \emph{Margulis function}. See \cite{Mar1, Mar2}.

For a closed rank one manifold of nonpositive curvature, Knieper \cite{Kn1} used his measure to obtain the following asymptotic estimates: there exists $C>0$ such that
$$\frac{1}{C}\le b_t(x)/e^{ht}\le C$$
for any $x\in X$.
However, it is difficult to improve the above to the Margulis-type asymptotic estimates \eqref{e:mar2}, see the remark after \cite[Chapter 5, Theorem 3.1]{Kn3} and an unpublished preprint \cite{Gun1}. Recently, a breakthrough was made by Link (\cite[Theorem C]{Link2}), where an asymptotic estimate for the orbit counting function is obtained for a CAT$(0)$ space, and as a consequence \eqref{e:mar2} is established for rank one manifolds of nonpositve curvature.

Beyond nonpositive curvature, the author \cite{Wu1} proved \eqref{e:mar2} for rank one manifolds without focal points, and briefly discussed the case for manifolds without conjugate points in class $\mathcal{H}$. Recently, Pollicott and War \cite{PW} proved \eqref{e:mar2} for manifolds without conjugate points in class $\mathcal{H}$.

In this paper, we establish volume asymptotics \eqref{e:mar2} for uniform visibility manifolds without conjugate points and with continuous Green bundles:
\begin{thmm}\label{margulis2}
Let $(M,g)$ be a closed $C^\infty$ uniform visibility manifold without conjugate points and with continuous Green bundles. Suppose that the geodesic flow has a hyperbolic periodic point. Then
$$b_t(x)/\frac{e^{ht}}{h} \sim c(x),$$
where $b_t(x)$ is the Riemannian volume of the ball of radius $t>0$ around $x\in X$, $h$ the topological entropy of the geodesic flow, and $c: X\to \RR$ is a continuous function.
\end{thmm}

Our proof of Theorem \ref{margulis2} also uses Margulis' approach. We construct geometric flow boxes and apply a type of $\pi$-convergence theorem, i.e., Lemma \ref{piconvergence}. First we will establish the asymptotics formula for a pair of flow boxes using the mixing property of the unique MME and scaling property of Patterson-Sullivan measure. The asymptotics in \eqref{e:mar2} involves countably many pairs of flow boxes and an issue of nonuniformity arises. To overcome this difficulty, we will use the fact that the unique MME is supported on $\R_1^c$ and apply Knieper's techniques (Lemmas \ref{singular} and \ref{nonuniform}). The assumption of continuous Green bundles is crucial here, as well as in the proof of Lemma \ref{null}. 

Let us compare the proofs of Theorem \ref{margulis} and Theorem \ref{margulis2}. In the proof of Theorem \ref{margulis}, due to the equidistribution of closed geodesics, we only need consider one flow box and count the number of self-intersections of the box under the geodesic flow. But we need to deal with the axis of deck transformations, in order to count effectively the homotopy classes of closed geodesics. On the other hand, in the proof of Theorem \ref{margulis2}, we only need to count intersections of flow boxes, and in fact even the uniqueness of MME is not directly used. The technical novelties in the proof of Theorem \ref{margulis2} is that we have to deal with countably many pairs of flow boxes and the essential difficulty caused by nonuniform hyperbolicity. Even for one pair of flow boxes, the calculations involved are more complicated. Moreover we also consider the intersection components of a flow box with a face under the geodesic flow, which makes the calculation more subtle.

\subsection{Margulis function and rigidity}
Let $M=X/\C$ be a closed $C^\infty$ uniform visibility manifold without conjugate points and with continuous Green bundles, and $X$ its universal cover. Suppose that the geodesic flow has a hyperbolic periodic point. The continuous function $c(x)$ obtained in Theorem \ref{margulis2} is called \emph{Margulis function}. It is $\Gamma$-invariant and hence descends to a function on $M$, which is still denoted by $c$.

As Margulis function is closely related to Patterson-Sullivan measure, its interactions with Bowen-Margulis measure, Lebesgue measure and harmonic measure usually give various characterizations of locally symmetric spaces. In this paper, we study rigidity results related to Margulis function for manifolds without conjugate points, which extend results for manifolds of negative curvature \cite{Yue} and without focal points \cite{Wu1}.

In negative curvature, Katok conjectured that $c(x)$ is almost always not constant and not smooth, unless $M$ is locally symmetric. In \cite{Yue} Yue answered Katok's conjecture, and his results can be extended to certain manifolds without conjuagte points.
\begin{thmm}\label{marfunction}
Let $M$ be a closed $C^\infty$ uniform visibility manifold without conjugate points and with continuous Green bundle, $X$ the universal cover of $M$. Suppose that the geodesic flow has a hyperbolic periodic point. Then
\begin{enumerate}
  \item The Margulis function $c$ is a $C^1$ function.
  \item If $c(x)\equiv C$, then for any $x\in X$,
  $$h=\int_{\partial X} tr U(x,\xi)d\tilde \mu_x(\xi)$$
where $U(x,\xi)$ and $tr U(x,\xi)$ are the second fundamental form and the mean curvature of the horosphere $H_x(\xi)$ respectively, and $\tilde \mu_x$ is the normalized Patterson-Sullivan measure.
\end{enumerate}
\end{thmm}

We note that the assumption of continuous asymptote in Theorem \ref{marfunction} guarantees that the Busemann function is $C^2$. When $c(x)\equiv C$, we expect stronger rigidity results to hold. To achieve this goal, Yue \cite{Yue} studied the uniqueness of harmonic measures associated to the strong stable foliation of the geodesic flow in negative curvature. However, as far as we know, the uniqueness of harmonic measures in nonpositive curvature is unknown in higher dimension.

Nevertheless, if $\dim M=2$, then $\text{Vol}(B^s(x,r))=2r$ where $B^s(x,r)$ is any ball of radius $r$ in a strong stable manifold. In this case, $B^s(x,r)$ is just a curve. Hence the leaves of the strong stable foliation have polynomial growth. Recently, Clotet \cite{Cl2} proved the unique ergodicity of the horocycle flow on surfaces  of genus at least $2$, without conjugate points and with continuous Green bundle. Combining with this result, we can show that there is a unique harmonic measure associated to the strong stable foliation. Then we obtain the following rigidity result by applying Katok's entropy rigidity result \cite[Theorem B]{Ka}. In dimension two, the fact that the geodesic flow has a hyperbolic periodic point is automatic, see the discussion at the beginning of Section $4$.
\begin{thmm}\label{rigidity}
Let $M$ be a closed Riemannian surface of genus at least $2$ without conjugate points and with continuous Green bundles. Then $c(x)\equiv C$ if and only if $M$ has constant negative curvature.
\end{thmm}

Without the uniqueness of harmonic measures, we can still obtain some rigidity results in arbitrary dimensions. The flip map $F:SM\to SM$ is defined as $F(v):=-v$. For general uniform visibility manifolds, the Bowen-Margulis measure $m$ is constructed by using measurable choice theorem, hence is not necessarily flip invariant. However, for uniform visibility manifolds with continuous Green bundles and the geodesic flow admitting a hyperbolic periodic point, the Bowen-Margulis measure $m$ is supported on $\R_1$ and hence it is indeed flip invariant. In this case, consider the conditional measures $\{\bar\mu_x\}_{x\in M}$ of $m$ with respect to the partition $SM=\cup_{x\in M}S_xM$. $\bar\mu_x$ can be identified as measures on $\partial X$, and it would be natural to consider if $\bar\mu_x$ and the normalized Patterson-Sullivan measures $\tilde\mu_x$ coincide.
Yue \cite{Yue1, Yue2} obtained related rigidity results in negative curvature, which can be generalized as follows.
\begin{thmm}\label{flip}
Let $M$ be a closed $C^\infty$ uniform visibility manifold without conjugate points and with continuous Green bundle. Suppose that the geodesic flows admits a hyperbolic periodic point. Then the conditional measures $\{\bar\mu_x\}_{x\in M}$ of the Bowen-Margulis measure coincide almost everywhere with the normalized Patterson-Sullivan measures $\tilde\mu_x$ if and only if $M$ is locally symmetric.
\end{thmm}
The assumption of continuous asymptote above again guarantees that the Busemann function is $C^2$. 

Both Theorem \ref{rigidity} and Theorem \ref{flip} give characterizations of rank one locally symmetric spaces among certain closed manifolds without conjugagte points. In fact, it is a well known conjecture \cite[p.179]{Yue} that in negative curvature $c(x)\equiv C, x\in M$ if and only if $M$ is locally symmetric. The conjecture is true for surfaces of negative curvature by \cite[Theorem 4.3]{Yue}, and for a more general class of surfaces without conjugate points by Theorem \ref{rigidity}.

\subsection{The Hopf-Tsuji-Sullivan dichotomy}
We also establish the Hopf-Tsuji-Sullivan dichotomy with respect to Bowen-Margulis measure for the geodesic flow on certain manifolds without conjugate points. At first, let us recall some necessary terminologies (see \cite{LP, Link1}).

Let $\Omega$ be a locally compact and $\sigma$-compact Hausdorff topological space and $\{\phi^t\}_{t\in \RR}$ a flow on $\Omega$. A point $\omega\in \Omega$ is said to be \textit{positively/negatively recurrent} if there exists a sequence $t_n\to \infty$ of real numbers such that
$\phi^{\pm t_n}\omega\to \omega.$
On the other hand, $\omega\in \Omega$ is said to be \textit{positively/negatively divergent} if for every compact set $K\subset \Omega$ there exists $T>0$ such that $\phi^{\pm t}\omega\notin K$ for all $t\ge T$.

Assume that $\mu$ is a Borel measure on $\Omega$ invariant by the flow $\{\phi^t\}_{t\in \RR}$. Then the Hopf decomposition theorem (see for instance \cite[Satz 13.1]{Ho00}) says that the space $\Omega$ decomposes into a disjoint union of $\phi^t$-invariant Borel sets $\Omega_C$ and $\Omega_D$ which satisfy the following properties:
\begin{enumerate}
    \item  There does not exist a Borel subset $E\subset \Omega_C$ with $\mu(E)>0$ and such that the sets $\{\phi^kE\}_{k\in \ZZ}$ are pairwise disjoint.
\item There exists a Borel set $W\subset \Omega_D$ such that $\Omega_D$ is the disjoint union of sets $\{W_k\}_{k\in \ZZ}$, where each $W_k$ is a translate of $W$ under the flow $\phi$.
\end{enumerate}

According to Poincar\'{e}'s recurrence theorem (see for example \cite[Satz 13.2]{Ho00}) $\mu$-almost every point of $\Omega_C$ is positively and negatively recurrent. On the other hand, by Hopf's divergence theorem (see also \cite[Satz 13.2]{Ho00}), $\mu$-almost every point of $\Omega_D$ is positively and negatively divergent. This implies in particular that the sets $\Omega_C$ and $\Omega_D$ are unique up to sets of measure zero.

The dynamical system $(\Omega, \{\phi^t\}_{t\in \RR}, \mu)$ is said to be \textit{completely conservative} if $\mu(\Omega_D)=0$, and \textit{completely dissipative} if $\mu(\Omega_C)=0$. 

Finally, a dynamical system $(\Omega, \{\phi^t\}_{t\in \RR}, \mu)$ is called \textit{ergodic} if every $\phi^t$-invariant Borel set $E\subset \Omega$ satisfies either $\mu(E) = 0$ or $\mu(\Omega\setminus E) = 0$. Hence if $(\Omega, \{\phi^t\}_{t\in \RR}, \mu)$ is ergodic, then it is either completely conservative or completely dissipative. The latter can only occur for an infinite measure $\mu$ which is supported on a single orbit $\{\phi^t\omega: t\in \RR\}$ with $\omega\in \Omega$.

In the 1930's, Hopf first observed that the geodesic flow on surfaces with constant negative curvature is either completely conservative or completely dissipative with respect to the Lebesgue measure, which in this case coincides with the  Bowen-Margulis measure. Then Tsuji \cite{Tsu} proved that conservativity is equivalent to ergodicity. Later Sullivan \cite{Sul} obtained these results for manifolds of constant negative curvature in any dimension. Since then, the study of the Hopf-Tsuji-Sullivan dichotomy has been an active topic in the study of the ergodic theory of geodesic flows. Kaimanovich \cite{Kai1} proved the dichotomy for Gromov hyperbolic metric spaces. On manifolds with pinched negative curvature, Yue \cite{Yue0} showed the dichotomy for Bowen-Margulis measure, and Paulin- Pollicott-Schapira \cite{PPS} proved it for equilibrium states. Roblin \cite{Rob} also proved it for the CAT($-1$) metric spaces. 

In nonpositive curvature, due to existence of flat strips, in general the Hopf parametrization is not one-to-one which causes essential difficulties. With very advanced technique, Link and Picaud \cite{LiP} proved the dichotomy with respect to Knieper's measure for rank one manifolds of nonpositive curvature. Link \cite{Link1} proved the dichotomy with respect to Ricks' measure for CAT($0$) spaces. Recently, Liu-Liu-Wang \cite{LLW} proved the dichotomy with respect to Bowen-Margulis measure for uniform visibility manifolds satisfying Axiom $2$, on which every pair $\xi \neq \eta\in \pX$ can be connected by exactly one geodesic. 

We will establish the Hopf-Tsuji-Sullivan dichotomy with respect to a Bowen-Margulis measure $m_\C$ for the geodesic flow on uniform visibility manifolds without conjugate points. Our Bowen-Margulis measure is constructed as a product of Patterson-Sullivan measures and a measure $\lambda_{\xi,\eta}$ supported on a generalized strip $(\xi\eta)$, where a measurable choice theorem is used. Unfortunately, for noncompact manifolds, argument of Krylov-Bogolyubov type cannot be applied, and such Bowen-Margulis measure $m_\C$ is not necessarily invariant under the geodesic flow. Nevertheless, the projection of $m_\C$ to the quotient space is invariant under the quotient geodesic flow. So strictly speaking, we are talking about complete dissipativity/conservativity and ergodicity for the quotient geodesic flow $([SM], \{[\phi^t]\}_{t\in \RR}, [m_\C])$. Since there is no confusion, we still state our results for the geodesic flow $(SM, \{\phi^t\}_{t\in \RR}, m_\C)$. 
\begin{thmm}\label{HTS}
Let $X$ be a simply connected smooth uniform visibility manifold without conjugate points. Suppose that $\C\subseteq Is(X)$ is a non-elementary discrete group which contains an expansive isometry. Let $m_\C$ be a Bowen-Margulis measure on $SM = SX/\C$ and $\mu_o$ is the Patterson-Sullivan measure normalized so that $\mu_o(\pX)=1$. Then exactly one of the following two complementary cases
holds, and the statements (i) to (iii) are equivalent in each case. If furthermore $X$ has bounded asymptote and  $\C$ contains a regular isometry, then the statements (i) to (iv) are equivalent in each case.
\begin{enumerate}
    \item Case:
    \begin{enumerate}
        \item[(i)]  $\sum_{\c\in \C}e^{-\d_\C d(o, \c o)}$ converges.
        \item[(ii)] $\mu_o(L_\C^{\text{rad}})=0$.
        \item[(iii)] $(SM, (\phi^t)_{t\in \RR}, m_\C)$ is completely dissipative.
         \item[(iv)] $(SM, (\phi^t)_{t\in \RR}, m_\C)$ is not ergodic.
    \end{enumerate}
    \item Case:
    \begin{enumerate}
        \item[(i)]  $\sum_{\c\in \C}e^{-\d_\C d(o, \c o)}$ diverges.
        \item[(ii)] $\mu_o(L_\C^{\text{rad}})=1$.
        \item[(iii)] $(SM, (\phi^t)_{t\in \RR}, m_\C)$ is completely conservative.
         \item[(iv)] $(SM, (\phi^t)_{t\in \RR}, m_\C)$ is ergodic.
    \end{enumerate}
\end{enumerate}
\end{thmm}

The most technical part of the proof of Theorem \ref{HTS} is Proposition \ref{divergent} (i.e., (ii)$\Rightarrow$(i) in Case (1)), particularly Lemma \ref{divergent2}. Using ideas from \cite{LiP}, we obtain the lower bound in Lemma \ref{divergent2} by considering a small neighborhood of an expansive vector, which in nonpositive curvature is Ballmann's lemma \cite[Lemma III.3.1]{BB}. 
The assumption of bounded asymptote is used to get Lemma \ref{close}. This is a key step in applying Hopf argument to show that conservativity implies ergodicity.

\subsection{Organization of the paper}
In Section $3$, we recall geometric properties of uniform visibility manifolds without conjugate points, including Green bundles, boundary at infinity, Busemann functions, generalized strips, etc. 

In Section $4$, we consider ergodicity of Liouville measure for certain surfaces without conjugate points, mostly focusing on the properties of singular geodesics. This proves Theorem \ref{singulargeodesic}. 

Section $5$ are devoted to the construction of Patterson-Sullivan measure and Bowen-Margulis measure. We show that Bowen-Margulis measure is exactly the unique MME, and it has the Bernoulli property which proves Theorem \ref{bernoulli}. We prove Theorem \ref{uniqueness}, the uniqueness of equilibrium states in Section $6$. 

In Section $7$, we prove Theorem \ref{margulis} for the asymptotic formula for the number of homotopy classes of closed geodesics.  In Section $8$, we prove Theorem \ref{margulis2} for the formula of volume asymptotics. We then prove rigidity results Theorems \ref{marfunction}, \ref{rigidity} and \ref{flip} related to Margulis function in Section $9$. 

In the last Section $10$, we prove the Hopf-Tsuji-Sullivan dichotomy Theorem \ref{HTS}.

\section{Geometric and ergodic preliminaries}
Let $M$ be an $n$-dimensional $C^\infty$ uniform visibility manifold without conjugate points. Until in last section, $M$ is assumed to be compact. We prepare some geometric and ergodic tools which will be used in subsequent sections.

\subsection{Jacobi fields, Green bundles and Pesin set}
In order to study the dynamics of geodesic flows, we should investigate the geometry of the second tangent bundle $TTM$. Let $\pi:TM\rightarrow M$ be the natural projection, i.e., $\pi(v)=p$ where $v \in T_{p}M$.
The connection map $K_{v}:T_{v}TM\rightarrow T_{\pi(v)}M$ is defined as follows. For any $\xi\in T_vTM$, $K_{v}\xi:= (\nabla X)(t)|_{t=0}$,
where $X:(-\epsilon,\epsilon)\rightarrow TM$ is a smooth curve satisfying $X(0)=v$ and $X'(0)=\xi$, and $\nabla$ is the covariant derivative along the curve $\pi(X(t))\subset M$. Then the standard Sasaki metric on $TTM$ is given by
$$\langle\xi,\eta\rangle_{v}=\langle d\pi_{v}\xi,d\pi_{v}\eta\rangle+\langle K_{v}\xi,K_{v}\eta\rangle, \quad \xi, \eta\in T_vTM.$$

Recall that the \emph{Jacobi equation} along a geodesic $c_v(t)$ is
\begin{equation}\label{e:jacobi0}
J''(t)+R(\dot c_v(t),J(t))\dot c_v(t)=0,
\end{equation}
where $R$ is the curvature tensor, and $J(t)$ is a Jacobi field along $c_v(t)$ and perpendicular to $\dot c_v(t)$.
Suppose that $J_{\xi}(t)$ is the solution of \eqref{e:jacobi0} which satisfies the initial conditions
\[J_{\xi}(0)=d\pi_{v}\xi, ~~\frac{d}{dt}\Big|_{t=0}J_{\xi}(t)=K_{v}\xi.\]
Then it follows that (cf. p. 386 in \cite{BP})
\[J_{\xi}(t)=d\pi_{\phi^{t}v}d \phi^{t}_{v}\xi, ~~\frac{d}{dt}J_{\xi}(t)=K_{\phi^{t}v}d\phi^{t}_{v}\xi.\]

Using the Fermi coordinates, we can write \eqref{e:jacobi0} in the matrix form
\begin{equation}\label{e:jacobi1}
\frac{d^2}{dt^2}A(t)+K(t)A(t)=0.
\end{equation}
We have the following classical result.

\begin{proposition}[Cf. \cite{Eb}]\label{Jacobi}
Given $s\in \mathbb{R}$, let $A_{s}(t)$ be the unique solution of \eqref{e:jacobi1}
satisfying $A_{s}(0)=Id$ and $A_{s}(s)=0$. Then there exists a limit
\[A^{+}=\lim_{s\rightarrow +\infty}\frac{d}{dt}\Big|_{t=0}A_{s}(t).\]
\end{proposition}

Now we can define the \emph{positive limit solution} $A^+(t)$ as the solution of \eqref{e:jacobi1}
satisfying the initial conditions
\[A^{+}(0)=Id, \quad \frac{d}{dt}\Big|_{t=0}A^{+}(t)=A^{+}.\]
It is easy to see that $A^{+}(t)$ is non-degenerate for all $t\in \mathbb{R}$. Similarly, letting $s \rightarrow -\infty$, one can define the \emph{negative limit solution} $A^{-}(t)$ of \eqref{e:jacobi1}.

For each $v \in SM$, define
\[G^s(v):=\{\xi \in T_{v}SM: \langle\xi,Z(v)\rangle=0 \text{\ and\ } J_{\xi}(t)=A^{+}(t)d\pi_{v}\xi\},\]
\[G^u(v):=\{\xi \in T_{v}SM: \langle\xi,Z(v)\rangle=0 \text{\ and\ } J_{\xi}(t)=A^{-}(t)d\pi_{v}\xi\},\]
where $Z$ is the vector field generated by the geodesic flow.
$G^s/G^u$ are called the \textit{stable/unstable Green bundles} of the geodesic flow respectively.

\begin{proposition}[Cf. Proposition 12.1.1 in \cite{BP}]\label{subspace}
$G^s(v)$ and $G^u(v)$ have the following properties:
\begin{enumerate}
  \item $G^s(v)$ and $G^u(v)$ are $(n-1)$-dimensional subspaces of $T_{v}SM$.
    \item $d\pi_{v}G^s(v)=d\pi_{v}G^u(v)=\{w \in T_{\pi(v)}M: w \text{\ is orthogonal to\ } v\}$.
 \item The subspaces $G^s(v)$ and $G^u(v)$ are invariant under the geodesic flow.
  \item Let $\tau: SM \rightarrow SM$ be the involution defined by $\tau v=-v$, then
\[G^s(-v)=d\tau G^u(v)  \text{\ and\ } G^u(-v)=d\tau G^s(v).\]
  \item If the curvature satisfies $K(p) \geq -a^{2}$ for some $a > 0$,
then $\|K_{v}\xi\| \leq a \|d\pi_{v}\xi\|$ for any $\xi \in G^s(v)$ or $\xi \in G^u(v)$.
  \item If $\xi \in G^s(v)$ or $\xi \in G^u(v)$, then $J_{\xi}(t)\neq 0$ for each $t \in \mathbb{R}$.

  \item $\xi \in G^s(v)$ (respectively, $\xi \in G^u(v)$) if and only if
\[\langle\xi,Z(v)\rangle=0  \text{\ and\ }  \|d\pi_{g^{t}v}dg^{t}_{v}\xi\| \leq c\]
for each $t > 0$ (respectively, $t < 0$) and some $c > 0$.
\end{enumerate}
\end{proposition}

Given $\xi\in T_vSM$, the \emph{Lyapunov exponent} $\chi(v, \xi)$ is defined as
$$\chi(v, \xi):=\lim_{t\to \infty}\frac{1}{t}\log\|d\phi^t\xi\|$$
whenever the limit exists.
Recall the set $\R_1$ defined in \eqref{set1}.
\begin{theorem}\label{regular1}(\cite[Theorem 4.1]{Ru3})
Let $(M, g)$ be a compact manifold without conjugate points such that Green bundles are continuous, and $\mu\in \M_{\phi}(SM)$. If the set $\mathcal{R}_1$ is nonempty, then for $\mu \ae v\in \mathcal{R}_1$, $\xi^s\in G^s(v)$, $\xi^u\in G^u(v)$,
\begin{equation*}
\chi(v, \xi^s)<0 \quad \text{and} \quad \chi(v, \xi^u)>0.
\end{equation*}
\end{theorem}

Recall the Pesin set $\Delta$ which is defined in \eqref{e:Pesin}.
\begin{theorem}\label{regular}
Let $(M,g)$ be a smooth, connected and closed visibility manifold without conjugate points and with continuous Green bundles,  and $\mu\in \M_{\phi}(SM)$. Suppose the geodesic flow has a periodic hyperbolic point then $\Delta$ agrees $\mu$-almost everywhere with the open dense set $\mathcal{R}_1$.
\end{theorem}

\begin{proof}
First note that by \cite[Lemma 7.1]{MR}, $\mathcal{R}_1$ is a nonempty, open and dense set. By  Theorem \ref{regular1}, we see that for $\mu \ae v\in \mathcal{R}_1$, $v\in \Delta$. 

On the other hand, $\mu \ae v\in \Delta$ is Oseledets regular, more precisely, there exist $1\le k=k(v)\le 2n-1$, a $\phi^t$-invariant decomposition $$T_vSM=E^1(v)\oplus \cdots \oplus E^k(v)$$ and numbers $\chi_1(v)<\cdots<\chi_k(v)$ such that $\chi(v,\xi)=\chi_i(v)$ for every $\xi\in E^i(v)\setminus \{0\}$. Denote 
\begin{equation*}
E^s(v):=\oplus_{\chi_i(v)<0}E^i(v), \ E^u(v):=\oplus_{\chi_i(v)>0}E^i(v), \  E^c(v):=\oplus_{\chi_i(v)=0}E^i(v).
\end{equation*}
By Freire-Ma\~{n}\'{e} \cite{FrMa},
\begin{equation*}
E^s(v)\subset G^s(v)\subset E^s(v)\oplus E^c(v), \quad E^u(v)\subset G^u(v)\subset E^u(v)\oplus E^c(v).
\end{equation*}
If $v\in \Delta$, then $E^c(v)=Z(v)$, and hence $E^s(v)= G^s(v)$ and $E^u(v)= G^u(v)$. It follows that $G^s(v)\cap G^u(v)=\{0\}$, that is, $v\in \mathcal{R}_1$.
\end{proof}

\subsection{Boundary at infinity}
Let $M$ be a closed Riemannian manifold without focal points, and $\pr: X\to M$ the universal cover of $M$. Let $\C \simeq \pi_1(M)$ be the group of deck transformations on $X$, so that each $\c\in \C$ acts isometrically on $X$. Let $\F$ be a fundamental domain with respect to $\C$.
Denote by $d$ both the distance functions on $M$ and $X$ induced by Riemannian metrics. The Sasaki metrics on $SM$ and $SX$ are also denoted by $d$ if there is no confusion.

We still denote by $\pr: SX\to SM$ and $\c: SX\to SX$ the map on unit tangent bundles induced by $\pr$ and $\c\in \C$. We use an underline to denote objects in $M$ and $SM$, e.g. for a geodesic $c$ in $X$ and $v\in SX$, $\lc:=\pr c$, $\lv:=\pr v$ denote their projections to $M$ and $SM$ respectively. However, sometimes we omit the underline if there is no confusion caused by notations.

We call two geodesics $c_{1}$ and $c_{2}$ on $X$ \emph{positively asymptotic} or just \emph{asymptotic} if there is a positive number $C > 0$ such that $d(c_{1}(t),c_{2}(t)) \leq C, ~~\forall~ t \geq 0.$
We say $c_{1}$ and $c_{2}$ are \emph{negatively asymptotic} if the above holds for $\forall t \leq 0$. $c_{1}$ and $c_{2}$ are said to be \emph{bi-asymptotic} or \emph{parallel} if they are both positively and negatively asymptotic.
The relation of asymptoticity is an equivalence relation between geodesics on $X$. The class of geodesics that are asymptotic to a given geodesic $c_v/c_{-v}$ is denoted by $c_v(+\infty)$/$c_v(-\infty)$ or $v^+/v^-$ respectively. We call them \emph{points at infinity}. Obviously, $c_{v}(-\infty)=c_{-v}(+\infty)$. We call the set $\pX$ of all points at infinity the \emph{boundary at infinity}. If $v^+=\eta\in \pX$, we say $v$ \emph{points at $\eta.$}

\begin{definition}(Cf. \cite[p. 8]{MR})\label{div}
Let $(X, g)$ be a simply connected manifold without conjugate points. We say that \emph{geodesic rays diverge in $X$} if for every $p\in X$ and every $\e, C>0$ there exists $T(p,\e, C)$ such that for every two geodesic rays $c_1\neq c_2$ with $c_1(0)=c_2(0)=p$ in $X$ and $\angle_p(\dot c_1, \dot c_2)\ge \e$, we have $d(c_1(t), c_2(t))\ge C$ for every $t\ge T(p,\e, C)$. The geodesics \textit{diverge uniformly} if $T(p,\e, C)$ does not depend on $p$.
\end{definition}
The uniform visibility property implies that the geodesic diverge uniformly (\cite{Ru0, Ru2}). If $X$ has no focal points, then geodesic rays diverge.

Let $\overline{X}=X \cup \pX$.
For each point $p \in X$ and $v\in S_{p}X$, each points $x, y \in \overline{X}-\{p\}$, positive numbers $\epsilon$ and $r$,
and subset $A\subset\overline{X}-\{p\}$,
we define the following notations:
\begin{itemize}
\item{} ~~$c_{p,x}$ is the geodesic from $p$ to $x$ and satisfies $c_{p,x}(0)=p$.

\item{} ~~$\measuredangle_{p}(x,y)=\measuredangle(\dot c_{p,x}(0),\dot c_{p,y}(0))$.
\item{} ~~$\measuredangle_{p}(A)=\sup \{\measuredangle_{p}(a,b)\mid a,\ b \in A \}$.
\item{} ~~$\measuredangle(v,x)=\measuredangle(v,\dot c_{p,x}(0))$.
\item{} ~~$C(v,\epsilon)=\{a \in \overline{X}-\{p\}\mid \measuredangle(v,a)< \epsilon\}$.
\item{} ~~$C_{\epsilon}(v) =C(v,\epsilon)\cap \pX= \{c_{w}(+\infty)\mid w \in  S_{x}X, \angle(v,w)<\epsilon\}$.
\item{} ~~$TC(v,\epsilon,r)= \{q \in \overline{X}\mid \measuredangle_{p}(c_{v}(+\infty),q)< \epsilon\} - \{q \in \overline{X}\mid d(p,q)\leq r\}$.
\end{itemize}

$TC(v,\epsilon,r)$ is called the \emph{truncated cone} with axis $v$ and angle $\epsilon$. Obviously $c_{v}(+\infty) \in TC(v,\epsilon,r)$. There is a unique topology $\tau$ on $\overline{X}$ such that for each $\xi \in \pX$ the set of truncated cones containing $\xi$ forms a local basis for $\tau$ at $\xi$. This topology is usually called the \emph{cone topology}. 

We can also define the visual topology on $\pX$. For each $p\in X$, there is a function $f_p: S_pX\to \pX$ defined by $$f_p(v)=v^+, v\in S_pX.$$
Since the geodesics diverge in $X$, $f_p$ is a bijection. So for each $p\in M$, $f_p$ induces a topology on $\pX$ from the usual topology on $S_pX$.
The topology on $\pX$ induced by $f_p$ is independent of $p\in X$, and is called the \emph{visual topology} on $\pX$.
Visual topology on $\pX$ and the manifold topology on $X$ can be extended naturally to the cone topology on $\overline X$.
Under the cone topology, $\overline{X}$ is homeomorphic to the closed unit ball in $\mathbb{R}^n$, and $\pX$ is homeomorphic to the unit sphere $\mathbb{S}^{n-1}$. For more details about the cone topology, see  \cite{Eb1} and \cite{EO}.

We say that $M$ is\textit{ quasi-convex} if there exist constants $A, B>0$ such that for every two geodesic segments $c_1: [t_1,t_2]\to X$ and $c_2: [s_1,s_2]\to X$ it holds that
\[d_H(c_1, c_2) \le A \sup\{d(c(t_1), c(s_1)), d(c_2(t_2), c_2(s_2))\}+B,\]
where $d_H$ is the Hausdorff distance.

Let us collect some good geometric properties of uniform visibility manifolds, see \cite[Proposition 2.7]{LLW} and \cite[Teorem 2.4]{MR}.
\begin{proposition}\label{geo}
Let $M$ be a closed uniform visibility manifold with no conjugate points, and $X$ its universal cover.
\begin{enumerate}
\item (\cite{Ru0, Ru2}) Geodesic rays diverge uniformly in $X$.
\item (\cite{Eb1}) The following map is continuous.
 $$\Psi: SX \times [-\infty,+\infty] \rightarrow \overline{X},\quad (v,t) \mapsto c_{v}(t).$$
\item (\cite{EO}) For any two points $\xi\neq \eta\in \pX$, there is at least one connecting geodesic of them.
\item (\cite{Eb1}) $X$ is quasi-convex.
\item (\cite{Eb1}) There exists a uniform constant $Q>0$ (called \textit{Morse constant}) such that for every two bi-asymptotic geodesic $c_1, c_2$ in $X$ we have $d_H(c_1, c_2)\le Q.$
\end{enumerate}
\end{proposition}

\subsection{Busemann functions and horospheres}
For each pair of points $(p,q)\in X \times X$ and each point at infinity $\xi \in \pX$, the \emph{Busemann function based at $\xi$ and normalized by $p$} is
$$b_{\xi}(q,p):=\lim_{t\rightarrow +\infty}\big(d(q,c_{p,\xi}(t))-t\big),$$
where $c_{p,\xi}$ is the unique geodesic from $p$ and pointing at $\xi$.
The Busemann function $b_{\xi}(q,p)$ is well-defined since the function $t \mapsto d(q,c_{p,\xi}(t))-t$
is bounded from above by $d(p,q)$, and decreasing in $t$ (this can be checked by using the triangle inequality). 

If $v\in S_pX$ points at $\xi\in \pX$, we also write $b_v(q):=b_{\xi}(q,p).$ The level sets of the Busemann function $b_{\xi}(q,p)$ are called the \emph{horospheres} centered at $\xi$. The horosphere through $p$ based at $\xi\in \pX$, is denoted by $H_p(\xi)$ or $H(v)$. Define the sets
\[\F^s(v):=\{(q, -\nabla_qb_v): q\in H(v)\}, \quad \F^s(v):=\{(q, \nabla_qb_{-v}): q\in H(-v)\}.\]
$\F^s/\F^u$ form foliations of $SX$, called \textit{stable/unstable horospherical foliations}.
For more details of the Busemann functions and horospheres, please see ~\cite{DPS,Ru1,Ru2}.
Here we are concerned with the continuity property of the horospheres and Busemann functions.

\begin{proposition}\label{horofoliation}(Cf. for example \cite[Proposition 2.1, Theorem 2.1 and Theorem 1.3]{MR})
Let $M$ be a compact manifold without conjugate points with universal cover $X$. Then for every $v\in SX$ 
\begin{enumerate}
    \item Busemann functions $b_v$ are $C^{1,L}$ with $L$-Lipschitz unitary gradient for a uniform constant $L>0$; 
    \item horospheres $H(v), H(-v)\subset X$ are $C^{1,L}$ embedded $(n-1)$-dimensional submanifolds, and $\F^s(v), \F^u(v)\subset SX$ are Lipschitz continuous embedded $(n-1)$-dimensional submanifolds.
\end{enumerate} 
If assume that $M$ is also a uniform visibility manifold, then
\begin{enumerate}
    \item  the families $\F^s/\F^u$ are continuous minimal foliations of $SM$ invariant by the geodesic flow, i.e. for every $t\in \RR$, $$\phi^t(\F^{s/u}(v))=\F^{s/u}(\phi^t(v));$$
    \item the geodesic flow is topologically mixing;
    \item there exist $A,B>0$ such that for every $v\in SX$ and every $w\in \F^s(v)$,
    \[d_s(\phi^t(v), \phi^t(w))\le Ad_s(v, w)+B \quad \text{for every\ }t\ge 0\]
    where $d_s$ is the induced metric on $\F^s(v)$.
\end{enumerate}  
If assume furthermore that $M$ is a visibility manifold with continuous Green bundles and the geodesic flow has a periodic hyperbolic point, then
\begin{enumerate}
    \item hyperbolic periodic points are dense on $SM$;
    \item Green bundles $G^s/G^u$ are uniquely integrable and tangent to the smooth horospherical foliations $\F^s/\F^u$ respectively.
\end{enumerate}  
\end{proposition}

With assumption of continuous asymptote, we have better regularity of the Busemann functions. That is why we have this assumption in Theorems \ref{marfunction} and \ref{flip}.

\begin{lemma}\label{Buseregu}
Let $M$ be a closed manifold without conjugate points and with continuous Green bundle. Then the Busemann function $q\mapsto b_v(q)$ is $C^2$. Moreover, $\triangle_yb_{\xi}(y,x)=tr U(y,\xi)$ where $U(y,\xi)$ and $tr U(y,\xi)$ are the second fundamental form and the mean curvature of the horosphere $H_y(\xi)$.
\end{lemma}
\begin{proof}
By Proposition \ref{horofoliation}, $y\mapsto b_{\xi}(y,x)$ is $C^{1,L}$ by uniform visibility, and $\nabla_yb_{\xi}(y,x)=-X$ where $X$ is the geodesic spray. Since Green bundles $G^s/G^u$ are uniquely integrable and tangent to the smooth horospherical foliations $\F^s/\F^u$ by Proposition \ref{horofoliation}, we see that $\triangle_yb_{\xi}(y,x)=\text{div} (-X)=tr U(y,\xi)$ and it is continuous in $\xi$.
\end{proof}
\begin{remark}
If $X$ has bounded asymptote, then $y\mapsto b_{\xi}(y,x)$ is $C^2$ \cite[Theorem 1 and Proposition 5]{Es}, and moreover both $\nabla_yb_{\xi}(y,x)=-X$ and $\triangle_yb_{\xi}(y,x)=tr U(y,\xi)$ depends continuously on $\xi$ (see \cite[Proposition 4]{EsO} and the discussion at the end of Subsection 2.3 in \cite{LS}).  
\end{remark}

The notion of continuity of $v\mapsto H(v)$ is defined as the following: let $v_n$ converge to $v$, then $ H(v_n)$ converges to $H(v)$ uniformly on compact subsets of $X$. 
\begin{lemma}(\cite[Theorem 3.7]{Ru3})
Let $(M,g)$ be a compact Riemannian manifold without conjugate points. Then geodesic rays diverge uniformly in $X$ if and only if the collection of submanifolds $v\mapsto H(v)$, depends continuously on $v$ in the compact open topology.
\end{lemma}
Recall that by Proposition \ref{geo}(1), uniform visibility implies that geodesic rays diverge uniformly. Then we have
\begin{proposition}(Cf. \cite[Theorem 6.1]{Pe2} \cite[Lemma 1.2]{Ru1})\label{horosphere}
Let $M$ be a closed uniform visibility manifold without conjugate points. Then for every $p\in X$, the map $\xi\mapsto H_\xi(p)$ is continuous in the following sense: if $\xi_n\to \xi \in \pX$ and $K\subset X$ is compact, then $H_{\xi_n}(p)\cap K\to H_\xi(p)\cap K$ uniformly in the Hausdorff topology.
\end{proposition}


The continuity of $v\mapsto H(v)$ is equivalent to the continuity in the $C^1$ topology of the map $v\mapsto b_v$ uniformly on compact subsets of $X$. So we have
\begin{corollary}\label{continuous}
Let $M$ be a closed uniform visibility manifold without conjugate points. Then the functions $(v,q)\mapsto b_v(q)$ and $(\xi,p,q)\mapsto b_\xi(p,q)$ are continuous on $SX\times X$ and $\pX\times X\times X$ respectively.
\end{corollary}

We would like to obtain certain equicontinuity of Busemann function $v\mapsto b_v(q)$. The idea is to extend the function to the compact space $\overline{X}$.
Given $x,p,q\in X$, define
$$b_x(q,p):=d(q,x)-d(p,x).$$

\begin{lemma}\label{continuity1}(\cite[Corollary 2.18]{CKW1})
Let $M$ be a closed uniform visibility manifold without conjugate points. For each pair of points $p,q\in X$, if there is a sequence $\{x_{n}\}\subset X$ with $\lim_{n\rightarrow +\infty}x_{n}=\xi\in \pX$, then $$\lim_{n\rightarrow +\infty}b_{x_n}(q,p)=b_{\xi}(q,p).$$
\end{lemma}

\begin{lemma}\label{equicon1}
Let $M$ be a closed uniform visibility manifold without conjugate points. Let $p\in X$, $A\subset S_pX$ be closed, and $B\subset X$ be such
that $A^+:= \{v^+: v\in A\}$ and $$B^\infty := \{\lim_n q_n \in \pX: q_n\in B\}$$ are
disjoint subsets of $\pX$. Then the family of functions $A\to \RR$ indexed
by $B$ and given by $v\mapsto b_v(q)$ are equicontinuous: for every $\e>0$ there
exists $\d >0$ such that if $\angle_p(v,w)<\d$, then $|b_v(q)-b_w(q)|<\e$ for every $q\in B$.
\end{lemma}
\begin{proof}
The proof is a repetition of that of \cite[Lemma A.1]{CKW2} and thus omitted. Notice that Lemma \ref{continuity1} is essentially used in the proof.
\end{proof}

If $X$ has uniform visibility, then it is quasi-convex and geodesic rays diverge by Proposition \ref{geo}. The following generalized strip lemma is a direct corollary of \cite[Lemma 3.1]{RiR}.
\begin{lemma}\label{striplemma}(Generalized strip lemma, \cite[Lemma 3.1]{RiR})
Let $M$ be a closed uniform visibility manifold without conjugate points and $X$ be its universal covering. Then for every $v,w \in SX$ with bi-asymptotic geodesics $c_v$ and $c_w$, there exists a connected set $\Sigma\subset I(v)$ containing the geodesic starting points such that any geodesic $c$ with $c(0) = q\in \Sigma$ and $\dot c(0)=-\nabla b_v(q)$ is bi-asymptotic to both $c_v$ and $c_w$. In particular, the set
\begin{equation}\label{strip}
\mathcal S:=\bigcup_{t\in \RR, q\in \Sigma}c_{\nabla b_v(q)}(t)       
\end{equation}
is homeomorphic to $\Sigma\times \RR$.   
\end{lemma}

Recall that $\mathcal{I}(v)=\mathcal{F}^s(v)\cap\mathcal{F}^u(v)$. Denote $I(v):=H(v)\cap H(-v)$.
\begin{lemma}\label{compact}(\cite[Corollary 2.1]{MR})
Let $M$ be a compact manifold without conjugate points with visibility universal covering $X$ and $Q>0$ Morse constant from Proposition \ref{geo}. Then there exists an uniform constant $Q'\le Q$ such that for every $v\in SX$ the sets $\mathcal{I}(v), I(v)$ are compact connected sets with $\text{Diam}(\mathcal{I}(v))\le Q$ and $\text{Diam}(I(v))\le Q'$.
\end{lemma}
We can define the width of a generalized strip $\mathcal S$ in \eqref{strip} as $\inf_{w\in \mathcal S}\text{Diam}(\mathcal{I}(w)).$
Then a nontrivial generalized strip might have width zero.

\section{Ergodicity of Liouville measure}
In this section, we always let $M$ be a closed surface with genus $\mathfrak{g} \geq 2$, without conjugate points and with bounded asymptote, and $X$ its universal cover. 

By Katok's result \cite[Theorem B]{Ka}, we have
$$h^2\ge -2\pi E /\text{Vol}(M)$$
where $h=h_{\top}(\phi^1)$ is the topological entropy of the geodesic flow and $E$ is the Euler characteristic of $M$. The equality holds if and only if $M$ has constant negative curvature. So in our setting, $h>0$. By the variational principle Theorem \ref{VPentropy}, there exists a $\mu\in \M_{\phi}(SM)$ such that $h_\mu(\phi^1)>0$. Thus $\mu$ is a hyperbolic measure, and there exists a horseshoe on $SM$. We know that there exists a hyperbolic periodic point under the geodesic flow (see \cite[Theorem 4.1]{Ka}, which is a flow version of \cite[Corollary 4.4]{Ka1}). 

So $\R_1\neq \emptyset$. By Proposition \ref{horofoliation}, Green bundles $G^{s/u}$ are tangent to the horospherical foliations $\F^{s/u}$ in our setting.
\subsection{Singular set and Expansivity}

On the surface $M$, we have the Fermi coordinates $\{e_1(t), e_2(t)\}$ along a geodesic $c_v(t)$, obtained by the time $t$-parallel translations along $c_v(t)$ of an orthonormal basis $\{e_1(0), e_2(0)\}$ where $e_1(0)=\dot c_v(0)$. Thus $e_1(t)=\dot c_v(t)$ and $e_2(t) \perp \dot c_v(t)$. Suppose that $J(t)=j(t)e_2(t).$ Then the Jacobi equation \eqref{e:jacobi0} becomes
\begin{equation}\label{e:jacobi}
j''(t)+K(t)j(t)=0,
\end{equation}
where $K(t)=K(c_v(t))$ is the curvature at point $c_v(t)$.
Let $u(t)=j'(t)/j(t)$. Then the Jacobi equation \eqref{e:jacobi} can be written equivalently as:
\begin{equation}\label{e:ricatti}
u'(t)+u^2(t)+K(t)=0,
\end{equation}
which is called \emph{the Riccati equation}.

For a given $\xi\in T_vSM$, we always let $J_{\xi}(t)$ be the unique Jacobi field satisfying the Jacobi equation \eqref{e:jacobi0} under initial conditions
\[J_{\xi}(0)=d\pi_{v}\xi,\ ~~\frac{d}{dt}\Big|_{t=0}J_{\xi}(t)=K_{v}\xi.\]
Suppose that $J_{\xi}(t)$ is perpendicular to $\dot c_v(t)$, then $J_\xi(t)=j_\xi(t)e_2(t)$ and $j_\xi(t)=\|J_\xi(t)\|$. Denote $u_\xi(t)=j'_\xi(t)/j_\xi(t)$. Then $u_\xi$ is a solution of the Riccati equation \eqref{e:ricatti}.

The following proposition shows the connection between the Lyapunov exponent $\chi(v, \xi)$ and the function $u_{\xi}$.
\begin{proposition}\label{lyapunov}
For any $v\in SM$ and $\xi\in G^s(v)$, one has
\[\chi(v, \xi)=\lim_{T\to \infty}\frac{1}{T}\int_0^Tu_\xi(t)dt.\]
\end{proposition}
\begin{proof}
By the definition of Lyapunov exponents and Proposition \ref{subspace}(5), we have
\begin{equation*}
\begin{aligned}
\chi(v, \xi)&=\limsup_{T\to \infty}\frac{1}{T}\log\|d\phi^t\xi\|=\limsup_{T\to \infty}\frac{1}{T}\log \sqrt{\|J_\xi(T)\|^2+\|J'_\xi(T)\|^2}\\
&=\limsup_{T\to \infty}\frac{1}{T}\log \|J_\xi(T)\|=\limsup_{T\to \infty}\frac{1}{T}\int_0^T(\log j_\xi(t))'dt\\
&=\limsup_{T\to \infty}\frac{1}{T}\int_0^T\frac{j_\xi'(t)}{j_\xi(t)}dt=\limsup_{T\to \infty}\frac{1}{T}\int_0^Tu_\xi(t)dt.
\end{aligned}
\end{equation*}
\end{proof}

Denote by $u(w), w\in SX$ the function satisfying Riccati equation  with respect to stable Jacobi field $J^s_w$. Then $u(w)=\frac{(j^s_w)'}{j^s_w}$ where $j^s_w=\|J^s_w\|$. Note that $u(w)\le 0$ if $M$ has no focal points, but this is not necessarily true when $M$ only has no conjugate points. 

Recall that the geodesic $c_v$ is singular if $v\in \R_1^c$. The following says that the stable (and hence unstable) Lyapunov exponent along a singular geodesic is zero.
\begin{lemma}\label{uniformrec}
If $v\in \R_1^c$, then $\lim_{T\to \infty}\frac{1}{T}\int_0^T u(\phi^t(v))dt=0.$
\end{lemma}
\begin{proof}
Since $v\in \R_1^c$, $J^s_v$ is also an unstable Jacobi field. Since $M$ has bounded asymptote, we actually have 
$$\lim_{T\to \infty}\frac{1}{T}\log j^s_v(T)=\limsup_{T\to \infty}\frac{1}{T}\log j^s_v(T).$$
So the following two limits exist and coincide
$$\lim_{T\to \infty}\frac{1}{T}\int_0^T u(\phi^t(v))dt=\lim_{T\to \infty}\frac{1}{T}\log j^s_v(T).$$

If $\lim_{T\to \infty}\frac{1}{T}\log j^s_v(T)>0$, by \cite[Proposition 5.3]{GR}, one can show that $v\in \mathcal{R}_1$ by explicitly constructing another linearly independent Jacobi field, which is a contradiction. 

If $\lim_{T\to \infty}\frac{1}{T}\log j^s_v(T)<0$, then $\lim_{T\to \infty} j^s_v(T)=0$. Since $J^s_v$ is also an unstable Jacobi field, by bounded asymptote we have for any $t\in \RR$, 
$$j^s_v(t)\le C\lim_{T\to \infty} j^s_v(T)=0$$
which is a contradiction. So we have proved that $$\lim_{T\to \infty}\frac{1}{T}\int_0^T u(\phi^t(v))dt=0.$$
\end{proof}

The proof of Theorem \ref{singulargeodesic}(3) uses an argument based on the following \emph{expansivity} property of a vector $v\in SM$ not tangent to a generalized strip defined in \eqref{strip}. This argument is also used in the proof of Lemma \ref{flat closed geodesic}.
\vspace{.1cm}

\begin{definition}[Cf. Definition 3.2.11 in \cite{KH}]
We say $v\in SM$ has the expansivity property if there exists a small $\delta_0>0$, such that whenever $d(\phi^t(v), \phi^t(w))< \delta_0$, $\forall t\in \mathbb{R}$, then $w=\phi^{t_0}(v)$ for some $t_0$ with $|t_0|<\delta_0$. 
\end{definition}
Let $\mathcal{O} (\lv)$ denote the orbit of $\lv\in SM$ under the geodesic flow.
\begin{lemma}\label{expansivity}
If $\lv\in SM$ is not tangent to a generalized strip, then it has the expansivity property.
\end{lemma}
\begin{proof}
Assume the contrary. Then for any small $\epsilon >0$ less than the injectivity radius of $M$, there exists a point $\lw\in SM$ such that $\lw\notin \mathcal{O} (\lv)$ and $d(c_\lv(t), c_\lw(t))< \epsilon$, $\forall t\in \mathbb{R}$. By the choice of $\epsilon$, we can lift $c_\lv(t)$ and $c_\lw(t)$ to the universal cover $X$ such that
$$d(c_v(t), c_w(t))< \epsilon,\ \ \forall t\in \mathbb{R}.$$
Thus by Lemma \ref{striplemma}, $c_v(t)$ and $c_w(t)$ bound a generalized strip. Hence $\lv$ is tangent to a generalized strip, a contradiction.
\end{proof}

\subsection{Area of ideal triangles}
Given distinct $x,y,z \in \pX$, an \textit{ideal triangle} with vertices $x, y, z$ means the region in $X$ bounded by the three geodesics joining the vertices. In the case when at least one of $x, y, z$ is on $\pX$ (the other points can be inside $X$), we also call the region bounded by the three geodesics an ideal triangle. 

The following theorem is a version of \cite[Theorem 1]{RR} for surfaces without conjugate points.

\begin{theorem}\label{ideal triangle}
Let $M$ be a closed surface with genus $\mathfrak{g} \geq 2$, without conjugate points and with bounded asymptote, and $X$ its universal cover. Suppose that $c_v$ is a singular geodesic in  $X$, i.e., $v\in \mathcal{R}_1^c$, and moreover it is a lift of a closed geodesic $c_\lv$. Then every ideal triangle having $c_v$ as an edge has infinite area.
\end{theorem}
\begin{proof}
Suppose that $h(s)\in SX, s\in [0,a]$ is the curve in the horospherical foliation $\F^{s}$ with $h(0)=v$. We want to show that the area of $\pi \phi^{[0,\infty]}h[0,a]$ is infinite. Assume the contrary. Then obviously $\liminf_{t\to \infty}l(\pi \phi^t h[0,a])=0$ where $l$ denotes the length of the curve. As a consequence of bounded asymptote,
\begin{equation*}
\lim_{t\to \infty}l(\pi \phi^t h[0,a])=0
\end{equation*}

Let $T_0$ be the least period of $\lv$. Since $v\in \R_1^c$, by Lemma \ref{uniformrec}
$$\lim_{n\to \infty}\frac{1}{nT_0}\int_0^{nT_0} u(\phi^t(v))dt=0.$$
By continuity of $u$, for any $s\in [0,a]$,
$$\lim_{n\to \infty}\frac{1}{nT_0}\int_0^{nT_0} u(\phi^t(h(s)))dt=0.$$
Then there exists large enough $n_0>0$ which depends on the unit vector $c(s)$ such that
\begin{equation*}
-\frac{1}{n_0T_0}\int_0^{n_0T_0} u(\phi^t(h(s)))dt\le l(\pi h[0,s]).
\end{equation*}
On compact manifold $M$, $\pi g^{[0,\infty]}(h[0,a])$ is contained in a small compact neighborhood of the closed geodesic $c_v$.
By the compactness and continuity, we can find a uniform constant $n_0>0$ independent of $s\in[0,a]$ such that
\begin{equation}\label{e:av1}
-\frac{1}{n_0T_0}\int_0^{n_0T_0} u(\phi^t(h(s)))dt\le l(\pi h[0,s]).
\end{equation}

By the argument at the beginning of this section, Green bundles $G^{s/u}$ are tangent to the horospherical foliations $\F^{s/u}$ in our setting. Then we have
\begin{equation*}
\begin{aligned}
l(\pi \phi^{nn_0T_0}(h[0,a]))&=\int_0^a \|d\pi d\phi^{nn_0T_0}h'(s)\|ds=\int_0^a \|J^s_{h(s)}(nn_0T_0)\|ds\\
&=\int_0^a \|d\pi d\phi^{(n-1)n_0T_0}h'(s)\|e^{\int_0^{n_0T_0}u(\phi^{t+(n-1)n_0T_0}h(s))dt}ds\\
&\ge \int_0^a \|d\pi d\phi^{(n-1)n_0T_0}h'(s)\|e^{-n_0T_0l(\pi \phi^{(n-1)n_0T_0}h([0,s]))dt}ds\\
&\ge l(\pi \phi^{(n-1)n_0T_0}(h[0,a]))e^{-n_0T_0l(\pi \phi^{(n-1)n_0T_0}h([0,a]))}\\
&\ge l(\pi \phi^{(n-1)n_0T_0}(h[0,a]))\left(1-n_0T_0l(\pi \phi^{(n-1)n_0T_0}h([0,a])\right)
\end{aligned}
\end{equation*}
where we used \eqref{e:av1} and $e^x\ge 1+x$ for $|x|$ small. Denote $a_n:=l(\pi \phi^{nn_0T_0}(h[0,a]))$ and $B=n_0T_0$, we get
\begin{equation}\label{e:series}
a_n\ge a_{n-1}(1-Ba_{n-1})\ge \cdots \ge a_0 \prod_{i=0}^{n-1}(1-Ba_i).
\end{equation}

Since we assume the contrary that the area of $\pi \phi^{[0,\infty]}h[0,a]$ is finite, by bounded asymptote, we have $\sum_{i=0}^{\infty}a_i<+\infty$ and $\lim_{n\to \infty}a_n=0$. Then by \eqref{e:series}, 
\[\lim_{n\to \infty}\prod_{i=0}^{n-1}(1-Ba_i)=0\]
which is equivalent to $\sum_{i=0}^{\infty}a_i=+\infty$. We have a contradiction and therefore the theorem follows.
\end{proof}

\subsection{Singular geodesics}
In this subsection, we discuss some important properties of the singular geodesics. Our Theorem  is a straightforward corollary of these properties. In fact, it is closely related to the following two lemmas (Lemma \ref{flat closed geodesic0} and \ref{flat closed geodesic}).
The first lemma shows that if a flat geodesic converges to a closed one (no matter singular or not), then the former geodesic must also be closed, and coincide with the latter.
\vspace{.1cm}

\begin{lemma}\label{flat closed geodesic0}
Suppose that $y\in \R_1^c\cap \text{Per}(\phi^t)$, and $\lim_{t\to \infty}d(\phi^{t}(y),\phi^{t}(z))=0$ where $c_z(\pm \infty)=\xi_\gamma^{\pm}$ for some $\gamma\in \Gamma$. Then $\mathcal{O}(y)=\mathcal{O}(z)$. In particular, $\mathcal{O}(y)$ is bi-asymptotic to an axis of $\c$.
\end{lemma}

\begin{proof}
\begin{figure}[ht]
\centering
\setlength{\unitlength}{\textwidth}

\begin{tikzpicture}
\draw (0,0) circle (2);

\draw [name path=line 1](0,2) to [out=260,in=60] (-1,-1.73)node[below] {$\tilde{c}_0$};
\draw  [name path=line 2](0,2)to [out=270,in=130] (1, -1.73)node[below] {$\tilde{c}$};
\draw  [name path=line 3](0,2)to [out=280,in=135] (1.42, -1.42)node[right] {$\phi(\tilde{c})$};

\draw [name path=line 4](-1.73, -1) to [out=355,in=182] (1.73,-1)node[right] {$\tilde{\alpha}$};
\draw ([name path=line 5](-2,0) to [out=355,in=182] (2,0) node[right] {$\phi(\tilde{\alpha})$};
\draw [name path=line 6](-1.95, 0.4) to [out=355,in=182] (1.95,0.4)node[right] {$\tilde{\beta}$};
\draw ([name path=line 7](-1.73,1.0) to [out=355,in=182] (1.73,1.0) node[right] {$\phi(\tilde{\beta})$};
\coordinate [label=below:$A$] (A) at (-0.67,-1.05);
\coordinate [label=below:$B$] (B) at (0.51,-1.05);
\coordinate [label=below left:$C$] (C) at (-0.35,-0.05);
\coordinate [label=below left:$D$] (D) at (0.15,-0.05);
\coordinate [label=below right:$E$] (E) at (0.5,-0.03);
\coordinate [label=above left:$C'$] (C') at (-0.15,0.95);
\coordinate [label=above:$D'$] (D') at (0.03,0.95);
\coordinate [label=above right:$E'$] (E') at (0.2,0.95);
\coordinate [label=above left:$A'$] (A') at (-0.25,0.35);
\coordinate [label=above right:$B'$] (B') at (0.05,0.35);
\coordinate [label=above:$F$] (F) at (0,2);
\fill (A) circle (1.5pt)(B) circle (1.5pt)(C) circle (1.5pt)(D) circle (1.5pt)(E) circle (1.5pt)(F) circle (1.5pt)(A') circle (1.5pt)(B') circle (1.5pt) (C') circle (1.5pt)(D') circle (1.5pt)(E') circle (1.5pt);
\end{tikzpicture}

\caption[]{Proof of Lemma \ref{flat closed geodesic0}}

\end{figure}

We can lift geodesics $c_z(t), c_y(t)$ to the universal cover $X$, 
denoted by $\tilde{c}_1(t)$ and $\tilde{c}(t)$ respectively, such that 
$$\lim_{t\to +\infty}d(\tilde{c}_1(t), \tilde{c}(t))=0.$$
In particular, $\tilde{c}_1(+\infty)=\tilde{c}(+\infty)$.

Since the isometry $\gamma$ of $X$ fixes $\tilde c_1(\pm \infty)$, there exists an axis $\tilde c_0$ of $\gamma$ such that $\gamma(\tilde{c}_0(t))=\tilde{c}_0(t+t_0)$ and $\tilde{c}_0(\pm \infty)=\tilde{c}_1(\pm \infty)$. Moreover, on the boundary of the disk $\pX$, $\gamma$ fixes exactly two points $\tilde{c}_0(\pm \infty)$, and for any other point $a\in \pX\setminus\{\tilde{c}_0(\pm \infty)\}$, $\lim_{n\to +\infty}\gamma^n(a)= \tilde{c}_0(+\infty)$.

Assume $\tilde{c}(\pm \infty)$ is not fixed by $\gamma$. Then $\tilde{c}$ and $\gamma(\tilde{c})$ do not intersect, since $\gamma(\tilde{c})(+\infty)=\tilde{c}(+\infty)$. Replacing $\gamma$ by $\gamma^N$ for a large enough $N\in \mathbb{N}$ if necessary, we know that the position of $\gamma(\tilde{c})$ must be as shown in Figure 1. We then pick another two geodesics $\tilde{\alpha}$ and $\tilde{\beta}$ as in Figure 1. The image of $ABB'A'$ under $\gamma$ is $CEE'C'$. Since $\gamma$ is an isometry, it preserves area. So the area of $ABCD$ is equal to the region $A'B'DEE'C'$, and thus greater than the area of the region $DEE'D'$. We can let $A'$ and $B'$ approach $F$. In this process, the area of the region $DEE'D'$ approaches the area of the ideal triangle $DEF$. But since $c_y$ is a singular and closed geodesic, the area of $DEF$ is infinite by Theorem \ref{ideal triangle}. Then $ABCD$ has infinite area which is absurd. So $\gamma(\tilde{c}(\pm \infty))=\tilde{c}(\pm \infty)$.

Therefore $\tilde{c}(\pm\infty)= \tilde{c}_0(\pm\infty)$. Then either $\tilde{c}(t)$ and $\tilde{c}_0(t)$ bound a strip by Lemma \ref{striplemma} or $\tilde{c}(t)=\tilde{c}_0(t)$. Recall that $\lim_{t\to +\infty}d(\tilde{c}(t), \tilde{c}_0(t))=0$, we must have $\tilde{c}(t)=\tilde{c}_0(t)$. Hence $\mathcal{O}(y)=\mathcal{O}(z)$.
\end{proof}
\begin{remark}
The above lemma holds more generally without assuming that $y,z$ are singular. In fact, this is a rather general property of isometries on $X$. See the proof of Lemma \ref{isometries}. The above proof is of independent interest and uses geometric properties of singular geodesics. 
\end{remark}

To refine Lemma \ref{flat closed geodesic0}, we need the following lemma. 
\begin{lemma}\label{boundary}
Let $M$ be a closed surface with genus $\mathfrak{g} \geq 2$, without conjugate points and with bounded asymptote. If $\lw$ is periodic, $\lw' \in \F^s(\lw)\subset SM$ and
$$\liminf_{t\to +\infty} d(c_\lw(t), c_{\lw'}(t))=\delta >0,$$
then $c_\lw(t)$ and $c_{\lw'}(t)$ converge to the boundaries of a nontrivial generalized strip.
\end{lemma}
\begin{proof}
Suppose that $\lim_{i\to +\infty} d(c_\lw(s_i), c_{\lw'}(s_i))=\delta >0$. By passing to a subsequence, we can assume that $\lim_{i\to +\infty}\phi^{s_i}(\lw) = v$ and $\lim_{s_i\to +\infty}\phi^{s_i}(\lw')= \lv'$. Then $\lv'\in \F^s(\lv)$. Then by Proposition \ref{horofoliation}, for any $t\in \mathbb{R}$
$$d(c_{\lv}(t), c_{\lv'}(t))= \lim_{s_i\to +\infty}d(c_{\lw}(t+s_i), c_{\lw'}(t+s_i))\in [\delta,Q_1]$$
where $Q_1:=Ad_s(c_{\lv}(0), c_{\lv'}(0))+B$.
Hence we can lift the geodesics to $X$ such that $v'\in \F^s(v)$ and
$$d(c_v(t), c_{v'}(t)) \in [\delta,Q_1] \text{\ \ \ for\ } \forall t \in \mathbb{R}.$$
By Lemma \ref{striplemma}, $c_v(t)$ and $c_{v'}(t)$ are the boundaries of a generalized strip of width at least $\delta$.

By assumption, $\lw$ is periodic and thus $\lv\in \mathcal{O}(\lw)$. So there exists $\lv_1\in \I(\lw)$ such that $ \lv_1\in \O (\lv')$. 
\begin{enumerate}
\item If $\lw'$ is on the arc of $\F^s(\lw)$ between $\lw$ and $\lv_1$, then $\lw'$ is inside the generalized strip bounded by geodesics $c_v(t)$ and $c_{v'}(t)$, and we are done.
\item If $\lv_1$ is on the arc of $\F^s(\lw)$ between $\lw$ and $\lw'$, then:
\begin{itemize}
    \item If $\liminf_{t\to +\infty} d(c_{\lv_1}(t), c_{\lw'}(t))=0,$ we are done since then by bounded asymptote
    $$\lim_{t\to +\infty} d(c_{\lv_1}(t), c_{\lw'}(t))=0,$$ and $\lw$ and $\lv_1$ are the boundaries of a nontrivial generalized strip. 
    \item If otherwise $\liminf_{t\to +\infty} d(c_{\lv_1}(t), c_{\lw'}(t))>0,$ then repeating the above process produces another $\lv_2\in \I(\lw)$. If we are not done in finite steps, since $\I(\lw)$ has finite length by Proposition \ref{horofoliation}, we will arrive at a limit point $\lv_\infty\in \I(\lw)$ such that $\lim_{t\to +\infty} d(c_{\lv_\infty}(t), c_{\lw'}(t))=0.$ Then we are done.
\end{itemize}
\end{enumerate}
\end{proof}

Now Lemma \ref{flat closed geodesic0} can be strengthened to the following.
\begin{lemma}\label{flat closed geodesic}
 Suppose that $y\in \R_1^c\cap \text{Per}(\phi^t)$ and $z\in \omega(y)$ where $c_z(\pm \infty)=\xi_\gamma^{\pm}$ for some $\gamma\in \Gamma$. Then $\mathcal{O}(y)=\mathcal{O}(z)$. In particular, $\mathcal{O}(y)$ is bi-asymptotic to an axis of $\c$.
\end{lemma}

\begin{proof}
Suppose that there exist $s_k \to +\infty$ such that $\phi^{s_k}(y) \to z$. If $\phi^{s_k}(y)\in \F^s(z)$ for some $k$ then we must have $\lim_{t\to \infty}d(\phi^{t+s_k}(y),\phi^{t}(z))=0$ by Lemma \ref{boundary} and its proof. Then by Lemma \ref{flat closed geodesic0}, we have $\mathcal{O}(\phi^{s_k}(y))=\mathcal{O}(z)$. So we are done.

Suppose that $\phi^{s_k}(y)\notin \F^s(z)$ for any $k$. Note that if $y\neq z$ then $y$ and $z$ can not be tangent to a same generalized strip. Therefore by Lemma \ref{expansivity}, for any small $\epsilon_0>0$ and any large $k$, there exists a $l_k$ with $l_k \to +\infty$ such that
$$d(\phi^{l_k}(\phi^{s_k}(y)), \phi^{l_k}(z))=\epsilon_0 \text{\ \ or\ \ }d(\phi^{l_k}(-\phi^{s_k}(y)), \phi^{l_k}(-z))=\epsilon_0$$
where we take $l_k$ to be the smallest positive number to satisfy the above equality. We only consider the first case. By taking a subsequence, we can assume that
\begin{equation*}
\phi^{l_k}(\phi^{s_k}(y)) \to y^+ \text{\ \ and \ \ }\phi^{l_k}(z) \to z^+
\end{equation*}
as $k\to +\infty$. Then $z^+$ is such that $c_{z^+}(\pm \infty)=\xi_\gamma^{\pm}$ and $d(y^+,z^+)=\epsilon_0$. 

We claim that in the above process we can make sure that the footpoint of $y^+$ is outside of the generalized strip bounded by $c_\gamma$ and $c_{z}$, where $c_\c$ is an axis (fixed once for all) of $\c$. Indeed, we can take $\phi^{s_k}(y)$ so that the footpoint of $\phi^{l_k}(\phi^{s_k}(y))$ is outside of the above generalized strip. If this is not the case, we then consider $-\phi^{s_k}(y)$ and $-z$ instead. Under this change, we still get a limit point $y^+$ and $z^+$, so that $y^+$ is outside of the generalized strip bounded by $c_\gamma$ and $c_{z}$.

Then for any $t>0$, since $0 < -t+l_k < l_k$ for large enough $k$, one has
$$d(\phi^{-t}(y^+), \phi^{-t}(z^+))=\lim_{k \to +\infty}d(\phi^{-t+l_k+s_k}(y)),\phi^{-t+l_k}(z) ) \leq \epsilon_0.$$
So $-y^+ \in \F^s(-z^+)$. Then we have the following two cases:
 \begin{enumerate}
 \item $\liminf_{t\to \infty}d(\phi^{t}(-y^+), \phi^t(-z^+))= 0$. Then by bounded asymptote, $\lim_{t\to \infty}d(\phi^{t}(-y^+), \phi^t(-z^+))= 0$. By Lemma \ref{flat closed geodesic0}, $c_{y^+}(\pm \infty)=\xi_\gamma^{\pm}$ and in fact $-y^+=-z^+$. This contradicts to $d(y^+,z^+)=\epsilon_0$.
 
 
 \item $\liminf_{t\to \infty}d(\phi^{t}(-y^+), \phi^t(-z^+))= \delta$ for some $\delta>0$. By Lemma \ref{boundary}, $\phi^t(-y^+)$ converges to the boundary of a generalized strip containing $-z^+$. Then by Lemma \ref{flat closed geodesic0} $c_{y^+}$ and $c_{z^+}$ are boundaries of a generalized strip of width at least $\delta$. Recall that $y^+$ is outside of the generalized strip  bounded by $c_\gamma$ and $c_{z}$ which implies that we have enlarged the generalized strip bounded by  $c_\gamma$ and $c_{z}$ to the one bounded by  $c_\gamma$ and $c_{y^+}$.
 Now note that $y^+ \in \omega(y)$, so we can apply all the arguments above to $y^+$ instead of $z$. Either we are arriving at a contradiction as in case (1) and we are done, or we get an even larger generalized strip. Since $\I(\dot{c}_\gamma)$ has finite length by Propostion \ref{horofoliation}, at last we will get a limit vector $y_\infty\in \I(\dot{c}_\gamma)$ such that $y_\infty\in \omega(y)$ if we are always in case (2). Then we can apply all the arguments above to $y_\infty$ and we must arrive at case (1). So we are done.
\end{enumerate}
\end{proof}

\begin{theorem}\label{A}
If $\R_1^c \subset \text{Per}(\phi^t)$, then there is a finite decomposition of $\R_1^c$:
$$\R_1^c = \mathcal{O}_1 \cup \mathcal{O}_2 \cup \ldots \mathcal{O}_k \cup \mathcal{S}_1\cup \mathcal{S}_2 \cup \ldots \cup \mathcal{S}_l,$$
where each $\mathcal{O}_i, 1\leq i \leq k$, is an isolated periodic orbit and each $\mathcal{S}_j, 1\leq j \leq l$, consists of vectors tangent to a generalized strip. Here $k$ or $l$ are allowed to be $0$ if there is no isolated closed singular geodesic or no generalized strip.
\end{theorem}
\begin{proof}

Assume the contrary to Theorem \ref{A}. Then there exists a sequence of different vectors $x'_k \in \R_1^c$ such that $\lim_{k\to +\infty}x'_k=x$ for some $x\in \R_1^c$. Here different $x'_k$ are tangent to different closed geodesics or to different generalized strips, and $x$ is tangent to a closed geodesic or to a generalized strip. For large enough $k$, we suppose that $d(x'_k, x) <\frac{1}{k}$. 

Fix a small number $\epsilon_0 >0$. It is impossible that $d(\phi^t(x'_k), \phi^t(x)) \leq \epsilon_0$, $\forall t>0$. For otherwise, $c_{x'_k}(t)$ and $c_x(t)$ are positively asymptotic closed singular geodesics. They must be tangent to a common generalized strip by Lemma \ref{boundary} and Lemma \ref{flat closed geodesic0}. This is impossible since different $x'_k$ are tangent to different closed geodesics or to different flat strips. Hence there exists a sequence $s_k\to +\infty$ such that
$$d(g^{s_k}(x'_k), g^{s_k}(x))=\epsilon_0,$$
and
$$d(g^{s}(x'_k), g^{s}(x))\leq \epsilon_0, \ \forall 0\leq s <s_k.$$

Let $y_k:=g^{s_k}(x)$ and $y'_k:=g^{s_k}(x'_k)$. Without loss of generality, we suppose that $y_k\to a$ and $y'_k\to b$.
Then $d(a, b)=\lim_{k_i\to +\infty}d(g^{s_{k_i}}(x_{k_i}),g^{s_{k_i}}(x'_{k_i}))=\epsilon_0.$ We get that $d(a, b)=\epsilon_0$.
For any $t<0$, since $0<s_{k_i}+t<s_{k_i}$ for some large $k_i$, one has
$$d(\phi^t(a), \phi^t(b))=\lim_{i\to +\infty}(d(g^{s_{k_i}+t}(x_{k_i}),g^{s_{k_i}+t}(x'_{k_i}))) \leq \epsilon_0.$$
Hence we get that $d(\phi^t(a), \phi^t(b)) \leq \epsilon_0$,  $\forall t\leq 0$.

Replacing $x, x_k'$ by $-x, -x_k'$ respectively and applying the same argument, we can obtain two points $a^-,b^-$ such that $d(a^-, b^-)=\epsilon_0$ and $$d(\phi^t(a^-), \phi^t(b^-)) \leq \epsilon_0,\quad \forall t\leq 0.$$
Then we have the following three cases:
 \begin{enumerate}
 \item $\liminf_{t\to \infty}d(\phi^t(-a), \phi^t(-b))=0$. By bounded asymptote, this is equivalent to $\lim_{t\to \infty}d(\phi^t(-a), \phi^t(-b))=0$. By Lemma \ref{flat closed geodesic0}, $-a=-b$. This contradicts to $d(a,b)=\epsilon_0$.
 \item $\liminf_{t\to \infty}d(\phi^t(-a^-), \phi^t(-b^-))=0$. Again by Lemma \ref{flat closed geodesic0}, $-a^-=-b^-$. This contradicts to $d(a^-,b^-)=\epsilon_0$.
 \item $\liminf_{t\to \infty}d(\phi^t(-a), \phi^t(-b))>0$ and $\liminf_{t\to \infty}d(\phi^t(-a^-), \phi^t(-b^-))>0$.
\end{enumerate}

For case (3), by Lemmas \ref{boundary} and \ref{flat closed geodesic}, $c_{a}$ and $c_{x}$ coincide. And moreover, $c_{x}$ and $c_b$ are boundaries of a generalized strip of width $\delta_1>0$. Similarly, $c_{x}$ and $c_{b^-}$ are boundaries of a generalized strip of width $\delta_2>0$. We claim that these two generalized strips lie on the different sides of $c_x$. Indeed, we choose $\epsilon_0$ small enough and consider the $\epsilon_0$-neighborhood of the closed geodesic $c_x$ which contains two regions lying on the different sides of $c_x$. By the definition of $b$ and $b^-$, they must lie in different regions as above. This implies the claim. 

In this way we get a generalized strip of width $\delta_1+\delta_2$ and $x$ is tangent to the interior of this generalized strip. Since $g^{s_k}(x'_k)\to b$, we can repeat all the arguments above to $b, g^{s_k}(x'_k)$ instead of $x, x_k'$. Then either we are arriving at a contradiction as in case (1) or case (2) and we are done, or we get a generalized strip of width greater than $\delta_1+\delta_2$ and $b$ is tangent to the interior of the generalized strip. Proceeding in this way, either we arrive at case (1) or case (2) in finite steps, or at last we will get a limit generalized strip which cannot be enlarged by Proposition \ref{compact}. Applying the above argument to this limit generalized strip, we have to arrive at case (1) or case (2) and thus obtain a contradiction. So we are done with the proof.
\end{proof}

\begin{proof}[Proof of Theorem \ref{singulargeodesic}]
By the argument at the beginning of this section, the geodesic flow has a periodic hyperbolic point. Then by Theorem \ref{regular}, $\Delta$ agrees $\nu$-almost everywhere with the open dense set $\mathcal{R}_1$. This proves the first statement in Theorem \ref{singulargeodesic}. Then Theorem \ref{singulargeodesic}(1) follows from Theorem \ref{ideal triangle}, and Theorem \ref{singulargeodesic}(2) follows from Theorem \ref{A}.
\end{proof}

\section{Patterson-Sullivan measure and Bowen-Margulis measure}

\subsection{Patterson-Sullivan measure}
In this subsection we assume that $X$ is a simply connected uniform visibility manifold without conjugate
points. We will recall the construction of the Patterson-Sullivan measure presented in \cite{LLW2} and obtain some properties, especially the shadowing lemma.  

\begin{definition}\label{density}
Let $X$ be a simply connected uniform visibility manifold without conjugate points and $\Gamma$ a discrete subgroup of $\text{Iso}(X)$, the group of isometries of $X$. For a given constant $r>0$, a family of finite Borel measures $\{\mu_{p}\}_{p\in X}$
on $\pX$ is called an $r$-dimensional \emph{Busemann density} if
\begin{enumerate}
\item for any $p,\ q \in X$ and $\mu_{p}$-a.e. $\xi\in \pX$,
$$\frac{d\mu_{q}}{d \mu_{p}}(\xi)=e^{-r \cdot b_{\xi}(q,p)}$$
 where $b_{\xi}(q,p)$ is the Busemann function;
\item $\{\mu_{p}\}_{p\in X}$ is $\Gamma$-equivariant, i.e., for all Borel sets $A \subset \pX$ and for any $\c \in \Gamma$,
we have $$\mu_{\c p}(\c A) = \mu_{p}(A).$$
\end{enumerate}
\end{definition}
In \cite{LLW2}, the authors constructed Busemann density via Patterson-Sullivan construction. For each pair of points $(p,p_{0})\in X \times X$ and $s \in \mathbb{R}$,
\emph{Poincar\'e series} is defined as
\begin{equation*}
P(s,p,p_{0}):= \sum_{\alpha \in \Gamma}e^{-s\cdot d(p,\alpha p_{0})}.
\end{equation*}
This series may diverge for some $s>0$. The number 
$$\delta_\Gamma:= \inf \{s \in \mathbb{R}\mid P(s,p,p_{0}) < +\infty\}$$
is called the \emph{critical exponent of $\Gamma$}. By the triangle inequality, it is easy to check that $\delta_\Gamma$ is independent of $p$ and $p_{0}$. We say $\Gamma$ is of \emph{divergent type} if the Poincar\'e series diverges when $s=\delta_\C$.
Whether $\Gamma$ is of divergent type or not is independent of the choices of $p$ and $p_{0}$.

Fix a point $p_{0}\in X$. For each $s > \delta_\Gamma$ and $p \in X$, consider the measure:
\begin{equation*}
\mu_{p,s}:=\frac{\sum_{\alpha \in \Gamma}e^{-s\cdot d(p,\alpha p_0)}\delta_{\alpha p_{0}}}{\sum_{\alpha\in\Gamma}e^{-s \cdot d(p_{0},\alpha p_{0})}},
\end{equation*}
where $\delta_{\alpha p_{0}}$ is the Dirac measure at point $\alpha p_{0}$. The following properties of $\mu_{p,s}$ can be easily checked by using the triangle inequality:
\begin{enumerate}
\item $e^{-s\cdot d(p, p_{0})} \leq \mu_{p,s}(\overline{X}) \leq e^{s\cdot d(p, p_{0})}$.
\item $\Gamma p_{0} \subset \text{supp}\mu_{p,s} \subset \overline{\Gamma p_{0}}$, where $\Gamma p_{0}$ is the orbit of $p_0\in X$ under the action of $\Gamma$.
\end{enumerate}

Here in constructing a Busemann density,  we consider the case that the group $\Gamma$ is of divergent type first.
When $\Gamma$ is of convergent type, Patterson showed in \cite{Pat} that one can weight the Poincar\'e series with a positive monotone increasing function to make the refined Poincar\'e series diverge at $s=\delta_\Gamma$.
More precisely, we consider a positive monotonely increasing function $g:\mathbb{R}^{+} \rightarrow \mathbb{R}^{+}$ with
$\lim_{t\rightarrow +\infty}\frac{g(x+t)}{g(t)}=1$, $\forall x \in \mathbb{R}^{+}$,
such that the weighted Poincar\'e series 
$$\widetilde{P}(s,p,p_{0}):=\sum_{\alpha\in \Gamma}g(d(p,\alpha p_{0}))e^{-s\cdot d(p,\alpha p_{0})}$$
diverges at $s=\delta_\Gamma$.  For the construction of the function $g$, see \cite[Lemma 3.1]{Pat}. Then for the weighted Poincar\'e series $\widetilde{P}(s,p,p_{0})$,
we can always construct a Busemann density by using the same method as follows.

Now for $p\in X$, consider a weak$^{\star}$ limit
$$\lim_{s_{k}\downarrow h}\mu_{p,s_{k}}=\mu_{p}.$$
Since $P(s,p_{0},p_{0})$ is divergent for $s=\d_\C$ and $\Gamma$ is discrete, it is obvious that $\text{supp}(\mu_{p})\subset \overline{\Gamma p_{0}} \cap \pX$. 

\begin{proposition}
There is a sequence $s_k \to \d_\C$ as $ k\to \infty$ such that for every $p\in X$ the weak$^*$ limit $\lim_{k\to \infty} \mu_{p,s_k}=:\mu_p$ exists. Moreover, $\{\mu_p\}_{p\in X}$ is a $\d_\C$-dimensional Buseman density.
\end{proposition}
\begin{proof}
The proof is the same as that of \cite[Proposition 5.1]{CKW1}. Note that we do not need to assume that $X$ admits a compact quotient. 
\end{proof}
\begin{remark}
The uniform visibility condition is essential in the above construction. For example, under that condition, $\partial X$ with cone topology is well defined and the continuity of Busemann functions is guaranteed. 
\end{remark}

\subsection{Bowen-Margulis measure}
Let $P:SX \to \pX \times \pX$ be the projection given by $P(v)=(v^-,v^+)$. Denote by $\mathcal{I}^{P}:=P(SX)=\{P(v): v \in SX\}$ the subset of pairs in $\pX \times \pX$
which can be connected by a geodesic. Note that the connecting geodesic may not be unique. Since $X$ has uniform visibility, every pair $\xi\neq \eta$ in $\pX$ can be connected by a geodesic, i.e.,  
$$\partial^2X=(\pX \times \pX)\setminus \{(\xi,\xi): \xi\in \pX\}.$$

Fix a point $p\in X$, we can define a $\Gamma$-invariant measure $\overline{\mu}$
on $\partial^2X$ by the following formula:
\begin{equation}\label{current}
d \overline{\mu}(\xi,\eta) := e^{\d_\C\cdot \beta_{p}(\xi,\eta)}d\mu_{p}(\xi) d\mu_{p}(\eta),
\end{equation}
where $\beta_{p}(\xi,\eta):=-\{b_{\xi}(q, p)+b_{\eta}(q, p)\}$ is the \textit{Gromov product}, and $q$ is any point on a geodesic $c$ connecting $\xi$ and $\eta$. It is easy to see that the function $\beta_{p}(\xi,\eta)$ does not depend on the choice of $c$ and $q$.
In geometric language, the Gromov product $\beta_{p}(\xi,\eta)$ is the length of the part of a geodesic
$c$ between the horospheres $H_{\xi}(p)$ and $H_{\eta}(p)$.

Let $M=X/\C$. Define an equivalence relation on $SM$ by writing $v\sim w$ iff $w\in \I(v)$.
Write $[v]=\I(v)$ for the equivalence class of $v$. Let $\tilde \chi: SX\to SX/\sim$ and  $\chi: SM\to SM/\sim$ be the quotient maps. 
These are continuous maps when we equip the quotient spaces with the metric 
$$d([v], [w]) = \min\{d(v', w') : v'\in [v], w'\in [w]\}.$$ 
The geodesic flow $\phi^t$ takes equivalence classes to equivalence classes, $\phi^t[v] = [\phi^tv]$, hence it induces a flow $\psi^t: SM/\sim\to SM/\sim$.
\begin{theorem}(\cite[Theorem 1.1]{MR})\label{expansivefactor}
Let $M$ be a compact $C^\infty$ $n$-dimensional manifold without conjugate points and with uniform visibility universal cover $X$. Then the geodesic flow $\phi^t$ is time-preserving semi-conjugate to the continuous flow $\psi^t$ acting on a compact metric space $SM/\sim$ of topological dimension at least $n-1$. Moreover, $\psi^t$ is expansive, topologically mixing and has a local product structure.   
\end{theorem}

We do not need compactness of $M$ in the following construction.
Since $\I(v)$ is compact and $SM$ is a Polish space, the measurable selection theorem of Kuratowski and Ryll-Nardzewski \cite[Section 5.2]{Sr} guarantees existence of a Borel measurable map $V: SM/\sim \to SM$ such that $\chi\circ V$ is the identity. Recall that $\pr:SX\to SM$ is the natural projection. Then for every $v\in SX$, $\pr^{-1} (V (\pr[v]))$ intersects $[v]$ in a single point, which we denote by $\tilde V([v])$. We conclude that $\tilde V: SX/\sim \to SX$ is a measurable map such that $\tilde \chi\circ \tilde V$ is the identity, and moreover $\tilde V (\gamma_*[v]) = \gamma_*\tilde V([v])$ for any $\c\in \C$.
Now define a measure $\lambda_{\xi,\eta}$ on each $P^{-1}(\xi, \eta)$ by fixing any $v\in P^{-1}(\xi, \eta)$ and putting for
a Borel measurable set $A\subset SX$
$$\lambda_{\xi,\eta}(A):=\text{Leb}\{t\in \RR: \tilde V([\phi^tv])\in A\}.$$
Note that this is independent of the choice of $v$. Use this to define a measure $m$ on $SX$ by
\begin{equation*}
m(A)=\int_{\partial^2X} \lambda_{\xi,\eta}(A) d \overline{\mu}(\xi,\eta).
\end{equation*}
Observe that $m$ is $\Gamma$-invariant, and it descends to a Borel measure $m_\C$ on $SM$. 

The measure $m_\C$ is not necessarily invariant under the geodesic flow, since along a flow orbit inside a generalized strip, the choice for the value of measurable function $V$ may vary. However, the measure $\lm_\C = \chi_*m_\C$ is a flow invariant measure on $SM/\sim$ because $\chi_*\lambda_{\xi,\eta}$ is flow invariant on each  $P^{-1}(\xi,\eta)/\sim$. 

Now let us suppose that $M$ is compact and then clearly $m_\C$ is finite. Without loss of generality we scale the measure so that $m_\C(SM) = 1$. Then the set $\chi_*^{-1}\lm_\C$ of Borel probability measures is weak$^*$ compact and closed under $(\phi^t)_*$ for all $t\in \RR$. So the usual argument from the Krylov-Bogolyubov theorem for producing an invariant probability measure shows that there is a flow invariant Borel probability measure $\mu$ on $SM$ with $\chi_*\mu=\lm_\C$. This lifts to a $\Gamma$-invariant and flow invariant Borel measure $\tilde \mu$ on $SX$ by 
\begin{equation*}
\tilde \mu(A)=\int_{SM} \#(\text{pr}^{-1}(\lv)\cap A)d\mu(\lv).
\end{equation*}
We will call both $m_\C$ and $\mu$ \textit{Bowen-Margulis measures} of the geodesic flow on $SM$.

\begin{remark}\label{noncom}
The above construction of $\mu$ for a compact manifold $M$ without conjugate points which satisfies the first two conditions in Definition \ref{classH} of class $\mathcal{H}$ has appeared in \cite{CKW1}. In the last section, we will prove the Hopf-Tsuji-Sullivan dichotomy for the geodesic flow on a uniform visibility (not necessarily compact) manifold without conjugate points with respect to the Bowen-Margulis measure $m_\C$.
\end{remark}

\subsection{Relation to the MME}
In this subsection, we always assume that $M=X/\C$ is a compact uniform visibility manifold with universal cover $X$. We shall show that the Bowen-Margulis measure $\mu$ coincides with the unique MME under the setting of Theorem \ref{mmeuni}.

By Freire and Ma\~n\'e's Theorem (cf.~\cite{FrMa}), the topological entropy of the geodesic flow coincides with the volume entropy, i.e., for any $x\in X$
$$h_{\text{top}}(\phi^1)=\lim_{r\to \infty}\frac{\log \text{Vol}B(x,r)}{r}$$
where $B(x,r)$ is the open ball of radius $r$ centered at $x$ in $X$. Denote $h=h_{\text{top}}(\phi^1)$. By the proof of \cite[Lemma 3.3]{LWW}, we have $\d_\C=h$ where $\d_\C$ is the critical exponent of $\C=\pi_1(M)$.

\begin{lemma}\label{support}
For any $p \in X$, $\emph{\text{supp}}(\mu_{p})=\pX$.
\end{lemma}
\begin{proof}
Suppose $\text{supp}(\mu_{p})\neq \pX$, then there exists $\xi \in \pX$ which is not in $\text{supp}(\mu_{p})$. Then we can find an open neighborhood $U$ of $\xi$ in $\pX$ such that $\mu_{p}(U)=0$.
Since $M$ is compact, the geodesic flow is topologically transitive and hence the action of $\C$ on $\pX$ is minimal (cf. \cite[Proposition 3.6 and the proof of Proposition 2.8]{Eb1}). Then $\bigcup_{\alpha \in \Gamma}\alpha U=\pX$. So for any $\eta \in \pX$, there is an element $\alpha \in \Gamma$ such that $\eta \in \alpha U$. By the $\Gamma$-invariance $\mu_{\alpha p}(\alpha U)=\mu_{p}(U)=0$. Since $\mu_{p}$ is equivalent to $\mu_{\alpha p}$, we have $\mu_{p}(\alpha U)=0$. Therefore $\text{supp}(\mu_{p}) = \emptyset$, a contradiction.
\end{proof}
For each $\xi \in \pX$ and $p \in X$, define the projections
\begin{itemize}
\item{} ~~$pr_{\xi}: X \rightarrow \pX, ~~q \mapsto c_{\xi,q}(+\infty)$,
\item{} ~~$pr_{p}: X\backslash \{p\} \rightarrow \pX, ~~q \mapsto c_{p,q}(+\infty)$,
\end{itemize}
where $c_{\xi,q}$ is the geodesic satisfying $c_{\xi,q}(-\infty)=\xi$ and $c_{\xi,q}(0)=q$. For $A \subset X$ and $p \in X$, $pr_p(A)=\{pr_p(q)\mid q \in A \backslash \{p\} \}.$

\begin{lemma}\label{shadowopen}
There exists $R>0$ such that for all $x\in \overline X$ and $p\in X$,  the shadow $pr_x B(p, R)$ contains an open set in $\pX$, where $B(p,R)$ is the open ball of radius $R$ centered at $p$.
\end{lemma}
\begin{proof}
For $x\in \overline X$ and $p\in X$, let $v=-\dot c_{p,x}(0)$. For small $\rho>0$, the set $A(v,\rho):=\{c_w(\infty): \angle_p(v,w)<\rho\}$ is open in $\pX$. Since $X$ is a uniform visibility manifold, for any $\eta\in A(v,\rho)$, there exists at least one geodesic $c_{x, \eta}$ connecting $x$ and $\eta$. For any $0<\e<\pi-\rho$, let $R=R(\e)$ be the constant in Definition \ref{vis} of uniform visibility manifolds. Assume that the geodesic $c_{x, \eta}$ does not intersect the ball $B(p, R)$, then $\angle_p(-v,\dot c_{p,\eta}(0))<\e$. Therefore, 
$$\angle_p(v,-v)\le \angle_p(v,\dot c_{p,\eta}(0))+\angle_p(-v,\dot c_{p,\eta}(0))<\rho+\e<\pi$$
which is a contradiction. Thus $A(v, \rho)\subset pr_x B(p, R)$ and the lemma follows. 
\end{proof}

\begin{proposition}\label{shadow}
Let $X$ be a simply connected uniform visibility manifold which has a compact quotient, and $\{\mu_{p}\}_{p\in X}$ be an $\d_\C$-dimensional Busemann density on $\pX$. Fix $r>R$ where $R$ is from Lemma \ref{shadowopen}.
\begin{enumerate}
\item[(1)] There exists $l=l(r)$ such that $\mu_{p}(pr_{x}B(p,r))\geq l$, for all $p \in X$ and $x \in \overline{X}$;
\item[(2)] there is a positive constant $a=a(r)$ such that for any $x \in X \setminus \{p\}$ and $\xi=c_{p,x}(-\infty)$,
we have $$\frac{1}{a} e^{-h d(p,x)} \leq \mu_{p}(pr_{\xi}B(x,r))\leq a e^{-h d(p,x)};$$
\item[(3)] there is a positive constant $b=b(r)$ such that for any $p, x \in X$,
$$\frac{1}{b} e^{-h d(p,x)} \leq \mu_{p}(pr_{p}B(x,r))\leq b e^{-h d(p,x)}.$$
\end{enumerate}
\end{proposition}
\begin{proof}
The proof is the same as that of \cite[Proposition 5.4]{CKW1}, for which Lemmas \ref{support} and \ref{shadowopen} are crucial. We have established these two lemmas based on the uniform visibility and the compactness of $M$.
\end{proof}

Let $\mu$ be the Bowen-Margulis measure of the geodesic flow on $SM$,  and $d_1$ be the Knieper metric on $SM$ defined as 
\[d_1(\lv,\lw):=\max_{t\in [0,1]}d(c_\lv(t), c_\lw(t)).\]
Consider a finite measurable partition $\mathcal{A}=\{A_{1},...,A_{l}\}$ of $SM$ 
satisfying $\mu(\partial A_i)=0, i=1, \cdots l$ and that the radius of every $A_i$ under $d_1$ is less than $\epsilon < \min\{\text{Inj}(M),R\}$,
where $R>0$ is the constant in Lemma \ref{shadowopen}.
Let $\mathcal{A}^{(n)}_{\phi}:=\vee^{n-1}_{i=0}\phi^{-i}\mathcal{A}$, $n=1,2,\cdots$.
For each $\alpha \in \mathcal{A}^{(n)}_{\phi}$, we know that if  $\lv\in\alpha$, then
$$\alpha \subset \bigcap^{n-1}_{k=0}\phi^{-k}B_{d_{1}}(\phi^{k}\lv,\epsilon),$$
where $B_{d_{1}}(\phi^{k}\lv,\epsilon)$ denotes the $\epsilon$-ball of $\phi^{k}\lv\in SM$ under the Knieper metric $d_1$.
If $\lw\in \a$, then $d(c_\lv(t), c_\lw(t))\le \e$ for all $0\le t\le n$.
\begin{proposition}\label{MME}
There exists a constant $a > 0$ such that for any $\alpha \in \mathcal{A}^{n}_{\phi}$, we have $\mu(\alpha)\leq e^{-hn}a$. As a consequence, $\mu$ is a MME.
\end{proposition}
\begin{proof}
\textbf{Step 1.} Let $p\in X$ be the reference point used in the definition of $\bar \mu$ in \eqref{current}, and let $v$ be a lift of $\lv$ such that $d(p, \pi\lv)\le \text{diam}M=:r_0$.  Since $\text{diam}\a<\text{Inj}(M)$, we can lift $\a$ to  $\tilde \a$ in $X$ such that for any $v, w\in \tilde \a$, then $d(c_v(t), c_w(t))\le \e$ for all $0\le t\le n$. Let $c_v(n)=x$, then there exists constant $r_1>0$,
$$P(\tilde \a)\subset  \bigcup_{\eta\in pr_x(B(p,r_1))}\{\eta\}\times pr_\eta(B(x,\e)).$$
Indeed, let $\eta=c_w(-\infty)$ and $c_{\eta,x}$ the geodesic connecting $\eta$ and $x$ such that $c_{\eta,x}(n)=x$. Since $c_{\eta,x}$ and $c_w$ are negatively asymptotic, by Proposition \ref{horofoliation}, 
$d(c_{\eta,x}(t),c_w(t))\le Ad_s(c_{\eta,x}(n),c_w(n))+B\le 2A\e+B=:R_1$
for any $0\le t\le n$.
Then $d(c_{\eta,x}(0),p)\le d(c_{\eta,x}(0),\pi w)+d(\pi w, p)\le R_1+r_0+\e=:r_1.$
This implies that $\eta\in pr_x(B(p,r_1))$.

\textbf{Step 2.} $\tilde \mu(\tilde\alpha)\leq e^{-hn}a$ for some constant $a>0$.

Indeed, for any $\eta\in pr_x(B(p,r_1))$, choose a point $q\in B(p,r_1)$ that lies on $c_{\eta, x}$.
First notice that 
$$d(q,x)\ge d(x,\pi v)-d(\pi v, p)-d(p,q)\ge n-r_0-r_1.$$ 
Then by Proposition \ref{shadow}(2),
\begin{equation}\label{proj}
 \mu_p(pr_\eta(B(x,\e)))\le e^{hd(p,q)}\mu_q(pr_\eta(B(x,\e)))\le e^{hr_1}ae^{-hd(q,x)}\le \tilde be^{-hn}.   
\end{equation}
From definition we have $\lambda_{\xi,\eta}(\tilde\alpha)\le \text{diam} \tilde\alpha$. Recall that $\text{diam} \tilde\alpha\le \e$. We have
$$m_\C (\tilde\alpha)\le \bar \mu (P(\tilde\alpha))\text{diam} \tilde\alpha$$
and the same holds for $(\phi^t)_*m_\C$. Since $\tilde \mu$ is a limit of a convex combination of such measures, we also have
\begin{equation}\label{diam}
\tilde \mu (\tilde\alpha)\le \bar \mu (P(\tilde\alpha))\text{diam} \tilde\alpha.
\end{equation}
At last, notice that for any $(\xi,\eta)\in P(\tilde \a)$, $\beta_p(\xi,\eta)\le 2r_1$ since any geodesic connecting $\xi$ and $\eta$ intersects $B(p,r_1)$. Combining this with \eqref{proj}, \eqref{diam} and the fact that $\mu_p( pr_x((B(p,r_1)))\le \mu_p(\pX)$, we get $\tilde \mu(\tilde\alpha)\leq e^{-hn}a$ for some constant $a>0$.

\textbf{Step 3.} We have $$H_{\mu}(\mathcal{A}^{(n)}_{\phi})=-\sum_{\alpha \in \mathcal{A}^{(n)}_{\phi}}\mu(\alpha)\log\mu(\alpha)\geq n\cdot h-\log a,$$
which implies that $$h_{\mu}(\phi^{1})\geq h_{\mu}(\phi^1,\mathcal{A})=\lim_{n\rightarrow \infty}\frac{1}{n}H_{\mu}(\mathcal{A}^{(n)}_{\phi})\geq h.$$
Then we have $h_{\mu}(\phi^{1}) = h$. 
\end{proof}

\begin{proposition}\label{BMMME}
Let $X$ be a simply connected uniform visibility manifold which has a compact quotient  $M$. Suppose that the geodesic flow on $SM$ has a unique MME giving full weight to $\R_0$. Then for any $A\subset SX$,
\begin{equation}\label{BM}
\tilde \mu (A)=\int_{\partial^2X}\lambda_{\xi,\eta}(A)d\bar \mu(\xi,\eta)
\end{equation}
where $\lambda_{\xi,\eta}$ is the Lebesgue measure on $P^{-1}(\xi,\eta)$ which is a single geodesic for $\bar \mu \ae (\xi,\eta)$.
\end{proposition}
\begin{proof}
 We claim that $P^{-1}(\xi,\eta)$ is a single geodesic for $\bar \mu \ae (\xi,\eta)$. Indeed, if otherwise we must have $\tilde \mu(\R_0^c)>0$, a contradiction. Then the right side of \eqref{BM} immediately gives an invariant measure and indeed it coincides with $\tilde \mu$.
\end{proof}

\begin{proposition} 
Let $X$ be a simply connected uniform visibility manifold which has a compact quotient  $M$. Suppose that the geodesic flow on $SM$ has a unique MME giving full weight to $\R_0$. Then the MME is fully supported on $SM$.
\end{proposition}
\begin{proof}
 It is enough to prove that $\tilde \mu(U)>0$ for any open subset $U\subset SX$. Recall that the Patterson-Sullivan measures have full support on $\pX$ and thus $\bar \mu(P(U))>0$. Then by \eqref{BM}, $\tilde \mu(U)>0$. The proposition follows.
\end{proof}
The above proposition proves the first statement of Theorem \ref{bernoulli}.

\begin{remark}
Under the assumption of Proposition \ref{BMMME}, we see that $\mu$ and $m_\C$ defined before Remark \ref{noncom} represents the same measure. In the following, we use $m_\C$ or just $m$ to denote this measure, i.e., the Bowen-Margulis measure and the unique MME, since the notation $\mu$ usually denotes an abstract measure. 
\end{remark}

\subsection{Bernoulli property of Bowen-Margulis measure}
We prove Theorem \ref{bernoulli} in this subsection. To start, we recall some necessary definitions.
\begin{definition}[Kolmogorov property]
Let $T: Y\to Y$ be an invertible measure-preserving transformation in a Lebesgue space $(Y,\mathcal{B},\mu)$. We say that $T$ has the \emph{Kolmogorov property}, or that $T$ is \emph{Kolmogorov}, if there is a sub-$\sigma$-algebra $\mathcal{K}$ of $\mathcal{B}$ which satisfies
\begin{enumerate}
  \item $\mathcal{K}\subset T\mathcal{K}$,
  \item  $\bigvee_{i=0}^{\infty}T^{i}\mathcal{K}=\mathcal{B}$,
  \item $\bigcap_{i=0}^{\infty}T^{-i}\mathcal{K}=\{\emptyset,Y\}$.
\end{enumerate}
\end{definition}

Consider $k\in \NN$ and a probability vector $p=(p_0,\cdots, p_{k-1})$, i.e., $p_i\ge 0, i=0,\cdots, k-1$ and $\sum_{i=0}^{k-1}p_i=1$. A \emph{Bernoulli shift} is a transformation $\sigma$ on the space $\Sigma_k:=\{0,\cdots,k-1\}^{\ZZ}$ given by
$\sigma((x_n)_{n\in \ZZ})=(x_{n+1})_{n\in \ZZ}$ which preserves the Bernoulli measure $p^{\ZZ}$ on $\Sigma_k$.

\begin{definition}[Bernoulli property]
Let $T: Y\to Y$ be an invertible measure-preserving transformation in a Lebesgue space $(Y,\mathcal{B},\mu)$. We say that $T$ has the \emph{Bernoulli property} or is called \emph{Bernoulli} if it is measurably isomorphic to a Bernoulli shift.
\end{definition}
Obviously, the Bernoulli property implies the Kolmogorov property.
\begin{definition}
A measure-preserving flow $(Y, (f^t)_{t\in \RR}, \mu)$ is said to have the \textit{Kolmogorov (resp. Bernoulli) property} if for every $t\neq 0$, the invertible measure-preserving transformation $(Y, f^t, \mu)$ has the Kolmogorov (resp. Bernoulli) property.
\end{definition}

For the the geodesic flow on manifolds of nonpositive curvature, Call and Thompson \cite{CT} showed that the unique equilibrium state for certain potential has the Kolmogorov property and the unique MME has the Bernoulli property, based on the earlier work of \cite{Led}. These two results are extended to manifolds without focal points in \cite[Theorem B]{CKP2} and \cite[Theorem C]{Wu2} respectively, via methods developed in \cite{CT}. Recently, Araujo-Lima-Poletti \cite{ALP} shows that the unique equilibrium state for certain potential has the Bernoulli property by symbolic approach.

\subsubsection{Kolmogorov property}
For compact uniform visibility manifolds without conjugate points, we first show the following.
\begin{proposition}\label{Kol}
Let $(M,g)$ be a closed $C^\infty$ visibility manifold without conjugate points. Suppose that 
$$\sup\{h_\mu(\phi^1): \mu \text{\ is supported on\ } \R_0^c\}<h_{\top}(\phi^1).$$
Then $m$ (i.e., the unique MME) is Kolmogorov.
\end{proposition}

We first recall Ledrappier's criterion for Kolmogorov property of equilibrium states \cite{Led}.
\begin{definition}
Let $f:Y\to Y$ be a continuous self-map on a compact metric space $(Y,d)$. For $\varepsilon>0$ and $x\in Y$, denote
$$\Gamma_{\varepsilon}(x):=\{y\in Y: d(f^{n}(x),f^{n}(y))<\varepsilon, \forall n\in\mathbb{Z}\},$$
and
$$h_{\text{loc}}(f,\varepsilon):=\sup\{h_\top(f,\Gamma_{\varepsilon}(x))):\ x\in Y\}.$$
 We say that $(Y,f)$ is \textit{entropy expansive} if there is $\varepsilon_{0}>0$ such that
$$h_{\text{loc}}(f,\varepsilon_{0})=0,$$
and $(Y,f)$ is \textit{asymptotically entropy expansive} if
\begin{center}
   $\lim_{\varepsilon\to0}h_{\text{loc}}(f,\varepsilon)=0$.
\end{center}
We say that $\nu\in \M_f(Y)$ is \textit{almost entropy expansive}  at scale $\varepsilon_{0}>0$ if 
$$h_\top(f,\Gamma_{\varepsilon}(x)))=0$$
for $\nu \ae x\in Y$.
\end{definition}

\begin{theorem}[Ledrappier \cite{Led}, Call-Thompson \cite{CT}]\label{Led}
Let $(Y, f^t)$ be a continuous flow on a compact metric space such that $f^t$ is asymptotically entropy expansive for all $t\neq 0$, and let $\psi\in C(Y,\RR$). Consider $(Y\times Y, f^t \times f^t)$ which is the product of two copies of $(Y, f^t)$, and $\Psi(x_1,x_2):=\psi(x_1)+\psi(x_2)$. If $\Psi$ has a unique equilibrium state in $\mathcal{M}_{f^t\times f^t}(Y\times Y)$, then the unique equilibrium state for $\psi$ in $\mathcal{M}_{f^t}(Y)$ has the Kolmogorov property.
\end{theorem}

We state again Theorem \ref{expansivefactor} here, since we will consider the Kolmogorov property of the expansive factor of the geodesic flow.
\begin{theorem}(Mamani-Ruggiero \cite[Theorem 1.1]{MR})
Let $M$ be a compact $C^\infty$ manifold of dimension $n$ without conjugate points and with uniform visibility universal cover $X$. Then the geodesic flow is time-preserving semi-conjugate to a continuous expansive flow $\psi^t$ acting on a compact metric space $Y:=SM/\sim$ of topological dimension at least $n-1$. Moreover, $\psi^t$ is topologically mixing and has a local product structure.
\end{theorem}

Denote $P(t,\e)$ the set of periodic orbits $\O$ of $(Y, \psi^t)$ with period in $[T-\e, T+\e]$. Define the measure
$$\nu_t:=\frac{1}{\#P(t,\e)}\sum_{\O\in P(t, \e)}\d_\O$$
where $\d_\O$ is an invariant probability measure supported on the orbit $\O$. If $\nu=\lim_{t\to \infty} \nu_t$, we say that $\nu$ is the limit of the invariant measures supported by periodic orbits of $\psi^t$. 

\begin{lemma}
Let $\nu$ be the unique MME of $(Y,\psi^t)$. Then $\nu\times \nu$ is the unique MME of $(Y\times Y,\psi^t\times \psi^t)$. Moreover, $\nu\times \nu$ is the limit of the invariant measures supported by periodic orbits of $\psi^t\times \psi^t$. 
\end{lemma}
\begin{proof}
Since  $(Y,\psi^t)$ has the specification property, by \cite[Lemma 3.1]{CT}, $(Y\times Y,\psi^t\times \psi^t)$ has the specification property as well. However, $(Y\times Y,\psi^t\times \psi^t)$ is not necessarily expansive even if $(Y,\psi^t)$
is expansive. We overcome this difficulty following ideas of \cite{CT}. 

Denote $\C_\e(x):=\{y\in Y: d(\psi^tx, \psi^ty)<\e, \forall t\in \RR\}$ and 
$$\C_\e(x,y):=\{(x',y')\in Y\times Y: d(\psi^tx, \psi^ty), d(\psi^tx', \psi^ty')<\e, \forall t\in \RR\}.$$
Note that $\C_\e(x,y)=\C_\e(x)\times \C_\e(y)$. The set of \textit{non-expansive points} at scale $\e>0$ is defined as
\[\text{NE}(\e):=\{x\in Y: \C_\e(x)\nsubseteq \psi_{[-s,s]}(x), \forall s>0\}.\]
According to \cite[Definition 4.2]{CT}, the set of \textit{product non-expansive points} at scale $\e>0$ is denoted by
\[\text{NE}^{\times}(\e):=\{(x,y)\in Y\times Y: \C_\e(x,y)\nsubseteq \psi_{[-s,s]}(x)\times \psi_{[-s,s]}(y), \forall s>0\}.\]
By \cite[Lemma 4.4]{CT}, 
\[\text{NE}^{\times}(\e)\subset (Y\times \text{NE}(\e))\cup (\text{NE}(\e)\times Y).\]

Since $\psi^t$ is expansive, $\text{NE}(\e)=\emptyset$ and hence $\text{NE}^{\times}(\e)=\emptyset$. Thus any $\mu\in \mathcal{M}_{\psi^t\times \psi^t}(Y\times Y)$ is product expansive at any scale $\e$, in the sense that $\mu(\text{NE}^{\times}(\e))=0$.
It follows from \cite[Proposition 4.7]{CT} that every subset $Q\subset Y\times Y$ invariant under the $\RR^2$-action (i.e., $\{\psi^s\times  \psi^t): s,t\in \RR\}$) can be approximated by elements of adapted partitions for $(t,\e)$-separated sets.

Now we recall some results obtained from Climenhaga-Thompson's proof \cite{CT1} of uniqueness of MME for $\psi^t\times \psi^t$. In fact, even Bowen's original proof (cf. \cite{Bo3} or \cite[Theorem 20.1.3]{KH}) is applicable to our setting. We can obtain similarly as \cite[Theorem 6.3 and Lemma 6.4]{CT}:
\begin{itemize}
    \item $\nu$ is weakly mixing, which is equivalent to that $\nu\times \nu$ is ergodic;
    \item $\nu\times \nu$ has the Gibbs property, i.e. there exists $Q>0$ such that for any $(x,y)\in Y\times Y$ and $t>0$,
    \[(\nu\times \nu)(B_t((x,y),\rho))\ge Q^2e^{-2th}\]
    where $B_t((x,y),\rho)$ is the Bowen ball and $h$ is the topological entropy of $\psi^t$ (and hence of $\phi^t$).
\end{itemize}

Combining these facts, we can adapt the classical arguments to subset $Q\subset Y\times Y$ invariant under the $\RR^2$-action. The proof is a verbatim of \cite[page 818]{CT}. Moreover, $\nu\times \nu$ is the limit of the invariant measures supported by periodic orbits of $\psi^t\times \psi^t$ (cf. \cite{Bo3} or \cite[Proposition 6.4]{BCFT}). 
\end{proof}

\begin{lemma}\label{MMEextension}(\cite[Theorem 1.5]{BFSV})
Let $f^t: Y\to Y$ and $g^t: Z\to Z$ be two continuous flows on compact metric spaces, $\tilde\chi: Y\to Z$ be a time-preserving semiconjugacy and $\nu$ be the unique MME of $g^t$. Assume that $g^t$ is expansive, has the specification property and
\begin{enumerate}
    \item $h(f^1, \tilde\chi^{-1}(x))=0$ for every $x\in Z$;
    \item $\nu(\{\tilde\chi(y): \tilde\chi^{-1}(\tilde\chi(y))=\{y\}\})=1$.
\end{enumerate}
Then $f^t$ has a unique MME.
\end{lemma}

\begin{proof}[Proof of Proposition \ref{Kol}]
We apply Lemma \ref{MMEextension} for $f^t=\phi^t\times \phi^t$ and $g^t=\psi^t\times \psi^t$ and $\tilde\chi(x,y)=(\chi(x), \chi(y))$ where $\chi:SM\to SM/\sim$ is the projection. All the conditions of Lemma \ref{MMEextension} can be verified except that $\psi^t\times \psi^t$ is expansive. However, if we look at the proof of \cite[Theorem 1.5]{BFSV}, we see that we do not need the expansive condition once we know that $\nu\times \nu$ is the limit of the invariant measures supported by periodic orbits of $\psi^t\times \psi^t$. 

By \cite{Buzzi}, any $C^\infty$ dynamical system is asymptoically $h$-expansive. So the geodesic flow $\phi^t$ is asymptoically $h$-expansive. By Lerappier's criterion Theorem \ref{Led}, we see that $m$ has the Kolmogorov property.
\end{proof}

\subsubsection{From Kolmogorov to Bernoulli}
A classical argument in \cite{OW1} showed that the Kolmogorov property implies Bernoulli property for smooth invariant measures of Anosov flows. The argument was also carried out by Chernov and Haskell \cite{CH} for smooth invariant measures of suspension flows over some nonuniformly hyperbolic maps with singularities. It also works for hyperbolic invariant measures with local product structure. See also \cite{PV, PTV} for an introduction to terminology and ideas of the Kolmogorov-Bernoulli equivalence.

Based on the argument in Chernov and Haskell \cite{CH}, Call and Thompson \cite[Section $7$]{CT} proved that the Knieper measure is Bernoulli when $M$ is a rank one compact manifold of nonpositive curvature. Following \cite{CH, CT}, we show that the Bowen-Margulis measure $m$ is Bernoulli for compact uniform visibility manifolds without conjugate points in the setting of Theorem \ref{mmeuni}. 

Let $f: Y\to Y$ be an invertible measure-preserving transformation in a Lebesgue space $(Y,\mathcal{B},\mu)$. If $A\subseteq Y$ is a measurable subset, then by the measure $\mu$ on $A$ we mean the conditional measure $\mu|_A(B):=\frac{\mu(A\cap B)}{\mu(A)}$. By $\a|_A$ we mean the induced partition on the space $(A,\mu|_A)$, i.e., $\a|_A:=\{A\cap A_1, \cdots, A\cap A_k\}.$

\begin{definition}
Let $\a$ be a finite measurable partition of $Y$ and $\varepsilon>0$. We say that a property holds for $\varepsilon$-almost every element of $\a$ if the union of those elements for which the property fails has $\mu$-measure at most $\varepsilon$.
\end{definition}

\begin{definition}
Let $f: Y\to Y$ be an invertible measure-preserving transformation in a Lebesgue space $(Y,\mathcal{B},\mu)$. A finite partition $\a$ is called a \emph{very weak Bernoulli partition} (VWB partition) if for any $\varepsilon>0$, there exists $N_0\in \NN$ such that for all $N'\ge N\ge N_0$,  $S\ge 1$, and $\varepsilon$-almost every element $A$ of $\vee_N^{N'}f^k\a$ we have
$$\bar d((f^{-i}\a)_{i=0}^{S-1},(f^{-i}\a|_A)_{i=0}^{S-1})\le \varepsilon$$
where $\bar d$ is a distance between two sequences of partitions (cf. \cite[Definition 3.4]{PV}).
\end{definition}
\begin{theorem}[Cf. \cite{Or1, Or2}]\label{appro}
Let $f: Y\to Y$ be an invertible measure-preserving transformation in a Lebesgue space $(Y,\mathcal{B},\mu)$. If there exists a sequence of VWB partitions $\a_k$ converging to the partition into points of $Y$, then $f$ is Bernoulli.
\end{theorem}

We continue to recall the classical method in \cite{OW1} and \cite{CH} to establish VWB property for systems with hyperbolic behaviour. Given a finite partition $\a=\{A_1,\cdots, A_k\}$, the $\a$-name of $x\in Y$ is the sequence $(x_i^\a)$ determined by $f^i(x)\in A_{x_i^\a}$. A measurable map $\theta:(Y,\mu)\to (Z,\nu)$ is called \emph{$\ve$-measure preserving} if there exists a subset $E\subset Y$ with $\mu(E)<\ve$ such that for every $A\subset Y\setminus E$, one has
$$\left|\frac{\nu(\theta(A))}{\mu(A)}-1\right|<\ve.$$

\begin{lemma}[Cf. \cite{OW1}]\label{function}
Let $f: Y\to Y$ be an invertible measure-preserving transformation in a Lebesgue space $(Y,\mathcal{B},\mu)$ and $\a$ a finite measurable partition of $Y$. If for any $\varepsilon>0$, there exists $N\in \NN$ such that for all $N'\ge N$,  $S\ge 1$, and $\varepsilon$-almost every atom $A$ of $\vee_N^{N'}f^k\a$ one can find a map $\theta:(A,\mu|_A)\to (Y,\mu)$ such that
\begin{enumerate}
  \item there is a subset $E'\subset A$ with $\mu|_A(E')<\ve$ such that for any $x\in A\setminus E'$,
  $$h(x,\theta(x)):=\frac{1}{S}\sum_{0\le i\le S-1: x_i^\a\neq (\theta(x))_i^\a}1<\ve,$$
  \item $\theta$ is $\ve$-measure preserving,
\end{enumerate}
then $\a$ is a VWB partition.
\end{lemma}

Now let $m$ be the Bowen-Margulis measure, which coincides with the the unique MME for the geodesic flow in Theorem \ref{mmeuni}. We will lift the Kolmogorov property of $m$ to the Bernoulli property, and thus finish the proof of Theorem \ref{bernoulli}. 
Since $(Y, \psi^t)$ has local product structure, we will prove that $\nu$ has the Bernoulli property. Since $m(\R_0)=1$, we see that $\chi:(SM,m)\to (Y,\nu)$ is a measure-preserving isomorphism. Then we can conclude that $m$ has the Bernoulli property. On the other hand, we will make use of the product structure of $m$ as in \eqref{BM} in the proof.

Consider a partition $\a=\{A_1,\cdots, A_k\}$ into subsets of $Y$ with arbitrarily small diameter and the following property: there exists some $D>0$ such that $\nu(\partial_\varepsilon A_i)\le D\varepsilon, i=1,\cdots, k$ for any $\varepsilon>0$ where $\partial_\varepsilon A_i:=\{y\in M: d(y,\partial A_i)<\varepsilon\}$. The existence of such a partition with respect to $m$ on compact manifold $SM$ is guaranteed by \cite[Lemma 4.1 and (4.3)]{OW1}. Then we project this partition onto $Y$ to get partition $\a$ we need. It is enough to show that any above partition $\a$ is VWB. Therefore, the task is to construct the function $\theta$ in Lemma \ref{function} for such $\a$. To do it, we localize the problem and construct $\theta$ in small rectangles defined as follows.

\subsubsection{$\ve$-regular coverings}

Since $(Y, \psi^t)$ has local product structure, we construct rectangles first in $Y$, and then lift them to $SM$. We note that even near expansive vectors, there is no well defined local product structure on $SM$. One may assume the continuity of Green bundles so that there is a local product structure near a vector in the open set $\R_1$. Here we do not assume this condition, but use the local product structure of $Y$ and local product structure of $m$ on $SM$ provided by \eqref{BM}.

For every $v\in SM$, we define $\F^{0s/0u}(v):=\cup_{t\in \RR}\phi^t(\F^{s/u}(v))$ and
\begin{equation}\label{foliation}
    V^{*}([v]):=\chi(\F^*(v)), *=s,u,0s,0u.
\end{equation}
For $[v],[w]$ close enough, define $[[v], [w]]$ to be the single point in 
$$V_\loc^{u}([v])\cap V_\loc^{0s}([w])=\chi(\F_\loc^{u}(v)\cap \F_\loc^{0s}(w)).$$
\begin{definition}
A \emph{rectangle} in $Y$ is a measurable set $R\subset Y$, equipped with a distinguished point $z \in R$ with the property that for all points $x, y \in R$ the local weak stable manifold $V_\loc^{0s}(x)$ and the local unstable $V_\loc^{u}(y)$ intersect each other at a single point in $R$, denoted by $[x,y]$.
\end{definition}
Recall that $m(\R_0)=\nu(\chi(\R_0))=1$. Take any $v_0\in \R_0$, and let 
$$R:=[V_\e^{0s}([v_0]),V_\e^{u}([v_0])]$$
for some small $\e>0$. Then $R$ is a rectangle in $Y$ centered at $[v_0]$ since $Y$ has local product structure. 
Since $v_0\in \R_0$, we can take  $\e>0$ small enough, so that $\tilde R= \chi^{-1}(R)$ has small enough diameter.

Notice that a rectangle $R$ centered at $z\in Y$ can be thought of as the Cartesian product of $V_\loc^{0s}(z)\cap R$ and $V_\loc^{u}(z)\cap R$, where a point $y\in R$ is given by $[V_\loc^{u}(y) \cap V_\loc^{0s}(z), V_\loc^{0s}(y)\cap V_\loc^{u}(z)]$.

Given a probability measure $\nu$ on $Y$, there is an associated natural product measure
$$\nu^p_R:=\nu^u_z\times \tilde\nu^{0s}_z,$$
where $\nu^u_z$ is the conditional measure induced by $\nu$ on $V_\loc^u(z)\cap R$ with respect to the measurable partition of $R$ into local unstable manifolds, and $\tilde\nu^{0s}_z$ is the corresponding factor measure on $V^{0s}_\loc(z)$.

Below is Chernov and Haskell's definition of $\e$-regular coverings.
\begin{definition}\label{coveringdef}
Given any $\e> 0$, we define an \emph{$\e$-regular covering} for a probability measure $\nu$ of $Y$ to be a finite collection of disjoint rectangles $\mathcal{R}=\mathcal{R}_\e$ such that
\begin{enumerate}
  \item $\nu(\cup_{R\in \mathcal{R}}R) > 1-\e$;
  \item For every $R\in \mathcal{R}$ with a distinguished point $z\in R$, the product measure $\nu^p_R:=\nu^u_z\times \tilde\nu^{0s}_z$ satisfies
$$\left|\frac{\nu^p_R(R)}{\nu(R)}-1\right| <\e.$$
Moreover, $R$ contains a subset $G$ with $\nu(G) > (1-\e)\nu(R)$ such that for all $x\in G$,
$$\left|\frac{d\nu^p_R}{d\nu}(x)-1\right| <\e.$$
\end{enumerate}
\end{definition}

\begin{lemma}\label{covering}
For any $\d>0$ and $\e>0$, there exists an $\e$-regular covering of connected rectangles  $\mathcal{R}_{\e}$ for the measure $\nu$ of $Y$, with $\diam(R)<\d$ for any $R\in \mathcal{R}_{\e}$.
\end{lemma}

\begin{proof}
Since $(Y,\psi^t)$ has local product structure, we can find a finite
collection of disjoint rectangles $R$ constructed as above, covering a subset of $\nu$-measure at least $1-\e$ and satisfying the first condition in Definition \ref{coveringdef}. Moreover, the rectangles $R$ can be chosen such that $\diam R<\d$. Let $\tilde R=\chi^{-1}(R)$. Since $R$ is centered at expansive vectors, the diameter of $\tilde R$ is also small. 

To verify the second condition in Definition \ref{coveringdef}, we use the local product structure of $m$ provided by \eqref{BM}. 
We lift all objects to the universal cover $X$. Take a lift of $\tilde R$, which is still denoted by $\tilde R$. Since $\tilde R$ is centered at some $v_0\in \R_0$, the local weak horospherical stable
and local horospherical unstable manifolds of $v_0$ intersect at a single point $v_0$. For any $x\in \tilde R$, there exists a continuous map $\varphi_x: \F_\loc^u(x)\to \pX$ given by $\varphi_x(v)=v^+$ where $v\in \F_\loc^u(x)$. By \eqref{BM}, the conditional measure $\mu^u_x$ on $\tilde R\cap \F_\loc^u(x)$ of the Bowen-Margulis measure $m$ is given by
\begin{equation}\label{conditional}
d\mu^u_x(v) = \frac{e^{h\b(x^-,\varphi_xv)} d\mu_p(\varphi_xv)}{\int_{\tilde R\cap \F_\loc^u(x)}e^{h\b(x^-,\varphi_xv)} d\mu_p(\varphi_xv)}
\end{equation}
and $\int d m(v)=c\int d\mu^u_x(v) d\mu_p(x^-)dt$ for some normalization constant $c>0$.

Given two points $x, y\in \tilde R$, the local weak stable holonomy map $\pi^{0s}_{xy}: \F_\loc^{u}(x)\cap \tilde R \to \F_\loc^{u}(y)\cap \tilde R$ is defined by
$$\pi_{xy}^{0s}(w):=[w,y]=\F_\loc^{u}(y)\cap \F_\loc^{0s}(w)\cap \tilde R, \quad w\in \F_\loc^{u}(x)\cap \tilde R.$$
Note that $\varphi_x(w)=\varphi_{y}(\pi^{0s}_{xy}w):=\eta_w$. By \eqref{conditional}, the Jacobian of the holonomy map \begin{equation*}
\begin{aligned}
\left|\frac{d(\pi_{yx}^{0s})_*\mu^u_y}{d\mu^u_x}(w)\right|=\left|\frac{e^{h\b(y^-,\eta_w)}}{e^{h\b(x^-,\eta_w)}}\cdot\frac{\int_{\tilde R\cap \F_\loc^u(x)}e^{h\b(x^-,\eta_w)} d\mu_p(\eta_w)}{\int_{\tilde R\cap \F_\loc^u(x)}e^{h\b(y^-,\eta_w)} d\mu_p(\eta_w)}\right|.
\end{aligned}
\end{equation*}
By Corollary \ref{continuous}, $e^{h\b}$ is uniformly continuous and bounded away from $0$. By taking $\d$ small enough so that $\text{diam\ }\tilde R$ is small, we have that
$$\left|\frac{d(\pi_{yx}^{0s})_*\mu^u_y}{d\mu^u_x}(w)-1\right| <\e$$
for $x,y\in \tilde R$.
By definition of conditional measures, this implies the second condition in Definition \ref{coveringdef} for $m$.

Since $\chi: (\tilde R, m)\to (R,\nu)$ is an isomorphism, and $\chi$ preserves foliations in the sense of \eqref{foliation}, the second condition in Definition \ref{coveringdef} for $\nu$ also holds.
\end{proof}

\subsubsection{Construction of the function $\theta$}
To conclude of the proof of Theorem \ref{bernoulli}, we need construct the function $\theta$ in Lemma \ref{function}, for any partition $\a=\{A_1,\cdots, A_k\}$ into subsets of $X$ with $\nu(\partial_\varepsilon A_i)\le D\varepsilon, i=1,\cdots, k$ for some $D>0$ and any $\varepsilon>0$. The existence of $\ve$-coverings by Lemma \ref{covering} allows us to construct the function $\theta$ by first constructing it in a small rectangle.

\begin{lemma}
For any $\d>0$, there exists $0<\d_1<\d$ such that if $R$ is a rectangle with $\diam R<\d_1$ and $E$ is a measurable set intersecting $R$ leafwise, i.e., $V^u_\loc(x) \cap R \subset E\cap R$ for any $x\in E\cap R$, we can construct a bijective function $\theta:E\cap R\to R$ such that for any measurable set $F\subset E\cap R$,
\begin{equation}\label{e:match}
\frac{\nu^p_R(\theta(F))}{\nu^p_R(R)}=\frac{\nu^p_R(F)}{\nu^p_R(E\cap R)}
\end{equation}
and $\theta(x)\in V^{cs}_\loc(x)$ for every $x\in E\cap R$.
\end{lemma}
\begin{proof}
We remark that according to the disintegration 
$$\int d m(v)=c\int d\mu^u_x(v) d\mu_p(x^-)dt,$$ 
the measure $d\mu_p(x^-)dt$ has no atom, and the same holds for $\nu$. Then the construction of $\theta$ in the lemma is to match the points in two local unstable sets by weak stable holonomy, see \cite[Lemma 5.10 ]{Po} or \cite[Lemma 4.9]{PTV} for the smooth measure. 
\end{proof}

\begin{proof}[Proof of Theorem \ref{bernoulli}]
Based on \eqref{e:match} and the definition of $\ve$-regular covering, we can prove item (2) in Lemma \ref{function}.  

Since $\theta(x)\in V^{0s}_\loc(x)$ for every $x\in E\cap R$, we have by \cite[Proposition 5.1]{MR} that for any $\e>0$, if we take $\d>0$ at the beginning of the construction of $R$ small enough, then 
$$d(\psi^t(x), \psi^t(\theta(x)))\le \e$$
for any $t>0$.
Item (1) in Lemma \ref{function} follows from the above fact
and condition on the small boundary of $\a$ with respect to $\nu$. 

Since the diameter of $\a$ can be arbitrarily small, by Theorem \ref{appro}, $\nu$ is Bernoulli. Since  $\chi:(SM,m)\to (Y,\nu)$ is a measure-preserving isomorphism, $m$ is Bernoulli, which completes the proof of Theorem \ref{bernoulli}.
\end{proof}

\section{Uniqueness of Equilibrium states}
In this section, we prove Theorem \ref{uniqueness} using the symbolic approach developed in Lima-Poletti \cite{LP}. We first prepare the general theorem on coding of nonuniformly hyperbolic flows. Then we apply the nonuniform hyperbolicity of the geodesic flow on uniform visibility manifolds without conjugate points, and with continuous Green bundles.

\subsection{Coding of nonuniformly hyperbolic flows}
Let $\G = (V, E)$ be an oriented graph, where $V, E$ are the vertex and edge sets. Denote edges by $v\to w$, and assume that $V$ is countable. Suppose that $\G$ is locally compact, namely,  for all $v\in V$, the number of ingoing edges $u\to v$ and outgoing edges $v \to w$ is finite. 

A \textit{topological Markov shift} (TMS for short) is a pair $(\Sigma, \sigma)$ where
\[\Sigma:=\{\lv=\{v_n\}_{n\in \ZZ}\in V^{\ZZ}: v_n\to v_{n+1}, \forall n\in \ZZ\} \]	
is the symbolic space endowed with the metric
$$d(\lv,\lw):=\exp(-\inf\{|n|\in \ZZ: v_n\neq w_n\}),$$ and $\sigma: \Sigma \to \Sigma$ is the left shift defined by $[\sigma(\lv)]_n=v_{n+1}$. The \textit{regular set} of $\Sigma$ is defined as
\begin{equation*}
    \begin{aligned}
\Sigma^{\#}:=\{\lv\in \Sigma: \exists u,w\in V \text{\ such that\ } &v_n=u \text{\ for infinitely many\ } n>0,\\
\text{\ and\ } &v_n=w \text{\ for infinitely many\ } n<0\}.        
    \end{aligned}
\end{equation*}

Let $(\Sigma, \sigma)$ be a TMS and $r: \Sigma \to  (0, +\infty)$ be a continuous function. For $n\ge 0$, the $n$-th Birkhoff sum of $r$ is
\[r_n:= r + r\circ \sigma + \cdots + r\circ \sigma^{n-1}.\]
The definition can be extended for $n<0$ in a unique way such that the following cocycle identity holds:
\[r_{m+n} = r_m + r_n \circ \sigma^m, \quad \forall m,n\in \ZZ.\]

\begin{definition}
The \textit{topological Markov flow} (TMF for short) is the pair $(\Sigma_r, \sigma_r)$ where 
$$\Sigma_r:=\{(\lv, t): \lv\in \Sigma, 0\le t<r(\lv)\}$$
and $\sigma_r: \Sigma_r \to \Sigma_r$ is a flow on $\Sigma_r$ defined by 
$$\sigma_r^t(\lv, t')=(\sigma^n(\lv), t'+t-r_n(\lv))$$
where $n$ is the unique integer such that $r_n(\lv)\le t'+t< r_{n+1}(\lv)$.  
\end{definition}
 $\Sigma_r$ can be endowed with a natural metric $d_r(\cdot, \cdot)$, the so-called Bowen-Walters metric, so that $\sigma_r$ is a continuous flow. The \textit{regular set} of $(\Sigma_r, \sigma_r)$ is $\Sigma_r^{\#}:=\{(\lv, t)\in \Sigma_r: \lv\in \Sigma^\#\}$.

If $\Sigma$ is a TMS defined by an oriented graph $\G = (V, E)$, its \textit{irreducible components} are the subshifts $\Sigma'\subset \Sigma$ defined over maximal subsets $V'\subset V$ satisfying the following condition:
$$\forall u, w\in V', \exists \lv\in \Sigma \text{\ and\ } n\ge 1 \text{\ such that\ } v_0 =u \text{\ and\ } v_n=w.$$
An \textit{irreducible component} $\Sigma'_r$ of  $\Sigma_r$ is a set of the form 
$$\Sigma'_r=\{(\lv, t)\in \Sigma_r: \lv \in \Sigma'\}$$
where $\Sigma'$ is an irreducible component of $\Sigma$.

Now we consider a $C^{1+\b}(\b>0)$-vector field $Z$ with $Z\neq 0$ everywhere on a compact smooth manifold $N$. Let $\varphi^t: N\to N$
be the flow generated by $Z$. We say that $\mu\in \M_{\varphi^t}(N)$ is \textit{$\lambda$-hyperbolic} ($\lambda>0$), if for $\mu \ae x\in N$, all the Lyapunov exponents are greater than $\lambda$ in absolute value, except for the zero exponent in the flow direction. By Pesin theory, for $\mu \ae x\in N$, there exists \textit{stable/unstable manifold} through $x$ 
\[W^{s/u}(x):=\{y\in N: \limsup_{t\to \pm\infty}\frac{1}{t}\log d(\varphi^t(x), \varphi^t(y))<0\}.\]
The \textit{weak stable/unstable manifold} through $x$ is then defined as
$$W^{0s/0u}(x):=\cup_{t\in \RR}\varphi^t(W^{s/u}(x)).$$

\begin{definition}
We say that two ergodic hyperbolic measures $\mu, \nu$ are \textit{homoclinically related} if for $\mu \ae x$ and $\nu \ae y$ there exist transverse intersections $W^{0s}(x)\pitchfork W^{0u}(y)\neq \emptyset$  and $W^{0u}(x)\pitchfork W^{0s}(y)\neq \emptyset$, i.e., points $z_1\in W^{0s}(x)\cap W^{0u}(y)$ and $z_2\in W^{0u}(x)\cap W^{0s}(y)$ such that $T_{z_1}N=T_{z_1}W^{0s}(x) +T_{z_2}W^{0u}(y)$ and $T_{z_2}N=T_{z_1}W^{0u}(x) + T_{z_2}W^{0s}(y)$.
\end{definition}

\begin{theorem}\label{coding}(\cite[Theorem 2.1]{LP})
Let $\varphi^t: N\to N$ be a flow generated by a $C^{1+\b}$-vector field $Z$ with $Z\neq 0$ everywhere, and $\lambda>0$. If $\mu_1, \mu_2$ are homoclinically related $\lambda$-hyperbolic ergodic measures, then there is an irreducible TMF $(\Sigma_r, \sigma_r)$ and a H\"{o}lder continuous map $\pi_r: \Sigma_r\to N$ such that:
\begin{enumerate}
    \item $r:\Sigma\to \RR^+$ is H\"{o}lder continuous and bounded away from zero and infinity.
\item $\pi_r\circ \sigma_r^t=\varphi^t\circ \pi_r$ for all $t\in \RR$.
\item $\pi_r[\Sigma_r^\#]$ has full measure with respect to $\mu_1$ and $\mu_2$.
\item Every $x\in N$ has finitely many preimages in $\Sigma_r^\#$.
\end{enumerate}
\end{theorem}

\subsection{Uniqueness of equilibrium states}
Let $(M,g)$ be a closed $C^\infty$ uniform visibility manifold without conjugate points and with continuous Green bundles, and let $\psi:SM\to \RR$ be H\"{o}lder continuous or of the form $\psi = q\psi^u$ for $q\in \RR$. 

Suppose that the geodesic flow has a hyperbolic periodic point $\O$. We let $W^{0s/0u}(\O) = W^{0s/0u}(x)$ denote the weak stable/unstable manifold of $\O$, for any $x\in \O$. The \textit{homoclinic class} of a hyperbolic periodic orbit $\O$ is the set
$$\text{HC}(\O) = \overline{W^{0u}(\O) \pitchfork W^{0s}(\O)}.$$

\begin{lemma}\label{foliation2}
Let $\F^{0s/0u}(\O):=\cup_{t\in \RR}\phi^t(\F^{s/u}(x))$ for any $x\in \O$. Then 
$$W^{0s/0u}(\O) = \F^{0s/0u}(\O).$$
\end{lemma}
\begin{proof}
We first show that $W^{0s/0u}(\O) \subset \F^{0s/0u}(\O)$. Recall that for hyperbolic orbit $\O$, 
\begin{equation}\label{stable}
W^{0s}(\O)=\cup_{t\ge 0}\phi^{-t}(W^{0s}_\e(\O))   
\end{equation}
where $$W^{0s}_\e(\O)=\{y\in SM: d(\phi^t(y), \O)<\e, \forall t>0\} $$
for any $\e>0$. Then for any $y\in W^{0s}_\e(\O)$, $y$ is asymptotic to $\O$, and therefore $y\in \F^{0s}(\O)$. So $W^{0s}_\e(\O)\subset \F^{0s}(\O)$. By the $\phi^t$-invariance of $\F^{0s}(\O)$, we have $W^{0s}(\O)\subset \F^{0s}(\O)$.

Note that both $W^{0s}(\O)$ and $\F^{0s}(\O)$ are connected smooth submanifold of dimension $\dim M$. To prove $W^{0s}(\O)= \F^{0s}(\O)$, it is enough to prove that $W^{0s}(\O)$ are both open and closed in the induced topology of $\F^{0s}(\O)$. Since we have proved that $W^{0s}_\e(\O)\subset \F^{0s}(\O)$ for any $\e>0$, it is easy to see from Proposition \ref{horofoliation} that $W^{0s}_\e(\O)$ is open in $\F^{0s}(\O)$. By \eqref{stable}, $W^{0s}(\O)$ is open in $\F^{0s}(\O)$. To prove that it is also closed in $\F^{0s}(\O)$, pick a sequence $z_k\in W^{0s}(\O)$ such that $z_k\to z$ as $k\to \infty$ for some $z\in \F^{0s}(\O)$. Then
\begin{equation*}
d(\phi^t(z), \O)\le d(\phi^t(z),\phi^t(z_k))+d(\phi^t(z_k), \phi^t(\O))\le Cd(z,z_k)+d(\phi^t(z_k), \O).
\end{equation*}
Choose $k$ large enough such that $d(z_k,z)<\e/2C$ and $t_0$ large enough such that $d(\phi^t(z_k), \O)\le \e/2$ for any $t\ge t_0$. Therefore,  $d(\phi^t(z), \O)\le \e$ for any $t\ge t_0$. We get $\phi^{t_0}(z)\subset W^{0s}_\e(\O)$ and thus $z\in W^{0s}(\O)$. This proves $W^{0s}(\O)= \F^{0s}(\O)$. 

The proof for $W^{0u}(\O)= \F^{0u}(\O)$ is analogous and we are done.
\end{proof}

By Proposition \ref{horofoliation}, $\F^s/\F^u$ are continuous minimal foliations of $SM$. Hence $W^{0s/0u}(\O) = \F^{0s/0u}(\O)$ are dense in $SM$.
\begin{lemma}\label{hco}
$\text{HC}(\O)=SM$.
\end{lemma}
\begin{proof}
Let $y\in \R_1$ and $U\subset \R_1$ an open neighborhood of $y$. Since $\R_1$ is open, we have $B_\e(y)\subset U\subset \R_1$ for small enough $\e>0$. By Lemma \ref{foliation2}, $W^{0s/0u}(\O) = \F^{0s/0u}(\O)$ are dense in $SM$. Hence there exist $z_1\in W^{0s}(\O)\cap B_\e(y)$ and $z_2\in W^{0u}(\O)\cap B_\e(y)$. 

By Proposition \ref{horofoliation}, $\F^{0s}(z_1)$ and $\F^{0u}(z_2)$  are submanifolds tangent to $G^s(z_1)\oplus  \langle Z(z_1) \rangle$ and $G^u(z_2)\oplus  \langle Z(z_2) \rangle$ respectively. Since $G^s$ and $G^u$ vary continuously, and $G^s(y)\oplus \langle Z(y) \rangle\oplus G^u(y)=T_ySM$, by choosing $\e$ small enough at the beginning, we see that $\F^{0s}(z_1)\pitchfork\F^{0u}(z_2)$, that is $W^{0s}(\O)\pitchfork W^{0u}(\O)$ at some point $z\in U$. It follows that $y\in \text{HC}(\O)$. Hence $\R_1\subset \text{HC}(\O)$. Since $\R_1$ is dense in $SM$ by \cite[Lemma 7.1]{MR}, the lemma follows.
\end{proof}

\begin{proof}[Proof of Theorem \ref{uniqueness}]
Let $\psi$ be H\"{o}lder continuous or of the form $\psi^t=q\psi^u$ with $q\in \RR$. Since the geodesic flow $\phi^t$ is $C^\infty$ and $\psi$ is continuous, the existence of an equilibrium state with respect to $\psi$ is
guaranteed by the classical result \cite{Buzzi, Newhouse}. 

For the uniqueness, assume that $\mu_1$ and $\mu_2$ are two ergodic equilibrium states with respect to $\psi$.
By assumption, $P(\R_1^c, \psi) < P(\psi)$, we have $\mu_1(\R_1)=\mu_2(\R_1)=1$. Then by Theorem \ref{regular}, $\mu_1$ and $\mu_2$ are hyperbolic. Since $\mu_1$ and $\mu_2$ are ergodic, we can take $\lambda>0$ small enough so
that $\mu_1$ and $\mu_2$ are both $\lambda$-hyperbolic.
\begin{lemma}\label{homoclinic}
$\mu_1$ and $\mu_2$ are homoclinically related.
\end{lemma}
\begin{proof}[Proof of Lemma \ref{homoclinic}]
Let $\O$ be a hyperbolic periodic orbit of $\phi^t$ and $x\in \O$. Let $y_i\in \R_1$ be a generic point for $\mu_i, i=1,2$ respectively. As $W^{0s/0u}(\O) = \F^{0s/0u}(\O)$ are dense in $SM$, we have
$W^{0s}_{\text{loc}}(y_1)\pitchfork W^{0u}(x)$ and $W^{0u}_{\text{loc}}(y_2)\pitchfork W^{0s}(x)$. By the inclination lemma, we conclude that $W^{0s}(y_1)\pitchfork W^{0u}(y_2)$ near $x$. Interchanging the roles of $y_1, y_2$, we also have $W^{0u}(y_1)\pitchfork W^{0s}(y_2)$. This proves the lemma.
\end{proof}

Applying Theorem \ref{coding} to $\mu_1$ and $\mu_2$, we get an irreducible TMF $(\Sigma_r, \sigma_r)$ and a H\"{o}lder continuous map $\pi_r: \Sigma_r\to SM$ such that $\mu_i(\pi_r[\Sigma_r^\#])=1, i=1,2$. Therefore, $\mu_1$ and $\mu_2$ lift to ergodic measures  $\hat \mu_1$ and $\hat \mu_2$ on $\Sigma_r$. These measures are equilibrium states of the potential $\hat{\psi}=\psi\circ \pi_r$. Following exactly the same lines as \cite[Page 8]{LP}, we can show that $\hat{\psi}$ is H\"{o}lder continuous.

The measures  $\hat \mu_1$ and $\hat \mu_2$ project to ergodic $\sigma$-invariant probability measures  $\hat \nu_1$ and $\hat \nu_2$ on the irreducible component $\Sigma$ which are equilibrium states of the H\"{o}lder continuous potential $\hat{\psi}_r-P(\psi)r$ where $\hat{\psi}_r=\int_0^{r(\lv)}\hat \psi (\lv,t)dt$. By \cite[Theorem 1.1]{BuS}, we have $\hat \nu_1=\hat \nu_2$, and so $\mu_1=\mu_2$.

To prove that the unique equilibrium state $\mu$ has full support, note that $\hat \nu$ has full support in $\Sigma$ and then $\hat \mu$  has full support in $\Sigma_r$. Thus $\mu$ has full support in $\overline{\pi_r(\Sigma_r)}$. Let $\O$ be a hyperbolic periodic orbit homoclinically related to $\mu$. Then $\pi_r(\Sigma_r)$ is dense in $\text{HC}(\O)$. Thus by Lemma \ref{hco}, 
$$\text{supp}(\mu)=\overline{\pi_r(\Sigma_r)}=\text{HC}(\O)=SM.$$
By \cite[Theorem 1.2]{ALP}, $\mu$ is Bernoulli, which finishes the proof of the theorem.
\end{proof}

\section{Counting closed geodesics}
Let $M=X/\C$ be a closed $C^\infty$ uniform visibility manifold without conjugate points and with continuous Green bundles, and suppose that the geodesic flow has a hyperbolic periodic point. We apply the mixing properties of the unique MME of the geodesic flow to prove Theorem \ref{margulis}. Another essential tool is a type of closing lemma, for which we need to study the dynamics of $\C$ on $\pX$. 

We do not assume continuous asymptote in Subsections $7.1$ and $7.2$.

\subsection{Periodic orbits and deck transformations}
Given any isometry $\c: X\to X$, the displacement function $d_\c: X\to X$ is defined by $d_\c(x):=d(x, \c x)$. A geodesic $c: \RR\to X$ is called an \textit{axis} of $\c\in \Gamma$, if there is a constant $L>0$ such that $\c c(t)=c(t+L)$ for all $t\in \RR$. For $\c\in \C$ we define $Ax(\c)$ to be the set of all points which are contained in an axis of $\c$. 

\begin{lemma} (see \cite[Lemma 2.1]{CS})
Let $M = X/\C$ be a compact manifold without conjuagate points and $\c\in \C$ be a nontrivial element. Then we have:
\begin{enumerate}
    \item $d_\c$ assumes a positive minimum, $|\c|:=\min d_\c$. 
    \item The set $Ax(\c)$ is equal to the set of critical points of $d_\c$. Furthermore $Ax(\c)$ is the set where $d_\c$ assumes the minimum. 
    \item For $m\in \NN$ we have $|\c^m|=m|\c|$.   
\end{enumerate}
\end{lemma}

\begin{lemma}\label{axis2}
If $\c\in \C$ is an axial isometry for both $c_0$ and $c_1$, then $c_0(\pm\infty)=c_1(\pm \infty)$ and the corresponding closed geodesics on $M$ lie in the same free homotopy class.
\end{lemma}
\begin{proof}
Let $C:=\max_{t\in  [0,|\c|]}d(c_0(t),c_1(t))$. For any $t_0\in [0,|\c|]$ and $n\in \ZZ$,
\begin{equation*}
    \begin{aligned}
&d(c_0(t_0+n|\c|), c_1(t_0+n|\c|))\\
=&d(\c^nc_0(t_0), \c^nc_1(t_0))=d(c_0(t_0), c_1(t_0))\le C.
 \end{aligned}
\end{equation*}
It follows that $c_0(\pm\infty)=c_1(\pm \infty)$, i.e, $c_0$ and $c_1$ are biasymptotic. 

Let us prove that the projections of $c_0$ and $c_1$ on $M$ lie in the same free homotopy class. Let $s\mapsto c_s(0)$ be any path from $c_0(0)$ to $c_1(0)$, and define $c_s(|\c|) := \c c_s(0)$. This defines $c_s(t)$ as a continuous function of $(s, t)$ on the boundary of $[0, 1]\times [0, |\c|]$. Since X is simply connected this extends to a continuous map on all of $[0, 1]\times [0, |\c|]$, and then to $[0, 1]\times \RR$ by defining $c_s(t\pm |\c|) := \c^{\pm 1}c_s(t)$, this gives the desired homotopy. 
\end{proof}

According to Lemma \ref{axis2}, for $\c\in \C$, denote $\xi_\c^{\pm}:=c(\pm\infty)$ where $c$ is an axis of $\c$.

\begin{lemma}\label{fix}
Suppose that $c$ is a geodesic on $X$ such that $\lc = \pr \circ c$ is closed. If $\c\in \C$ fixes $c(-\infty)$ and $c(\infty)$, then either $c(\pm \infty)=\xi_\c^{\pm}$ or $c(\pm \infty)=\xi_\c^{\mp}$.
\end{lemma}
\begin{proof}
Note that $\c c$ is a geodesic such that $(\c c)(\pm \infty)=c(\pm \infty)$. In fact, $(\c^n c)(\pm \infty)=c(\pm \infty)$ for any $n\in \ZZ$. Thus $\c^n c$ lies in a generalized strip (possibly just a single geodesic) connecting $c(\pm \infty)$. Let $c_\c$ be an axis of $\c$. Then
\begin{equation*}
    \begin{aligned}
d(\c^n c(0), \c^nc_\c(0))=d(c(0), c_\c(0)).
 \end{aligned}
\end{equation*}
So $\c^nc_\c(0)$ lies in a $d(c(0), c_\c(0))$-neighborhood of the generalized strip. It is clear that $\c^nc_\c(0)$ converges to $\xi_\c^{\pm}$ according to $n\to \pm \infty$. We get $c(\pm \infty)=\xi_\c^{\pm}$ or $c(\pm \infty)=\xi_\c^{\mp}$.
\end{proof}

\begin{lemma}\label{isometries}
Given any geodesic $c$ on $X$ such that $\lc = \pr \circ c$ is closed, let $\c_1, \c_2\in \C$ be two axial isometries fixing $c(+\infty)$. Then there exists integers $m,n$ such that $\c_1^m=\c_2^n$.
\end{lemma}

\begin{proof}
Let $\c_1, \c_2$ be two axial isometries fixing $c(+\infty)$. Then by Lemma \ref{fix} and replacing
$\c_1$ and $\c_2$ by their inverses if necessary, we can assume that $\c_1$ and $\c_2$ have axes $c_1$ and $c_2$ respectively so that $c_1(+\infty) = c_2(+\infty)=c(+\infty)$. Assume that $\c_1$ translates
$c_1$ by $a > 0$ and $\c_2$ translates $c_2$ by $b > 0$. For each integer $n\ge 1$ let
$c_2(t_n)$ be the foot of $c_1(na)$ on $c_2$. Then there exists an integer $m =m(n)$ such that $|t_n-mb|< b$. Now by Proposition \ref{horofoliation},
\begin{equation*}
\begin{aligned}
    &d(c_1(0), \c_2^{-m}\c_1^nc_1(0)) = d(\c_2^mc_1(0), \c_1^nc_1(0)) \\
\le &d(\c_2^mc_1(0),\c_2^mc_2(0))+ d(c_2(mb),c_2(t_n)+ d(c_2(t_n),c_1(na))\\
\le &d(c_1(0),c_2(0))+b+Ad(c_1(0),c_2(0))+B.
\end{aligned}    
\end{equation*}
By proper discontinuity of the action of $\C$ on $X$, there exists $n_1,n_2$ such that $$\c_2^{-m(n_1)}\c_1^{n_1}=\c_2^{-m(n_2)}\c_1^{n_2}.$$
So there exists integers $m,n$ such that $\c_1^m=\c_2^n$.
\end{proof}
\begin{remark}\label{cyclic}
It seems that we can not conclude that the set of $\c\in \C$ fixing $c(-\infty)$ and $c(+\infty)$ is an infinite cyclic group as \cite[Lemma 2.12]{CKW2}. The axes of these axial isometries form a generalized strip by Lemmas \ref{fix} and \ref{striplemma}, which is not necessarily a flat strip as in the no focal points case. \cite[Lemma 2.12]{CKW2} makes use of the topological conjugacy between the $\C$-action on $\pX$ and the one associated to a metric of negative curvature, which is absent in our setting. That is why we assume continuous asymptote in Theorem \ref{margulis} in order to prove Lemma \ref{injective}. See also Remark \ref{axisissue}.
\end{remark}

\subsection{Neighborhood of an expansive vector}
Let $X$ be a simply connected $C^\infty$ uniform visibility manifold without conjugate points, and $\C\subset\text{Iso}(X)$ acts proper discontinuously on $X$. We study the dynamics of $\C$ on $\pX$, especially in neighborhoods of $v_0^{\pm}$ in $\overline{X}$ where $v_0\in \R_0$. We do not assume that $X$ admits a compact quotient $M=X/\C$, and so the results are also applicable to the Hopf-Tsuji-Sullivan dichotomy in the last section.
 
\begin{lemma}\label{expan}
Let $X$ be a simply connected uniform visibility manifold without conjugate points and $v\in SX$. If there is a constant $c>0$ such that for all $k \in \mathbb{Z}^+$, we can always find a pair of points in the cones:
$$p_{k}\in C(-v,\frac{1}{k}), ~~q_{k}\in C(v,\frac{1}{k}),$$ 
with $d(c_{v}(0),c_{p_{k},q_{k}})\geq c$, then $v\in \R_0^c$.
\end{lemma}
\begin{proof}
The proof is analogous to that of \cite[Proposition 5]{LWW} which essentially uses the cone topology of $\overline{X}$. One can prove that $c_v$ is the boundary of a generalized strip and thus $v\in \R_0^c$.
\end{proof}

\begin{proposition}\label{regularneighbor}
Let $X$ be a simply connected uniform visibility manifold without conjugate points and $v\in \R_0\subset SX$. Then for any $\epsilon > 0$, there are neighborhoods $U_{\epsilon}$ of $c_{v}(-\infty)$ and $V_{\epsilon}$ of $c_{v}(+\infty)$ in $\pX$ such that for each pair $(\xi,\eta)\in U_{\epsilon}\times V_{\varepsilon}$, there exists a geodesic $c_{\xi,\eta}$ connecting $\xi$ and $\eta$, and for any such geodesic $c_{\xi,\eta}$ we have $d(c_{v}(0),c_{\xi,\eta})<\epsilon$.   
\end{proposition}
\begin{proof}
Since $v\in \R_0$, by Lemma \ref{expan}, there exist $k_{0}\in \mathbb{Z}^+$ and a constant number $c > 0$ such that for any points $p\in C(-v,\frac{1}{k_{0}})$ and $q\in C(v,\frac{1}{k_{0}})$, we have $d(c_{v}(0),c_{p,q}) < c$.

Pick up points $p_{i}\in C(-v,\frac{1}{k_{0}})$ and $q_{i}\in C(v,\frac{1}{k_{0}})$ such that $p_{i}\rightarrow \xi$, $q_{i}\rightarrow \eta$. We choose the parametrization of the geodesic $c_{p_{i},q_{i}}$ such that $c_{p_{i},q_{i}}(0)\in B(c_{v}(0),c)$. By the compactness, passing to a subsequence if necessary, we can assume $w = \lim_{i\rightarrow +\infty}\dot c_{p_{i},q_{i}}(0)$. Then $c_{w}(-\infty)=\xi$ and $c_{w}(+\infty)=\eta$, and thus $c_{w}$ is the geodesic $c_{\xi,\eta}$ we want.

It follows that given any $\epsilon>0$, for sufficiently small neighborhoods $U_{\epsilon}$, $V_{\epsilon}$ and all $\xi\in U_{\epsilon}$, $\eta\in V_{\epsilon}$,
there is a geodesic $c_{\xi,\eta}$ connecting $\xi$ and $\eta$. Now we prove $d(c_{v}(0),c_{\xi,\eta})<\epsilon$. Suppose this is not true, then we can find a positive number $\epsilon_{0}$ and two sequences $\{\xi_{n}\}^{+\infty}_{n=1},\ \{\eta_{n}\}^{+\infty}_{n=1}\subset \pX$ such that for all $n\in\mathbb{Z}^+$,
\begin{enumerate}
\item{} ~~$\lim_{n\rightarrow +\infty}\xi_{n}=c_{v}(-\infty), ~~\lim_{n\rightarrow +\infty} \eta_{n}=c_{v}(+\infty)$,
\item{} ~~The connecting geodesic $c_{\xi_{n},\eta_{n}}$ satisfies $d(c_{v}(0),c_{\xi_{n},\eta_{n}}) \geq \epsilon_{0}$.
\end{enumerate}
Then by a similar argument of Lemma \ref{expan}, we can show that $v$ is the boundary of a generalized strip, which contradicts to our assumption that $v\in \R_0$.
\end{proof}

The following lemma is a consequence of uniform visibility.
\begin{lemma}(\cite[Lemma 4.9]{CKW2})\label{piconvergence}
Let $X$ be a simply connected uniform visibility manifold without conjugate points. Let $\textbf{Q}, \textbf{R} \subset \pX$ be disjoint compact sets. Fix a compact subset $A\subset X$ and consider for each $T>0$ the set
\[C^T_{\textbf{Q}}:=\{c_{p,\xi}(t):p\in A, \xi\in \textbf{Q}, t\ge T\},\]
and define $C^T_{\textbf{R}}$ similarly.

Then given any open sets $\textbf{U}, \textbf{V}\subset \pX$ such that $\textbf{U}\supset \textbf{Q}$ and $\textbf{V}\supset \textbf{R}$, there exists $T>0$ such that if $\c\in \C\subset \text{Iso}(X)$ satisfies $\c(C^T_{\textbf{Q}})\cap A\neq \emptyset$ and $\c(A)\cap C^T_{\textbf{R}}\neq \emptyset$, then 
\[\c(\pX\setminus \textbf{U})\subset \textbf{V }\text{\ and\ }\c^{-1}(\pX\setminus \textbf{V})\subset \textbf{U}.\]
\end{lemma}

\begin{proposition}\label{north}
Let $X$ be a simply connected uniform visibility manifold without conjugate points and let $c$ be an axis of $h \in \Gamma \subset \text{Iso}(X)$. For all neighborhood $U \subset \overline X$ of $c(-\infty)$ and $V \subset \overline X$ of $c(+\infty)$,
there is $N \in \NN$ such that
$$h^{n}(\overline X\setminus U)\subset V, ~~h^{-n}(\overline X\setminus V)\subset U$$
for all $n \geq N$.    
\end{proposition}
\begin{proof}
We apply Lemma \ref{piconvergence} with compact sets $\textbf{Q}\subset U\cap \pX$ and $\textbf{R}\subset V\cap \pX$, and $A$ a compact neighborhood of $c(0)$. Then there exists $N \in \NN$ such that $h^{n}(C^T_{\textbf{Q}})\cap A\neq \emptyset$ and $h^{n}(A)\cap C^T_{\textbf{R}}\neq \emptyset$ for all $n \geq N$. Then 
\[h^{n}(\pX\setminus U)\subset V, ~~h^{-n}(\pX\setminus V)\subset U\]
for all $n \geq N$. 

It remains to consider points in $X\setminus U$. Assume the contrary, then there exists a sequence of points $p_k\in X\setminus U$ such that $h^kp_k\to \xi$ for some point $\xi\notin V$. This can happen only when $d(c(0), p_k)\to \infty$  since $h$ is an axial isometry with endpoint in $V$. Thus we can assume that $p_k\to \eta$ for some $\eta\in \pX\setminus U$. On the other hand, the proof of \cite[Lemma 4.9]{CKW2} uses the requirement $\angle_q(\c^{-1}x, \eta)>2\e$ with $\eta\in \pX$ there. We can check that the proof also works for $p_k\in X$ (instead of $\eta$) satisfying $\angle_q(\c^{-1}x, p_k)>2\e$. $\angle_q(\c^{-1}x, p_k)>2\e$ holds for any $k$ large enough. As a consequence of the proof of \cite[Lemma 4.9]{CKW2}, we have that there exists $N\in \NN$, such that $c_{c(0),h^np_k}(+\infty)\in V$ for any $n\ge N$ and any large enough $k$. This contradicts to that $h^kp_k\to \xi\notin V$. This proves the first inclusion in the proposition.

The second inclusion can be proved analogously, and the proof of the proposition is complete.
\end{proof}

\subsection{Flow boxes and a closing lemma}
From this subsection, we let $M=X/\C$ be a closed $C^\infty$ uniform visibility manifold without conjugate points and with continuous Green bundles, and suppose that the geodesic flow has a hyperbolic periodic point. Let $m$ denote the unique MME of the geodesic flow, which is also the Bowen-Margulis measure. 

We call $v$ and $c_v$ regular if $v\in \R_1$. Fix a regular vector $v_0\in \R_1\subset SX$.  Let $p:=\pi(v_0)$ be the reference point in the following discussions. 
We also fix a scale $\e\in (0, \min\{\frac{1}{8}, \frac{\inj (M)}{4}\})$. 

The \emph{Hopf map} $H: SX\to \pX\times \pX\times \RR$ for $p\in X$ is defined as
$$H(v):=(v^-, v^+, s(v)), \text{\ where\ }s(v):=b_{v^-}(\pi v, p).$$
From definition, we see
$s(\phi^t v)=s(v)+t$
for any $v\in SX$ and $t\in \RR$. $s$ is continuous by Corollary \ref{continuous}.

Following \cite{CKW2, Wu2}, we define for each $\theta>0$ and $0<\a<\frac{3}{2}\e$,
\begin{equation*}
\begin{aligned}
&\bP=\bP_\theta:=\{w^-: w\in S_pX \text{\ and\ }\angle_p(w, v_0)\le \theta\},\\
&\bF=\bF_\theta:=\{w^+: w\in S_pX \text{\ and\ }\angle_p(w, v_0)\le \theta\},\\
&B=B_\theta^\a:=H^{-1}(\bP\times \bF\times [0,\a]),\\
&S=S_\theta:=B_\theta^{\e^2}=H^{-1}(\bP\times \bF\times [0,\e^2]).
\end{aligned}
\end{equation*}
$B=B_\theta^\a$ is called a \emph{flow box} with depth $\a$. We will consider $\theta>0$ small enough, which will be specified in the following.

Recall that $v_0\in \R_1$ and $\R_1$ is open. The following is essentially a restatement of Proposition \ref{regularneighbor} in the new setting.
\begin{proposition}\label{crucial}
For any $\e>0$, there is an $\theta_1 > 0$ such that, for any $\xi \in \bP_{\theta_1}$ and $\eta \in \bF_{\theta_1}$,
there is a unique geodesic $c_{\xi,\eta}$ connecting $\xi$ and $\eta$, i.e.,
$c_{\xi,\eta}(-\infty)=\xi$ and $c_{\xi,\eta}(+\infty)=\eta$.

Moreover, the geodesic $c_{\xi,\eta}$ is regular and $d(c_{v}(0), c_{\xi,\eta})<\epsilon/10$.
\end{proposition}

Based on Proposition \ref{crucial}, we have the following result.
\begin{lemma}\label{diameter}
Let $v_0, p, \e$ be as above and $\theta_1$ be given in Proposition \ref{crucial}. Then for any $0<\theta<\theta_1$,
\begin{enumerate}
  \item $\text{diam\ } \pi H^{-1}(\bP_\theta\times \bF_\theta\times \{0\})<\frac{\e}{2}$;
  \item $H^{-1}(\bP_\theta\times \bF_\theta\times \{0\})\subset SX$ is compact;
  \item $\text{diam\ } \pi B_\theta^\a <4\e$ for any $0<\a\le \frac{3\e}{2}$.
\end{enumerate}
\end{lemma}
\begin{proof}
The proof is analogous to that of \cite[Lemma 2.13]{Wu2}, and thus omitted.
\end{proof}

At last, we also assume that the choice of $\theta\in (0, \theta_1)$ satisfies the following properties. Since $\theta\mapsto \bar\mu(\bP_\theta\times \bF_\theta)$ is nondecreasing and hence has at most countably many discontinuities, we always choose $\theta$ to be the continuity point of this function, i.e.,
$$\lim_{\rho\to \theta}\bar\mu(\bP_\rho\times \bF_\rho)=\bar\mu(\bP_\theta\times \bF_\theta).$$
Furthermore, $\theta$ is chosen so that $\bar\mu(\partial\bP_\theta\times \partial\bF_\theta)=0.$ By the product structure of $B$ and $S$ and the definition of $m$, we have for any $\a\in (0,\frac{3\e}{2})$,
\begin{equation}\label{e:choice}
\begin{aligned}
\lim_{\rho\to \theta} m(S_\rho)=m(S_\theta), \quad \lim_{\rho\to \theta} m(B^\a_\rho)=m(B^\a_\theta),
\text{\ \ and\ \ }m(\partial B^\a_\theta)=0.
\end{aligned}
\end{equation}

As a consequence of Corollary \ref{equicon1}, we have:
\begin{corollary}\label{equicon}
Given $v_0, p, \e>0$ as above, there exists $\theta_2>0$ such that for any $0<\theta<\theta_2$, if $\xi,\eta\in \bP_\theta$ and any $q$ lying within $\diam \F+4\e$ of $\pi H^{-1}(\bP_\theta\times \bF_\theta\times [0,\infty))$, we have $|b_\xi(q,p)-b_\eta(q,p)|<\e^2$. Similar result holds if the roles of $\bP_\theta$ and $\bF_\theta$ are reversed.
\end{corollary}
Let $\theta_0:=\min\{\theta_1, \theta_2\}$, where $\theta_1$ is given in Lemma \ref{diameter}, and $\theta_2$ is given in Corollary \ref{equicon}. In the following, we always suppose that $0<\theta<\theta_0$ and \eqref{e:choice} is satisfied.

Let us collect the definitions of important subsets of $\C$ here for convenience.
\begin{equation*}
\begin{aligned}
\C(t,\a)=\C_\theta(t,\a)&:=\{\c\in \C: S_\theta\cap \phi^{-t}\c_* B_\theta^\a\neq \emptyset\},\\
\C^*=\C^*_\theta&:=\{\c\in \C: \c\bF_\theta \subset \bF_\theta \text{\ and\ }\c^{-1}\bP_\theta\subset \bP_\theta\},\\
\C^*(t,\a)&:=\C^*\cap \C(t,\a),\\
\C'(t,\a)&:=\{\c\in \C^*(t,\a): \c\neq \b^n \text{\ for any\ } \b\in \C, n\ge 2\}.
\end{aligned}
\end{equation*}

\begin{lemma}[Closing lemma]\label{closing}
Let $X$ be a simply connected manifold without focal points, $v_0, p, \e$ be fixed as before, and $\theta_0$ be given as above. Then for every $0<\rho<\theta<\theta_0$, there exists some $t_0> 0$ such that for all $t\ge t_0$, we have $\C_\rho(t,\a)\subset \C_\theta^*$.
\end{lemma}
\begin{proof}
The proof is analogous to that of \cite[Lemma 4.8]{CKW2}, and is based on Lemma \ref{piconvergence}.
\end{proof}

In the following, we use the scaling and mixing properties of Bowen-Margulis measure $m$, to give an asymptotic estimates of $\#\C^*(t,\a)$ and $\#\C(t,\a)$.
\subsection{Depth of intersection}
Firstly, we collect some results on the relation between $t$ and $|\c|$ when $\c\in \C^*(t,\a)$. The proof is analogous to those of \cite[Lemmas 4.11, 4.12, 4.13. and 4.14]{CKW2}, and hence the details are omitted here.

Given $\xi\in \pX$ and $\c\in \C$, define $b_\xi^\c:=b_\xi(\c p, p)$.
\begin{lemma}\label{intersection}
The following relations hold:
\begin{enumerate}
    \item Let $\xi,\eta\in \bP$ and $\c\in \C(t,\a)$ with $t>0$. Then $|b_\xi^\c-b_\eta^\c|<\e^2$.
    \item Let $c$ be an axis of $\c\in \C$ and $\xi=c(-\infty)$. Then $b_\xi^\c=|\c|$.
    \item Given any $\c\in \C$ and any $t\in \RR$, we have
$$S\cap \phi^{-t}\c B^\a=\{w\in P^{-1}(\bP\times \c\bF): s(w)\in [0,\e^2]\cap (b_{w^{-}}^\c-t+[0,\a])\}.$$
\item If $\c\in \C^*(t,\a)$, then $|\c|\in [t-\a-\e^2, t+2\e^2]$.
\end{enumerate}
\end{lemma}
\begin{proof}
The proof of (1) is an application of Corollary \ref{equicon}. The computation is completely parallel to that in \cite[Lemma 4.11]{CKW2}.

The proof of (2) involves computation of Busemann functions, which is completely parallel to that in \cite[Lemma 4.12]{CKW2}. 

The proof of (3) is completely parallel to that in \cite[Lemma 4.13]{CKW2}.

The proof of (4) uses (1), (2), (3) and also the fact that $\c\in \C^*$. See the proof of \cite[Lemma 4.14]{CKW2}. 
\end{proof}

The following lemma implies that the intersections also have product structure.
\begin{lemma}\label{depth1}
If $\c\in \C^*(t,\a)$, then
$$S\cap \phi^{-(t+2\e^2)}\c B^{\a+4\e^2}\supset H^{-1}(\bP\times \c\bF\times [0,\e^2]):=S^\c.$$
\end{lemma}
\begin{proof}
The proof is completely parallel to that in \cite[Lemma 5.1]{CKW2}. We reproduce the proof since it reflects that advantage of the use of the slice $S$.

Let $\c\in \C^*(t,\a)$, then $S\cap g^{-t}\c B^\a\neq \emptyset$. By Lemma \ref{intersection}(3), there exists $\eta\in \bP$ such that
$$[0,\e^2]\cap (b_{\eta}^\c-t+[0,\a])\neq \emptyset.$$
It follows that $[0,\e^2]\subset (b_{\eta}^\c-t-\e^2+[0,\a+2\e^2]).$ Then by Lemma \ref{intersection}(1), for any $\xi\in \bP$ we have
$$[0,\e^2]\cap (b_{\xi}^\c-t-\e^2+[0,\a+2\e^2])\neq \emptyset,$$
which in turn implies that
$$[0,\e^2]\subset  (b_{\xi}^\c-t-2\e^2+[0,\a+4\e^2]).$$
We are done by Lemma \ref{intersection}(3).
\end{proof}

\subsection{Scaling and mixing calculation}
The following notations in the asymptotic estimates will be convenient.
\begin{equation*}
\begin{aligned}
f(t)=e^{\pm C}g(t)&\Leftrightarrow e^{-C}g(t)\le f(t)\le e^{C}g(t) \text{\ for all\ } t;\\
f(t) \lesssim g(t) &\Leftrightarrow \limsup_{t\to \infty}\frac{f(t)}{g(t)}\le 1;\\
f(t) \gtrsim g(t) &\Leftrightarrow \liminf_{t\to \infty}\frac{f(t)}{g(t)}\ge 1;\\
f(t) \sim g(t) &\Leftrightarrow \lim_{t\to \infty}\frac{f(t)}{g(t)}= 1;\\
f(t)\sim e^{\pm C}g(t)&\Leftrightarrow e^{-C}g(t)\lesssim f(t)\lesssim e^{C}g(t).
\end{aligned}
\end{equation*}

\begin{lemma}\label{scaling1}
If $\c\in \C^*$, then
$$m(S^\c)=e^{\pm 2h\e}e^{-h|\c|}m(S).$$
\end{lemma}

\begin{proof}
The proof is completely parallel to that in \cite[Lemma 5.2]{CKW2}. The main work is to estimate $\b_p(\xi,\eta)$ and $b_\eta(\c^{-1}p,p)$ given $\xi\in \bP, \eta\in \bF$.

Firstly, take $q$ lying on the geodesic connecting $\xi$ and $\eta$ such that $b_\xi(q,p)=0$. Then
$$|\b_p(\xi,\eta)|=|b_\xi(q,p)+b_\eta(q,p)|=|b_\eta(q,p)|\le d(q,p)<\frac{\e}{2}$$
where we used Lemma \ref{diameter} in the last inequality.

Secondly, since $\c\in \C^*$, we know $\c^{-1}$ has a regular axis $c$, i.e. $\dot c(0)\in \R_1$. By Lemma \ref{intersection}(2), $b_{c(-\infty)}(\c^{-1}p,p)=|\c^{-1}|=|\c|$. Then by Corollary \ref{equicon}, we have $|b_\eta(\c^{-1}p,p)-|\c||<\e^2$ for any $\eta\in \bF$.

Notice that since $\c\bF\subset \bF$, we have $\c\eta\in \bF$ if $\eta\in \bF$. Thus we have
\begin{equation*}
\begin{aligned}
\frac{m(S^\c)}{m(S)}=&\frac{\e^2\bar \mu(\bP \times \c\bF)}{\e^2\bar \mu(\bP \times \bF)}=\frac{e^{\pm h\e/2}\mu_p(\bP)\mu_p(\c\bF)}{e^{\pm h\e/2}\mu_p(\bP)\mu_p(\bF)}\\
=&e^{\pm h\e}\frac{\mu_{\c^{-1}p}(\bF)}{\mu_p(\bF)}=e^{\pm h\e}\frac{\int_{\bF}e^{-hb_\eta(\c^{-1}p,p)}d\mu_p(\eta)}{\mu_p(\bF)}\\
=&e^{\pm h\e}\e^{\pm h\e^2}e^{-h|\c|}=e^{\pm 2h\e}e^{-h|\c|}.
\end{aligned}
\end{equation*}
\end{proof}

Combining Lemma \ref{intersection}(4) and Lemma \ref{scaling1}, we have
\begin{corollary}[Scaling]\label{scaling2}
Given $\a\le \frac{3}{2}\e$ and $\c\in \C^*(t,\a)$, we have $|t-|\c||\le 2\e$, and thus
$$m(S^\c)=e^{\pm 4h\e}e^{-ht}m(S).$$
\end{corollary}

\begin{remark}
It is clear that the conclusions in Lemma \ref{scaling1} and Corollary \ref{scaling2} hold if $m, S, S^\c$ are replaced by $\lm, \lS, {\lS}^\c$ respectively. Recall that we use an underline to denote objects in $M$ and $SM$, see Subsection $3.2$.
\end{remark}

Finally, we combine scaling and mixing properties of Bowen-Margulis measure to obtain the following asymptotic estimates. The proof is a repetition of \cite[Section 5.2]{CKW2}. Since it is a key step, we provide a detailed proof here.
\begin{proposition}\label{asymptotic}
We have
\begin{equation*}
\begin{aligned}
e^{-4h\e}\lesssim &\frac{\#\C^*_\theta(t,\a)}{e^{ht}m(B_\theta^\a)}\lesssim e^{4h\e}(1+\frac{4\e^2}{\a}),\\
e^{-4h\e}\lesssim &\frac{\#\C_\theta(t,\a)}{e^{ht}m(B_\theta^\a)}\lesssim e^{4h\e}(1+\frac{4\e^2}{\a}).
\end{aligned}
\end{equation*}
\end{proposition}
\begin{proof}
Recall that $\a\in (0, \frac{3\e}{2}]$. By Lemmas \ref{closing} and \ref{depth1}, for any $0<\rho<\theta$ and $t$ large enough, we have
$$\lS_\rho\cap \phi^{-t}\lB_\rho^{\a}\subset \bigcup_{\c\in \C^*_\theta(t,\a)}\lS_\theta^\c \subset \lS_\theta \cap \phi^{-(t+2\e^2)}\lB_\theta^{\a+4\e^2}.$$
By Corollary \ref{scaling2}, $\lm(\lS_\theta^\c)=e^{\pm 4h\e}e^{-ht}\lm(\lS_\theta).$ Thus we have
\begin{equation*}
\begin{aligned}
e^{-4h\e}\lm(\lS_\rho\cap \phi^{-t}\lB_\rho^\a)&\le \#\C^*_\theta(t,\a)e^{-ht}\lm(\lS_\theta)\\
&\le e^{4h\e}\lm(\lS_\theta\cap \phi^{-(t+2\e^2)}\lB_\theta^{\a+4\e^2}).
\end{aligned}
\end{equation*}
Dividing by $\lm(\lS_\theta)\lm(\lB_\theta^\a)$ and using mixing of $\lm$, we get
\begin{equation}\label{e:sim}
\begin{aligned}
e^{-4h\e}\frac{m(S_\rho)m(B_\rho^\a)}{m(S_\theta)m(B_\theta^\a)} \lesssim \frac{\#\C^*_\theta(t,\a)}{e^{ht}m(B_\theta^\a)}
\lesssim e^{4h\e}\frac{m(B_\theta^{\a+4\e^2})}{m(B_\theta^\a)}.
\end{aligned}
\end{equation}
By \eqref{e:choice}, letting $\rho \nearrow \theta$, we obtain the first line of inequalities in the proposition.

To prove the second line, we consider $\theta<\rho< \theta_0$. Then by Lemma \ref{closing}, $\C^*_\theta(t,\a)\subset \C_\theta(t,\a) \subset \C^*_\rho(t,\a).$ By \eqref{e:sim},
\begin{equation*}
\begin{aligned}
e^{-4h\e}\frac{m(S_\rho)m(B_\rho^\a)}{m(S_\theta)m(B_\theta^\a)} &\lesssim \frac{\#\C^*_\theta(t,\a)}{e^{ht}m(B_\theta^\a)}\lesssim \frac{\#\C_\theta(t,\a)}{e^{ht}m(B_\theta^\a)}\\
&\lesssim \frac{\#\C^*_\rho(t,\a)}{e^{ht}m(B_\theta^\a)}\lesssim e^{4h\e}\frac{m(B_\rho^{\a+4\e^2})}{m(B_\theta^\a)}.
\end{aligned}
\end{equation*}
Letting $\rho\searrow \theta$ and by \eqref{e:choice}, we get the second line of inequalities in the proposition.
\end{proof}

\subsection{Upper bound for \texorpdfstring{$\#C(t)$}{\#C(t)}}
Recall that $C(t)$ is any maximal set of pairwise non-free-homotopic closed geodesics with length $(t-\e,t]$ in $M$ (see Theorem \ref{equi}). We shall obtain an upper bound and a lower bound respectively for $\#C(t)$.

Recall the definition of $\nu_t$ (see Theorem \ref{equi}), 
$$\nu_t:=\frac{1}{C(t)}\sum_{c\in C(t)}\frac{Leb_c}{t}.$$
Then we have
$$\#C(t)=\frac{\sum_{\lc\in \#C(t)} Leb_\lc(\lB_\theta^\a)}{t\nu_t(\lB_\theta^\a)}.$$
Define
$$\Pi(t):=\{\dot{\lc}(s)\in \pr H^{-1}(\bP\times \bF)\times \{0\}: \lc\in C(t), s\in \RR\}.$$
Then we have
\begin{equation}\label{e:ct}
\begin{aligned}
\#C(t)=\frac{\a\#\Pi(t)}{t\nu(B_\theta^\a)}.
\end{aligned}
\end{equation}

Now we define a map $\Theta: \Pi(t)\to \C(t,\e)$ as follows. Given $\lv\in \Pi(t)$, let $\ell=\ell(\lv)\in (t-\e,t]$ be such that $\phi^\ell\lv=\lv$. Let $v$ be the unique lift of $\lv$ such that $v\in H^{-1}(\bP\times \bF)\times \{0\}\subset B_\theta^\a$.
Define $\Theta(v)$ to be the unique axial isometry of $X$ such that $\phi^\ell v=\Theta(v) v$. Then $|\Theta(v)|=\ell$. If $\c=\Theta(v)$, then 
$$\phi^tv=\phi^{t-\ell}\c v\in \c B_\theta^\e,$$ 
i.e., $v\in S_\theta\cap \phi^{-t}\c B_\theta^\e$. So we get
\begin{equation}\label{e:theta}
\begin{aligned}
\Theta(\Pi(t))\subset \C(t,\e).
\end{aligned}
\end{equation}

To estimate $\#\Pi(t)$, we first show that $\Theta$ is injective.
\begin{lemma}\label{injective}
$\Theta$ is injective.
\end{lemma}
\begin{proof}
Suppose that $\lv, \lw\in \Pi(t)$ are such that $\Theta(\lv)=\Theta(\lw):=\c$.
Let $v,w\in B_\theta^\a$ be the lifts of $\lv, \lw$ respectively. Then by definition, both $c_v$ and $c_w$ are axes of $\c$. By Lemma \ref{axis2},  $v^+=w^+$ and $v^-=w^-$.

It follows that $c_v$ and $c_w$ are bi-asymptotic. If $c_v$ and $c_w$ are geometrically distinct, then they bound a generalized strip by Lemma \ref{striplemma}. So $v$ and $w$ are singular vectors, which is a contradiction by the definition of $B_\theta^\a$ and Proposition \ref{crucial}. So $v$ and $w$ lie on a common geodesic. As $v,w\in H^{-1}(\bP\times \bF\times \{0\})$, we have $v=w$ and hence $\lv=\lw$. So $\Theta$ is injective.
\end{proof}

\begin{remark}\label{axisissue}
Lemma \ref{injective} essentially uses that our flow box is contained in $\R_1$ by Proposition \ref{crucial}. If we take the center of the flow box inside $\R_0$, then the flow box might contain non-expansive vectors. In this case, if we follow the proof of \cite[Lemma 4.3]{CKW2}, we need use \cite[Lemma 2.12]{CKW2}, which is not available in our setting. See Remark \ref{cyclic}.
\end{remark}

\begin{proposition}\label{upper}
We have
$$\#C(t)\le \frac{\e \#\C(t,\e)}{t\nu_t(\lB^\e)}.$$
\end{proposition}
\begin{proof}
The proposition follows from \eqref{e:ct}, \eqref{e:theta} and Lemma \ref{injective}.
\end{proof}
\subsection{Lower bound for \texorpdfstring{$\#C(t)$}{\#C(t)}}

First we deal with the multiplicity of $\c\in \C$. Given $\c\in \C$, let $d=d(\c)\in \NN$ be maximal such that $\c= \b^d$ for some $\b\in \C$. $\c\in \C$ is called \emph{primitive} if $d(\c)=1$, i.e., $\c\neq \b^d$ for any $\b\in \C$ and any $d\ge 2$.

Define $\C_2(\bP, \bF, t)$ to be the set of all $\c\in \C$ such that
\begin{enumerate}
  \item $\c$ has an axis $c$ with $c(-\infty)\in \bP, c(\infty)\in \bF$;
  \item $|\c|\in (t-\e,t]$;
  \item $d(\c)\ge 2$.
\end{enumerate}

\begin{lemma}\label{multi}
There exists $K>0$ such that for any $t>0$ we have
$$\sum_{\c\in \C_2(\bP, \bF, t)}d(\c)\le Ke^{\frac{2}{3}ht}.$$
\end{lemma}
\begin{proof}
By Freire and Ma\~n\'e's theorem (cf.~\cite{FrMa}), the topological entropy of the geodesic flow coincides with the volume entropy, i.e., for any $x\in X$
$$h=\lim_{r\to \infty}\frac{\log \text{Vol}B(x,r)}{r}$$
where $B(x,r)$ is the open ball of radius $r$ centered at $x$ in $X$. Based on this fact, the proof is almost analogous to that in \cite[Lemma 4.5]{CKW2} with only minor modifications as follows.

If $\c=\b(\c)^{d(\c)}$ for some $\b(\c)\in \C$, we must argue that $\b(\c)$ has an axis with endpoints in $\bP$ and $\bF$, so that we can choose $v\in H^{-1}(\bP\times \bF\times \{0\})$ tangent to such an axis with $\phi^{|\b(\c)|}v=\b(\c)v$. We cannot use \cite[Lemma 2.10]{CKW2} in our setting. Nevertheless, since $\c$ has a regular axis with endpoints in $\bP$ and $\bF$, we know from the proof of Lemma \ref{injective} that this is the only axis for $\c$. Every axis of $\b(\c)$ is an axis of $\c$, so  $\b(\c)$ also has a unique axis, which is the axis of $\c$, as we want.

Repeat the remaining part of the proof of \cite[Lemma 4.5]{CKW2} and we are done.
\end{proof}

Recall that $\C'(t,\a):=\{\c\in \C^*(t,\a): \c\neq \b^n \text{\ for any\ } \b\in \C, n\ge 2\}$.

\begin{lemma}\label{lower}
Consider $\a=\e-4\e^2$. Then $\Theta(\Pi(t))\supset \C'(t-2\e^2, \a)$ and
$$\#C(t)\ge \frac{\a \#\C'(t-2\e^2,\a)}{t\nu_t(\lB^\a)}.$$
\end{lemma}
\begin{proof}
Let $\c\in \C'(t-2\e^2, \a)$.
Then there exists $v\in H^{-1}(\bP\times \bF\times \{0\})$ such that $\phi^{|\c|}v=\c v$.
By Lemma \ref{intersection}(4), we have
\begin{equation*}
\begin{aligned}
|\c|&\ge (t-2\e^2)-\a-\e^2=t-\e+\e^2>t-\e,\\
|\c|&\le (t-2\e^2)+2\e^2=t.
\end{aligned}
\end{equation*}
Since $\c$ is primitive, $\underline c_{\lv}$ is a closed geodesic with length $|\c|\in (t-\e,t]$.
Note that if $\underline c$ is another closed geodesic in the free-homotopic class of $\underline c_{\lv}$, then we can lift $\underline c$ to a geodesic $c$ such that $c$ and $c_v$ are bi-asymptotic. So $c$ and $c_v$ bound a generalized strip by Lemma \ref{striplemma}, which is a contradiction since $v$ is regular. It follows that $\underline c_{\lv}$ is the only geodesic in its free-homotopic class. Thus $\underline c_{\lv}\in C(t)$.

As a consequence, $v\in \Pi(t)$ and $\c=\Theta(v)$. So $\Theta(\Pi(t))\supset \C'(t-2\e^2, \a)$ and thus by \eqref{e:ct},
$$\#C(t)\ge \frac{\a \#\C'(t-2\e^2,\a)}{t\nu_t(\lB^\a)}.$$
\end{proof}

\begin{proposition}\label{lower1}
Consider $\a=\e-4\e^2$. We have
$$\#C(t)\ge \frac{\a }{t\nu_t(\lB^\a)}\cdot (\#\C^*(t-2\e^2,\a)-Ke^{\frac{2}{3}ht}).$$
\end{proposition}
\begin{proof}
From the proof of Lemma \ref{lower}, we also see that $|\c|\in (t-\e,t]$ if $\c\in \C^*(t-2\e^2,\a)$. Thus
$$\C^*(t-2\e^2,\a)\setminus \C'(t-2\e^2,\a)\subset \C_2(\bP, \bF, t).$$
Then by Lemma \ref{multi},
$$\#\C^*(t-2\e^2,\a)- \#\C'(t-2\e^2,\a)\le Ke^{\frac{2}{3}ht}.$$
This together with Lemma \ref{lower} proves the proposition.
\end{proof}

\subsection{Equidistribution and completion of the proof}
The following result is standard in ergodic theory, which is a corollary of the classical proof of variational principle \cite[Theorem 9.10]{W}. The flow version is given in \cite[Proposition 4.3.14]{FH}.
\begin{lemma}\label{equilemma}
Let $Y$ be a compact metric space and $\{\phi^t\}_{t\in \RR}$ a continuous flow on $Y$. Fix $\e > 0$ and suppose that $E_t\subset Y$ is a $(t,\e)$-separated set for all sufficiently large $t$. Define the measures $\mu_t$ by
$$\mu_t(A) :=\frac{1}{\#E_t}\sum_{v\in E_t}\frac{1}{t}\int_0^t\chi_A(\phi^sv)ds.$$
If $t_k\to \infty$ and the weak$^*$ limit $\mu=\lim_{k\to \infty}\mu_{t_k}$
exists, then
$$h_\mu(\phi^1)\ge \limsup_{k\to \infty}\frac{1}{t_k}\log \#E_{t_k}.$$
\end{lemma}

\begin{proof}[Proof of Theorem \ref{equi}]
First we prove the following claim.\\
\textbf{Claim:} The set $\{\underline{\dot  c}(0): c\in C(t)\}$ is $(t,\e)$-separated for any $0<\e<\inj(M)/2$.

Assume the contrary, namely, $C(t)$ contains two closed geodesics $\underline c_1, \underline c_2$ in distinct free-homotopic classes
such that $d(\underline c_1(s), \underline c_2(s))\le \e< \inj(M)/2$ for all $s\in [0,t]$. Define $\lv=\underline{\dot c}_1(0)$ and $\lw=\underline{\dot c}_2(0)$.
We can lift  $\lv, \lw$ to $v,w\in SX$, and $\underline c_1, \underline c_2$ to $c_1,c_2$ respectively, such that $d(c_1(s), c_2(s))\le \e$ for all $s\in [0,t]$.
Moreover, there exist $\c_1,\c_2\in \C$ and $t_1,t_2\in (t-\e,t]$ such that $\c_1v=\phi^{t_1}v$ and $\c_2w=\phi^{t_2}w$. Then
\begin{equation*}
\begin{aligned}
d(\c_2^{-1}\c_1 c_1(0), c_2(0))=&d(\c_2^{-1}c_1(t_1), \c_2^{-1}c_2(t_2))=d(c_1(t_1), c_2(t_2))\\
\le &d(c_1(t_1), c_2(t_1))+|t_2-t_1|\le \e+2\e=3\e.
\end{aligned}
\end{equation*}
Hence $d(\c_2^{-1}\c_1 c_1(0), c_2(0))\le 3\e<2 \inj(M)$, which is possible only if $\c_1=\c_2$. Then $c_1$ and $c_2$ are both axes for a common $\c:=\c_1=\c_2$. By Lemma \ref{axis2}, we see that $c_1$ and $c_2$ must be bi-asymptotic and that $\underline c_1$ and $\underline c_2$ are free-homotopic. A contradiction, so the claim holds.

Now by Propositions \ref{lower1} and \ref{asymptotic}, we know
\begin{equation*}
\begin{aligned}
\#C(t)\ge  &\frac{\a }{t\nu_t(\lB^\a)}\cdot (\#\C^*(t-2\e^2,\a)-Ke^{\frac{2}{3}ht})\\
\gtrsim  &\frac{\a }{t\nu_t(\lB^\a)}\cdot (e^{-4h\e}e^{ht}m(B_\theta^\a)-Ke^{\frac{2}{3}ht}).
\end{aligned}
\end{equation*}
So $\liminf_{t\to \infty}\frac{1}{t}\log \#C(t)\ge h.$ Applying Lemma \ref{equilemma}, we know any limit measure of $\nu_t$ has entropy equal to $h$, and thus it must be $m$, the unique MME. This proves Theorem \ref{equi}.
\end{proof}

\begin{proposition}\label{sumup}
We have
$$\#C(t)\sim e^{\pm Q\e}\frac{\e}{t}e^{ht}$$
where $Q>0$ is a universal constant depending only on $h$.
\end{proposition}
\begin{proof}
By Theorem \ref{equi} and \eqref{e:choice}, we have $\nu_t(\lB^\a)\to \lm(\lB^\a)=m(B^\a)$ for $\a=\e$ and $\a=\e-4\e^2$.
\begin{enumerate}
    \item Consider $\a=\e$. By Propositions \ref{upper} and \ref{asymptotic}, we have
$$\#C(t)\lesssim \frac{\e \#\C(t,\e)}{tm(\lB^\e)} \lesssim e^{4h\e}(1+4\e)\frac{\e}{t}e^{ht}.$$
\item Now consider $\a=\e-4\e^2$. By Propositions \ref{lower1} and \ref{asymptotic},
\begin{equation*}
\begin{aligned}
\#C(t)\gtrsim &\frac{\a }{tm(\lB^\a)}\cdot (\#\C^*(t-2\e^2,\a)-Ke^{\frac{2}{3}ht})\\
\gtrsim &(1-4\e)e^{-4h\e}\frac{\e}{t}e^{-2h\e^2}e^{ht}.
\end{aligned}
\end{equation*}
\end{enumerate}
As $0<\e<\frac{1}{8}$, there exists a universal $Q>0$ depending only on $h$ such that
$\#C(t)\sim e^{\pm Q\e}\frac{\e}{t}e^{ht}$.
\end{proof}

\begin{proof}[Proof of Theorem \ref{margulis}]
The last step of the proof of Theorem \ref{margulis} is to estimate $\#P(t)$ via $\#C(t)$ using a Riemannian sum argument.
Indeed, a verbatim repetition of the proof in \cite[Section 6.2]{CKW2} gives
$$\#P(t)\sim e^{\pm 2(Q+h)\e}\frac{e^{ht}}{ht}.$$
Since $\e>0$ can be arbitrarily small, we get $\#P(t)\sim \frac{e^{ht}}{ht}$ which completes the proof of Theorem \ref{margulis}.
\end{proof}

\section{Volume Asymptotics}

Let $M=X/\C$ be a closed $C^\infty$ visibility manifold without conjugate points and with continuous Green bundles. Suppose that the geodesic flow has a hyperbolic periodic point. We prove Theorem \ref{margulis2} in this section, that is,
$$b_t(x)/\frac{e^{ht}}{h} \sim c(x),$$
where $b_t(x)$ is the Riemannian volume of the ball of radius $t>0$ around $x\in X$, $h$ the topological entropy of the geodesic flow, and $c: X\to \RR$ is a continuous function. The strategy is to use the mixing properties of the MME to count the number of intersections of a flow box under the geodesic flow with another flow box. However, we have to deal with countably many pairs of flow boxes centered at regular vectors.

\subsection{Regular partition-cover and local uniform expansion}
We first prove two auxiliary lemmas which demonstrate that regular vectors are abundant inside stable horospherical manifolds and unit tangent spheres.

\begin{lemma}\label{recreg}
Assume that $v \in \R_1$ is a recurrent vector, then any $w \in \F^{s}(v)$ is in $\R_1$.
\end{lemma}
\begin{proof}
First of all, since $v$ is a recurrent vector in $SX$, then we can find a sequence $\{\alpha_{n}\}^{\infty}_{n=1}\subset \Gamma$ and a sequence of time $\{t_{n}\}^{\infty}_{n=1}$ with $t_{n}\rightarrow +\infty$, such that
\begin{equation}\label{rec v}
d\alpha_{n}(\phi^{t_n}v)\to v,~~~n\to +\infty.
\end{equation}

Now take an arbitrary vector $w \in \F^{s}(v)$. Assume the contrary that $w\in \R_1^c$. By Proposition \ref{horofoliation}, the function $t \to d(\phi^{t}(w), \phi^{t}(v))$ is bounded above by some constant $C>0$. 
Then for each $t_n$ in \eqref{rec v}, we have for any $s \in [-t_{n},+\infty)$,
\begin{equation}\label{e:recc}
\begin{aligned}
&d(\phi^{t_{n}+s}(w),\phi^{t_{n}+s}(v))= d(d\alpha_{n}\circ\phi^{t_{n}+s}w,d\alpha_{n}\circ\phi^{t_{n}+s}v)\\
= & d(\phi^{s}\circ d\alpha_{n}\circ\phi^{t_{n}}w,\phi^{s}\circ d\alpha_{n}\circ\phi^{t_{n}}v).      
\end{aligned}
\end{equation}

Consider $s=0$ in \eqref{e:recc}. Note that $d(\phi^{t_{n}}(w),\phi^{t_{n}}(v))\leq C$, and $d\alpha_{n}(\phi^{t_{n}}v)\rightarrow v$.
So we know that for any $\epsilon>0$ there is an integer $N>0$ such that for any $n>N$, 
$$d(v,d\alpha_{n}(\phi^{t_{n}}w))\leq C+\epsilon.$$
This implies that the set $\{d\alpha_{n}(\phi^{t_{n}}w)\}$ has an accumulate point. Without loss of generality, we assume $d\alpha_{n}(\phi^{t_{n}}w)\rightarrow w'\in \R_1^c$ as $\R_1^c$ is a closed subset. 

Letting $n\rightarrow +\infty$ in \eqref{e:recc}, we get that
$$d(\phi^{s}(w'),\phi^{s}(v)) \leq  C, ~~~\forall~s \in \mathbb{R}.$$
Since $w'\neq v$, it follows that the geodesics $c_{v}$  and $c_{w'}$ bound a nontrivial generalized strip. Then $v\in \R_0^c\subset \R_1^c$ which contradicts to the assumption that $v\in \R_1$. We are done with the proof.
\end{proof}

Recall that for any $x\in X$ the bijection $f_x: S_xX\to \pX$ is defined by $f_x(v)=v^+, v\in S_xX$. Let $\tilde \mu_x:=(f_x^{-1})_*\mu_p$ which is a finite Borel measure on $S_xX$.
\begin{lemma}\label{null}
For any $x\in X$, we have $\tilde \mu_x(\R_1\cap S_xX)=\tilde \mu_x(S_xX)$.
\end{lemma}
\begin{proof}

Let $\Rec\subset SM$ be the subset of vectors recurrent under the geodesic flow. Then its lift to $SX$, which is also denoted by $\Rec$, has full $m$-measure. By assumption of Theorem \ref{margulis2} and Corollary C.1, $\R_1$ also has full $m$-measure, and thus $\Rec\cap \R_1$ has full $m$-measure. Define $R:=\{v^+: v\in \Rec\cap \R_1\}$. By definition of the Bowen-Margulis measure, we see that $R$ has full $\mu_x$-measure.

Let $v\in \Rec\cap \R_1$. By Lemma \ref{recreg}, for every $w\in \F^s(v)$, we have $w\in \R_1$. It follows that $R\cap \Sing^+=\emptyset$ where $\Sing^+:=\{v^+: v\in \R_1^c\}$. So $\mu_x(\Sing^+)=0$. The lemma then follows.
\end{proof}
Let $\F$ be a fundamental domain of $\C$. Fix $x,y\in \F\subset X$, and $p=x$ the reference point. For each regular vector $w\in S_{x}X\cap \R_1$, we can construct a local product flow box around $w$ as in Subsection 7.3. More precisely, consider the interior of $B_\theta^\a(w)$, $\text{int} B_\theta^\a(w)$, which is an open neighborhood of $w$ for some $\a>0$ and $0<\theta<\theta_0$ (here $\theta_0$ depends on $w$). By second countability of $S_xX$, there exist countably many regular vectors $w_1, w_2, \cdots$ such that $S_yX\cap \R_1 \subset \cup_{i=1}^\infty\text{int} B_{\theta_i}^\a(w_i)$.
Similarly, there exist countably many regular vectors $v_1, v_2, \cdots$ such that $S_yX\cap \R_1 \subset \cup_{i=1}^\infty\text{int} B_{\theta'_i}^\a(v_i)$. We note that the reference point $p$ is always chosen to be $x$ in the construction of all above flow boxes.

A \emph{regular partition-cover} of $S_xX$ is a triple $(\{w_i\}, \{\text{int} B_{\theta_i}^\a(w_i)\}, \{N_i\})$ where $\{N_i\}$ is a disjoint partition of $\R_1\cap S_xX$ and such that $N_i\subset \text{int} B_{\theta_i}^\a(w_i)$ for each $i\in\NN$. Similarly a regular partition-cover of $S_yX$ is a triple $(\{v_i\}, \{\text{int} B_{\theta'_i}^\a(v_i)\}, \{V_i\})$ such that  $\{V_i\}$ is a disjoint partition of $\R_1\cap S_yX$ and $V_i\subset \text{int} B_{\theta'_i}^\a(v_i)$ for each $i\in\NN$.

In the next subsection, we will first consider a pair of $V_i$ and $N_j$ from the regular partition-covers of $S_xX$ and $S_yX$ respectively. At last, we will sum up our estimates over countably many such pairs from  regular partition-covers.

Consider a pair of $V_i$ and $N_j$ from the regular partition-covers of $S_xX$ and $S_yX$ respectively. For simplicity, we just denote $V:=V_i$ and $N:=N_j$. Then $N\subset \text{int}B_{\theta_j}^\a(w_0)$ and $V\subset \text{int}B_{\theta'_i}^\b(v_0)$ for some $w_0\in S_xX$ and $v_0\in S_yX$. We also denote for $a>0$,
\begin{equation*}
\begin{aligned}
B_aN&:=\{v\in S_xX: d(v,N)\le a\}\\
B_{-a}N&:=\{v\in N: B(v,a)\subset N\}.
\end{aligned}
\end{equation*}
Write $t_0:=s(v_0)=b_{v_0^-}(\pi v_0,p)$ where $p=x$. We denote the flow boxes by
\begin{equation*}
\begin{aligned}
&N^{\a}:=H^{-1}(N^-\times N^+\times [-\a,\a]),\\
&V^\b:=H^{-1}(V^-\times V^+\times (t_0+[-\b,\b])).
\end{aligned}
\end{equation*}

Notice that $N^\a\subset \text{int}B_{\theta_j}^\a(w_0)$ and $V^\b\subset \text{int}B_{\theta'_i}^\b(v_0)$.
Given $\e>0$, we always consider $\frac{\e^2}{100}\le \a, \b\le \frac{3\e}{2}$.
By carefully adjusting the regular partition-covers, we can guarantee that
\begin{equation}\label{e:boundary}
\begin{aligned}
\mu_p(\partial V^+)=\mu_p(\partial V^-)=\mu_p(\partial N^+)=\mu_p(\partial N^-)=0.
\end{aligned}
\end{equation}

In the next subsection, we will count the number of elements in certain subsets of $\Gamma$. Let us collect the definitions here for convenience.
\begin{equation*}
\begin{aligned}
\C(t,\a,\b)&:=\{\c\in \C: N^\a\cap \phi^{-t}\c V^\b\neq \emptyset\},\\
\C_{-\rho}(t,\a,\b)&:=\{\c\in \C: (B_{-\rho}N)^\a\cap \phi^{-t}\c (B_{-\rho}V)^\b\neq \emptyset\},\\
\C^*&:=\{\c\in \C: \c V^+ \subset N^+ \text{\ and\ }\c^{-1}N^-\subset V^-\},\\
\C^*(t,\a,\b)&:=\C^*\cap \C(t,\a,\b).
\end{aligned}
\end{equation*}
The notation here is similar to that in Subsection 7.3. However, it should not cause any confusion.

\begin{lemma}\label{expansion}
For every $\rho>0$, there exists some $T_2> 0$ such that for all $t\ge T_2$, we have $\C_{-\rho}(t,\a,\b)\subset \C^*(t,\a,\b)$.
\end{lemma}
\begin{proof}
We apply Lemma \ref{piconvergence} with $\textbf{Q}=(B_{-\rho}V)^-$, $\textbf{R}=(B_{-\rho}N)^+$, $A=\pi\overline{(B_{-\rho}N)^\a\cup(B_{-\rho}V)^\b}$, and $\textbf{V}, \textbf{U}$ be the interiors of $N^+$ and $V^-$ respectively. Then $(B_{-\rho}N)^+\subset \textbf{V}$ and $(B_{-\rho}V)^-\subset \textbf{U}$. Let $T$ be given by Lemma \ref{piconvergence}. Then for every $t\ge T$ and $\c\in \C_{-\rho}(t,\a,\b)$, there exists $v\in (B_{-\rho}N)^\a$ such that $w=\c_*^{-1}\phi^tv\in (B_{-\rho}V)^\b$. Putting $z_1=\pi v\in A$, then 
$$\c^{-1}z_1=\pi\c_*^{-1}v=\pi \phi^{-t}w.$$
Since $\pi w\in A$, $w^-\in (B_{-\rho}V)^-=\textbf{Q}$, $t\ge T$, we see that $\c^{-1}z_1\in C_{\textbf{Q}}^T$. 

Similarly, putting $z_2=\pi w\in A$, we have
$$\c z_2=\pi\c_*w=\pi \phi^tv.$$
Since $\pi v\in A$, $v^+\in (B_{-\rho}N)^+=\textbf{R}$ and $t\ge T$, we see that $\c z_2\in  C_{\textbf{R}}^T$. The lemma then follows from Lemma \ref{piconvergence}.
\end{proof}

\subsection{Using scaling and mixing}
In this subsection, we use the scaling and mixing properties of Bowen-Margulis measure $m$, to give an asymptotic estimates of $\#\C^*(t,\e^2,\b)$ and $\#\C(t,\e^2,\b)$.

\subsubsection{Intersection components}
We need some delicate estimates on numbers of intersection components. Note that $N$ is on the unit tangent sphere $S_xX$, which is different from $H^{-1}(N^{-}\times N^+)\times \{0\}.$

The following is obtained in \cite[Lemmas 5.1 and 5.2]{Wu1}.
\begin{lemma}\label{nhd}
We have:
\begin{enumerate}
    \item $N\subset N^{\e^2}$, $V\subset V^{\e^2}$.
    \item For any $t>0$ and $\c\in\C$, 
$$\{\c\in \C: N^\e\cap \phi^{-t}\c V\neq \emptyset\}\subset \{\c\in \C: N^{\e^2}\cap \phi^{-t}\c V^{\e}\neq \emptyset\}.$$
\end{enumerate}
\end{lemma}

\begin{lemma}\label{enlarge}
Let $p\in X$, then given any $a>0,\e>0$, there exists $T>0$ such that for any $t\geq T$ and $v,w\in S_pX$, $d(\phi^tv,\phi^tw)\leq \e$ implies that $\angle(v,w)<a$.
\end{lemma}
\begin{proof}
The lemma is a direct consequence of uniform visibility property. Indeed, let $L(a)$ be the constant from uniform visibility property. Let $T$ be large enough so that $T-\e>L(a)$ and $t\ge T$. Then by the triangle inequality, the geodesic connecting $\phi^tv$ and $\phi^tw$ stays at distance at least 
$$2t-\e\ge 2T-\e>L(a)$$ 
from $p$. Then $\angle(v,w)<a$ by uniform visibility.
\end{proof}

Based on Lemmas \ref{nhd}(1), \ref{diameter}, \ref{enlarge}, by a proof parallel to \cite[Lemma 5.4]{Wu1}, we have
\begin{lemma}\label{intersect2}
For any $a>0$, there exists $T_1>0$ large enough such that for any $t\ge T_1$,
\begin{equation*}
\begin{aligned}
&\{\c\in \C: N^\a\cap \phi^{-t}\c V^\b\neq \emptyset\}\\
\subset & \{\c\in\C: (B_aN)^{\a+\b+\e^2}\cap \phi^{-t}\c V\neq \emptyset\}.
\end{aligned}
\end{equation*}
\end{lemma}

\subsubsection{Depth of intersection}

Given $\xi\in \pX$ and $\c\in \C$, define $b_\xi^\c:=b_\xi(\c p, p)$.
Using Corollary \ref{equicon}, we have
\begin{lemma}\label{intersection11}(\cite[Lemma 5.5]{Wu1})
Let $\xi,\eta\in N^-$, and $\c\in \C(t,\a,\b)$ with $t>0$. Then $|b_\xi^\c-b_\eta^\c|<\e^2$.
\end{lemma}

The following two lemmas give estimate on the depth of intersections, and implies that the intersection components also have product structure.
\begin{lemma}\label{intersection31}(\cite[Lemma 5.6]{Wu1})
Given any $\c\in \C^*(t,\a,\b)$ and any $t\in \RR$, we have
\begin{equation*}
\begin{aligned}
N^{\e^2}\cap \phi^{-t}\c V^\b=\{&w\in E^{-1}(N^-\times \c V^+): \\
&s(w)\in [-\e^2,\e^2]\cap (b_{w^{-}}^\c-t+t_0+[-\b,\b])\}
\end{aligned}
\end{equation*}
where $t_0=s(v_0)$.
\end{lemma}

\begin{lemma}\label{depth22}(\cite[Lemma 5.7]{Wu1})
If $\c\in \C^*(t,\e^2,\b)$, then
$$N^{\e^2}\cap \phi^{-(t+4\e^2)}\c V^{\b+8\e^2}\supset H^{-1}(N^-\times \c V^+\times [-\e^2,\e^2]):=N^\c.$$
\end{lemma}

\subsubsection{Scaling and mixing calculation}
The following result is obtained using the product structure of Bowen-Margulis measure $m$. It says that the measure of an intersection component decreases with a factor $e^{-ht}$ due to the contraction of $\phi^{-t}$ in the unstable direction.
\begin{lemma}\label{tscaling22}(\cite[Lemma 5.8]{Wu1})
If $\c\in \C^*(t,\e^2,\b)$, then
$$\frac{m(N^\c)}{m(V^\b)}=e^{\pm 26h\e}e^{ht_0}e^{-ht}\frac{\e^2\mu_p(N^-)}{\b\mu_p(V^-)}$$
where $N^\c$ is from Lemma \ref{depth22}.
\end{lemma}

Finally, we can combine scaling and mixing properties of Bowen-Margulis measure to obtain the following asymptotic estimates.
\begin{proposition}\label{asymptotic1}(\cite[Proposition 5.9]{Wu1})
We have
\begin{equation*}
\begin{aligned}
e^{-30h\e} \lesssim &\frac{\#\C^*(t,\e^2,\b)e^{ht_0}}{\mu_p(V^-)\mu_p(N^+)e^{ht}}\frac{1}{2\b}
\lesssim e^{30h\e}(1+\frac{8\e^2}{\b}),\\
e^{-30h\e} \lesssim &\frac{\#\C(t,\e^2,\b)e^{ht_0}}{\mu_p(V^-)\mu_p(N^+)e^{ht}}\frac{1}{2\b}
\lesssim e^{30h\e}(1+\frac{8\e^2}{\b}).
\end{aligned}
\end{equation*}
\end{proposition}

\subsubsection{Integration}
Let $V\subset S_yX, N\subset S_xX$ be as above, and $0\le a<b$. Let $n(a,b,V,N^0)$ denote the number of connected components at which $\phi^{[-b,-a]}\lV$ intersects $\lN^0$, and $n_t(V^\b,N^\a)$ (resp. $n_t(V,N^\a)$) denote the number of connected components at which $\phi^{-t}\lV^\b$  (resp. $\phi^{-t}\lV$) intersects $\lN^\a$.

Applying  Lemma \ref{nhd}(2) and Proposition \ref{asymptotic1}, we have
\begin{lemma}\label{asy1}
$n_t(V,N^\e)\lesssim e^{-ht_0}\mu_p(V^-)\mu_p(N^+)e^{ht} 2\e (1+8\e)e^{30 h\e}.$
\end{lemma}
\begin{proof}
Applying Lemma \ref{nhd}(2) with $\a=\e^2$, we have
$$n_t(V,N^{\e})\le n_t(V^\e,N^{\e^2}).$$
Letting $\b=\e$ in Proposition \ref{asymptotic1}, we have
\begin{equation*}
\begin{aligned}
&n_t(V^\e,N^{\e^2})=\#\C(t,\e^2,\e)
\lesssim  &e^{-ht_0}\mu_p(V^-)\mu_p(N^+)e^{ht} 2\e (1+8\e)e^{30 h\e}.
\end{aligned}
\end{equation*}
\end{proof}

\begin{lemma}\label{asy2}
One has
$$n_t(V,N^\e)\gtrsim e^{-ht_0}\mu_p(V^-)\mu_p(N^+)e^{ht} 2\e (1-2\e)e^{-30 h\e}.$$
\end{lemma}
\begin{proof}
Setting $\a=\e^2$, $\b=\e-2\e^2$ in Lemma \ref{intersect2}, for any $a>0$, we have
$$n_t(V^{\e-2\e^2},(B_{-a}N)^{\e^2})\le n_t(V,N^{\e})$$
for any $t\ge T_1$ where $T_1$ is provided by Lemma \ref{intersect2}.
Applying Proposition \ref{asymptotic1} with $\b=\e-2\e^2$, we have
\begin{equation*}
\begin{aligned}
n_t(V^{\e-2\e^2},(B_{-a}N)^{\e^2})\gtrsim  &e^{-ht_0}\mu_p(V^-)\mu_p((B_{-a}N)^+)e^{ht}2\e(1-2\e) e^{-30 h\e}.
\end{aligned}
\end{equation*}
Letting $a\to 0$, by \eqref{e:boundary} we obtain the conclusion of the lemma.
\end{proof}

Combining Lemmas \ref{asy1} and \ref{asy2}, using standard integration technique, we have
\begin{proposition}\label{key}
There exists $Q>0$ such that
\begin{equation*}
\begin{aligned}
e^{-2Q\e}e^{-ht_0}\mu_p(V^-)\mu_p(N^+)\frac{1}{h}e^{ht} &\lesssim n(0,t,V,N^0)\\
&\lesssim  e^{2Q\e}e^{-ht_0}\mu_p(V^-)\mu_p(N^+)\frac{1}{h}e^{ht}.
\end{aligned}
\end{equation*}
\end{proposition}

\subsection{Summing over the regular partition-cover}
Denote $$a_t(x,y):=\#\{\c\in \C: \c y\in B(x,t)\}$$
and $$a^{1}_t(x,y):=\#\{\c\in \C: \c y\in B(x,t), \text{\ and\ }\dot c_{x,\c y}(0)\in \R_1^c\}.$$
It is easy to see that $b_t(x)=\int_\F a_{t}(x,y)d  \text{Vol} (y)$. 

The following shows that $a_t^1(x,y)$ is not dominating in the asymptotic estimates. The proof essentially uses the entropy gap.
\begin{lemma}\label{singular}
There exist $C>0$ and $0<h'<h$, such that for any $x,y\in \F$,
$$a^{1}_t(x,y) \le Ce^{h't}.$$
\end{lemma}
\begin{proof}
Given any $\e<\inj M/5$ and $t>0$, let $\c_1\neq \c_2\in \C$ be such that $\c_1y,\c_2 y\in B(x,t)\setminus B(x,t-\e)$, and $\dot c_{x,\c_1 y}(0), \dot c_{x,\c_2 y}(0)\in \R_1^c$. Then it is easy to see that $\dot c_{x,\c_1 y}(0)$ and $\dot c_{x,\c_2 y}(0)$ are $(t,\e)$-separated. By the uniqueness of MME and the fact $m(\R_1^c)=0$ (see Corollary C.1), we have the entropy gap: $h_\top(\R_1^c)<h$. It follows that the number of $\c\in \C$ as above is less than $C_1e^{h't}$ for some $C_1>0$ and $h_\top(\R_1^c)< h'<h$.

Let $t_i=i \e$, then
$$a^{1}_t(x,y) \le \sum_{i=1}^{[t/\e]+1} C'e^{h't_i}= \frac{C'}{\e}\sum_{i=1}^{[t/\e]+1}\e e^{h't_i}\le \frac{C'}{\e}\int_{0}^{t+\e}e^{h's}ds\le Ce^{h't}$$
for some $C>0$.
\end{proof}

Denote by
$$a(0,t, x,y, \cup_{j=n+1}^\infty N_j):=\#\{\c\in \C: \c y\in \pi\phi^{[0,t]}\cup_{j=n+1}^\infty N_j\},$$
and similarly,
$$a(0,t, x,y, \cup_{i=m+1}^\infty V_i):=\#\{\c\in \C: \c^{-1} x\in \pi\phi^{[-t,0]}\cup_{i=m+1}^\infty V_i\}.$$
Applying the shadow lemma, Proposition \ref{shadow}(3), we have the following inequality:
\begin{lemma}\label{nonuniform}(\cite[Lemma 5.14]{Wu1})
There exists $C>0$ such that
$$\limsup_{t\to \infty}e^{-ht}a(0,t, x,y, \cup_{j=n+1}^\infty N_j)\le C\cdot \mu_p((\cup_{j=n+1}^\infty N_j)^+),$$
and
$$\limsup_{t\to \infty}e^{-ht}a(0,t, x,y, \cup_{i=m+1}^\infty V_i)\le C\cdot \mu_p((\cup_{i=m+1}^\infty V_i)^-).$$
\end{lemma}

\begin{proof}[Proof of Theorem \ref{margulis2}]
Since the diameter of each flow box is no more than $4\e$, we have
\begin{equation}\label{e:less1}
\begin{aligned}
n(0,t, \cup V_i, \cup (N_j)^0)\le a_{t+4\e}(x,y)
\end{aligned}
\end{equation}
and
\begin{equation}\label{e:less2}
\begin{aligned}
a_{t-4\e}(x,y)&\le a^{1}_{t-4\e}(x,y)+ a(0,t, x,y, \cup_{j=n+1}^\infty N_j)\\
&+a(0,t, x,y, \cup_{i=m+1}^\infty V_i)+ n(0,t, \cup_{i=1}^m V_i, \cup_{j=1}^n (N_j)^0).
\end{aligned}
\end{equation}

For each $V_i$, denote by $t_0^i:=b_{v_i^-}(\pi v_i,p)$ where $v_i\in V_i$. Recall that in Subsection 8.1 we suppressed $i$ and write $t_0=t_0^i$, since only one $V_i$ is considered there. By Proposition \ref{key}, for any $m,n\in \NN$
\begin{equation*}
\begin{aligned}
&\liminf_{t\to \infty}e^{-ht}n(0,t, \cup V_i, \cup (N_j)^0)\\
\ge&\liminf_{t\to \infty}e^{-ht}\sum_{i=1}^m\sum_{j=1}^n n(0,t, V_i,  (N_j)^0)\\
\ge &\sum_{i=1}^m\sum_{j=1}^n e^{-2Q\e}e^{-hb_{v_i^-}(\pi v_i,p)}\mu_p((V_i)^-)\mu_p((N_j)^+)\frac{1}{h}.
\end{aligned}
\end{equation*}
Note that $v\mapsto b_{v^-}(\pi v,p)$ is a continuous function by Corollary \ref{continuous}. So if we choose a sequence of finer and finer regular partition-covers, and let $m,n\to \infty$ on the right hand,
\begin{equation*}
\begin{aligned}
&\liminf_{t\to \infty}e^{-ht}n(0,t, \cup V_i, \cup (N_j)^0)\\
\ge &e^{-2Q\e}\frac{1}{h}\int_{S_xM\cap \R_1}\int_{S_y M\cap \R_1}e^{-hb_{v^-}(\pi v,p)}d\tilde\mu_y(-v)d\tilde\mu_x(w).
\end{aligned}
\end{equation*}
Then by \eqref{e:less1} we have
\begin{equation*}
\begin{aligned}
a_{t+4\e}(x,y)&\gtrsim e^{-2Q\e}\frac{1}{h}e^{ht} \int_{S_xM\cap \R_1}\int_{S_y M\cap \R_1}e^{-hb_{v^-}(\pi v,p)}d\tilde\mu_y(-v)d\tilde\mu_x(w).
\end{aligned}
\end{equation*}
Replacing $t$ by $t-4\e$, we have
\begin{equation}\label{e:ge}
\begin{aligned}
a_{t}(x,y)&\gtrsim e^{-2Q\e}e^{-4\e}\frac{1}{h}e^{ht} \int_{S_xM\cap \R_1}\int_{S_y M\cap \R_1}e^{-hb_{v^-}(\pi v,p)}d\tilde\mu_y(-v)d\tilde\mu_x(w).
\end{aligned}
\end{equation}

On the other hand, by Proposition \ref{key} and Corollary \ref{equicon}, for any $m,n\in \NN$
\begin{equation}\label{e:upper}
\begin{aligned}
&e^{2Q\e}\frac{1}{h} \int_{S_xM\cap \R_1}\int_{S_y M\cap \R_1}e^{-hb_{v^-}(\pi v,p)}d\tilde\mu_y(-v)d\tilde\mu_x(w)\\
\ge &\sum_{i=1}^m\sum_{j=1}^n e^{2Q\e}e^{-h(t_0^i+\e^2+4\e)}\mu_p((V_i)^-)\mu_p((N_j)^+)\frac{1}{h}\\
\ge &\limsup_{t\to \infty}\sum_{i=1}^m\sum_{j=1}^n n(0,t, V_i,  (N_j)^0)e^{-ht}e^{-h(\e^2+4\e)}.
\end{aligned}
\end{equation}
Combining with \eqref{e:less2} and Lemmas \ref{singular}, \ref{nonuniform}, one has
\begin{equation*}
\begin{aligned}
&a_{t-4\e}(x,y)\lesssim  Ce^{h'(t-4\e)}\\+&Ce^{h(t-4\e)}\cdot \mu_p((\cup_{j=n+1}^\infty N_j)^+)+ Ce^{h(t-4\e)}\cdot \mu_p((\cup_{j=m+1}^\infty V_i)^+)\\
+&e^{2Q\e}e^{h(\e^2+4\e)}\frac{1}{h}e^{ht} \int_{S_xM\cap \R_1}\int_{S_y M\cap \R_1}e^{-hb_{v^-}(\pi v,p)}d\tilde\mu_y(-v)d\tilde\mu_x(w).
\end{aligned}
\end{equation*}
Letting $m,n\to \infty$ and replacing $t-4\e$ by $t$, we have
\begin{equation}\label{e:le}
\begin{aligned}
&a_{t}(x,y)\\
\lesssim &e^{2Q\e}e^{h(\e^2+4\e)+4\e}\frac{1}{h}e^{ht} \int_{S_xM\cap \R_1}\int_{S_y M\cap \R_1}e^{-hb_{v^-}(\pi v,p)}d\tilde\mu_y(-v)d\tilde\mu_x(w).
\end{aligned}
\end{equation}

Letting $\e\to 0$ in \eqref{e:ge} and \eqref{e:le} and recalling that $p=x$, we get
$$a_{t}(x,y)\sim \frac{1}{h}e^{ht}\cdot c(x,y)$$
where
$$c(x,y):=\int_{S_xX\cap \R_1}\int_{S_y X\cap \R_1 }e^{-hb_{v^-}(\pi v,x)}d\tilde\mu_y(-v)d\tilde\mu_x(w).$$
By Lemma \ref{null}, in fact we have
$$c(x,y)=\int_{S_xX}\int_{S_y X}e^{-hb_{v^-}(\pi v,x)}d\tilde\mu_y(-v)d\tilde\mu_x(w).$$

It follows that
$b_t(x)=\int_\F a_{t}(x,y)d  \text{Vol} (y)\sim  \frac{1}{h}e^{ht}\int_\F c(x,y)d  \text{Vol}(y)$ by the dominated convergence Theorem. Indeed, by \eqref{e:less2}, Lemma \ref{singular}, \cite[Equation (10)]{Wu1} and \eqref{e:upper}, there exist contants $B_1, B_2>0$ such that
$$e^{-ht}a_{t}(x,y)\le B_1+B_2c(x,y),$$
and the right hand side is clearly integrable.

Define $c(x):=\int_\F c(x,y)d  \text{Vol}(y)$, we get $b_t(x)\sim c(x)\frac{e^{ht}}{h}$.
Obviously, $c(\c x)=c(x)$ for any $\c\in \C$. So $c$ descends to a function from $M$ to $\RR$, which still denoted by $c$.

It remains to prove the continuity of $c$. Let $y_n\to x\in X$. For any $\e>0$, there exists $n\in \NN$ such that $d(x,y_n)<\e$. Then
$$b_{t-\e}(x)\le b_{t}(y_n)\le b_{t+\e}(x).$$
We have
\begin{equation*}
\begin{aligned}
c(y_n)=\lim_{t\to \infty}\frac{b_t(y_n)}{e^{ht}/h}\le \lim_{t\to \infty}\frac{b_{t+\e}(x)}{e^{ht}/h}=c(x)e^{h\e}
\end{aligned}
\end{equation*}
and
\begin{equation*}
\begin{aligned}
c(y_n)=\lim_{t\to \infty}\frac{b_t(y_n)}{e^{ht}/h}\ge \lim_{t\to \infty}\frac{b_{t-\e}(x)}{e^{ht}/h}=c(x)e^{-h\e}.
\end{aligned}
\end{equation*}
Thus $\lim_{n\to \infty}c(y_n)=c(x)$ and $c$ is continuous.
\end{proof}

\section{Margulis function and rigidity}

\subsection{Properties of Margulis function}
We will study the regularity properties of Margulis function and prove the integral formula for topological entropy in Theorem \ref{marfunction} in this section. Let $M$ be a closed $C^\infty$ uniform visibility manifold without conjugate points and with continuous asymptote, $X$ the universal cover of $M$. Suppose that the geodesic flow has a hyperbolic periodic point.

It is easy to see that
\begin{equation}\label{e:sphere}
s_t(x)/e^{ht}\sim c(x),
\end{equation}
where $s_t(x)$ is the spherical volume of the sphere $S(x,t)$ around $x\in X$ of radius $t>0$. Based on the Margulis function, we give an equivalent definition for the Patterson-Sullivan measure as a limit of spherical volume measure on spheres as follows. 

Recall that $f_x: S_xX\to \partial X$, $f_x(v)=v^+$. Similarly, we define the canonical projection $f_x^R: S(x,R)\to \partial X$ by $f_x^R(y)=v_y^+$ where $v_y$ is the unit normal vector of the sphere $S(x,R)$ at $y$. For any continuous function $\varphi: \partial X \to \RR$, define a measure on $\pX$ by
$$\nu_x^R(\varphi):=\frac{1}{e^{hR}}\int_{S(x,R)}\varphi \circ f_x^R(y)d \text{Vol} (y).$$
By \eqref{e:sphere}, $\nu_x^R(\partial X)$ is uniformly bounded from above and below, and hence there exist limit measures of $\nu_x^R$ when $R\to \infty$. Take any limit measure $\nu_x$. By \eqref{e:sphere}, we see that
$$\nu_x(\partial X)=\lim_{R\to \infty}\frac{s_R(x)}{e^{hR}}=c(x).$$
Moreover, by definition one can check that
\begin{enumerate}
\item For any $p,\ q \in X$ and $\nu_{p}$-a.e. $\xi\in \pX$,
$$\frac{d\nu_{q}}{d \nu_{p}}(\xi)=e^{-h \cdot b_{\xi}(q,p)}.$$
\item $\{\nu_{p}\}_{p\in X}$ is $\Gamma$-equivariant, i.e., for every Borel set $A \subset \pX$ and for any $\c \in \Gamma$,
we have $$\nu_{\c p}(\c A) = \nu_{p}(A).$$
\end{enumerate}
(2) is obvious. See \cite{Wu1} for the proof of (1). Thus $\{\nu_x\}_{x\in X}$ is a $h$-dimensional Busemann density.

By the following theorem, $\{\nu_x\}_{x\in X}$ coincide with the Patterson-Sullivan measure $\{\mu_x\}_{x\in X}$ up to a scalar constant. So we do not distinguish them in the following discussion.

\begin{theorem}\label{buse}
Let $M=X/\Gamma$ be a compact uniform visibility manifold without conjugate points. Then up to a multiplicative constant, the Busemann density is unique, i.e., the Patterson-Sullivan measure is the unique Busemann density.
\end{theorem}
\begin{proof}
We sketch the idea of the proof, following closely \cite[Page 764-767]{Kn1}.

\textbf{Step 1:} Fix a reference point $p_{0}\in X$. For an arbitrary point $p \in X \setminus p_{0}$, let $\xi=c_{p_{0},p}(+\infty)$. For any positive number $\rho > R$ (where $R$ is the constant in Proposition \ref{shadow}), we call the set
$$B^{\rho}_{r}(\xi):=pr_{p_{0}}(B(p,\rho)\cap S(p_{0},d(p,p_{0})))$$
a ball of radius $r=e^{-d(p,p_{0})}$ at $\xi\in \pX$. By the triangle inequality we know $pr_{p_{0}}B(p,\rho) \subseteq B^{2\rho}_{r}(\xi)$.
Then by Proposition \ref{shadow}(3), for $\rho > 2R$, there exists a constant number $b(\rho) > 1$ such that
$$\frac{1}{b(\rho)}e^{-h\cdot d(p,p_{0})} \leq \mu_{p_{0}}(B^{\rho}_{r}(\xi)) \leq b(\rho)e^{-h\cdot d(p,p_{0})}.$$

\textbf{Step 2:}  For any subset $A \subseteq \pX$ we can define the $s$-dimensional Hausdorff measure $H^{s}(A)$ of $A$
and Hausdorff dimension $Hd(A)$ of $A$ as in \cite[Definition 4.5]{Kn1}.
Then we can show that there exists a constant number $c > 1$ satisfying the following estimate
\begin{equation}\label{haus}
 \frac{1}{c}\mu_{p_{0}}(A) \leq H^{h}(A)\leq c\cdot \mu_{p_{0}}(A),  ~ \forall ~\mbox{Borel~set} ~A\subseteq \pX.   
\end{equation}
Then the Hausdorff dimension of $\pX$ equals $h$. Moreover, each Busemann density is $h$-dimensional.

There is a fundamental difference in the proof of the above estimations in no conjugate points case. In the proof of a type of Vitali covering lemma \cite[Lemma 4.7]{Kn1}, Knieper used the convexity property in nonpositive curvature, namely, the distance function of two geodesics is convex. The convexity is no longer valid for manifolds without conjugate points. Under assumption of uniform visibility, we can use quasi-convexity in Proposition \ref{geo} to replace convexity. Then we obtain the following
\begin{lemma}
Let $U\subset \pX$ be an open set. For fixed $\rho>0$ and $\e >0$ there exist a constant $C(\rho)>0$ and a countable covering of $U$ by balls
$$B^{C(\rho)}_{r_j}(\xi_j), \quad 0<r_j\le \e$$
such that the balls $B^{\rho/3}_{r_j}(\xi_j)$ are subsets of $U$ and pairwise disjoint.
\end{lemma}
Here $C(\rho)=A\rho+B+\rho$ were $A,B$ are constants from quasi-convexity. One can see that replacing $\rho$ by $C(\rho)$ does not affect the proof of \eqref{haus} (\cite[Proof of Theorem 4.6]{Kn1}).

\textbf{Step 3:} Finally, we can prove the Patterson-Sullivan measure is ergodic with respect to the $\C$-action on $\pX$ as \cite[Proposition 4.9]{Kn1}. Then up to a multiplicative constant, there exists exactly one Busemann density which is realized by the Patterson-Sullivan measure, see \cite[Theorem 4.10]{Kn1}.
\end{proof}

\begin{proof}[Proof of Theorem \ref{marfunction}]
Recall that $\tilde \mu_x$ is the normalized Patterson-Sullivan measure. By the above discussion, particularly $\nu_x(\partial X)=c(x)$ and $\frac{d\nu_{y}}{d \nu_{x}}(\xi)=e^{-h \cdot b_{\xi}(y,x)}$, we have
$$\frac{d\bar \mu_{y}}{d \bar\mu_{x}}(\xi)=\frac{c(x)}{c(y)}e^{-h \cdot b_{\xi}(y,x)}.$$
Then $c(y)=c(x)\int_{\pX}e^{-h \cdot b_{\xi}(y,x)} d\bar \mu_x(\xi).$

By Lemma \ref{Buseregu}, $y\mapsto b_{\xi}(y,x)$ is $C^2$. Moreover, $\triangle_yb_{\xi}(y,x)=\text{div} (-X)=tr U(y,\xi)$ and it is continuous in $\xi$. It follows that $c$ is $C^1$.

If $c$ is constant, then $\int_{\pX}e^{-h \cdot b_{\xi}(y,x)} d\bar \mu_x(\xi)\equiv 1$. Taking Laplacian with respect to $y$ on both sides, we have
$$\int_{\pX}h(h-tr U(y,\xi))e^{-h \cdot b_{\xi}(y,x)} d\bar \mu_x(\xi)\equiv 0.$$
It follows that
$$h=\frac{\int_{\pX} tr U(y,\xi)e^{-h \cdot b_{\xi}(y,x)} d\bar \mu_x(\xi)}{\int_{\pX} e^{-h \cdot b_{\xi}(y,x)} d\bar \mu_x(\xi)}=\int_{\pX} tr U(y,\xi)d\bar \mu_y(\xi)$$
 for any $y\in X$.
\end{proof}

\subsection{Rigidity in dimension two}
Let $M$ be a closed Riemannian surface of genus at least $2$ without conjugate points and with continuous Green bundle in this subsection. 

\subsubsection{Unique ergodicity of horocycle flow}
A \textit{horocycle flow} is a continuous flow $h_s$ on $SM$ whose orbits are horocycles, i.e., for $v\in SM$, $\{h_sv: s\in \RR\}=\F^s(v),$ where $\F^s(v)$ is the strong stable horocycle manifold of $v$ in $SM$. Clotet  \cite{Cl2} proved recently the unique ergodicity of the horocycle flow, see also \cite{Cl} for the case of nonpositive curvature. 

If $M$ has constant negative curvature, Furstenberg \cite{Fur} proved the unique ergodicity of the horocycle flow, which is extended to compact surfaces of variable negative curvature by Marcus \cite{Ma}.
To apply Marcus's method to our case, we need to define the horocycle flow using the so-called \emph{Margulis parametrization}.  Gelfert-Ruggiero  \cite{GR} defined a quotient map $\chi: SM\to Y$ by an equivalence relation ``collapsing'' each generalized strip to a single curve, which semiconjugates the geodesic flow on $SM$ and a continuous flow $\psi^t$ on $Y$. See Theorem \ref{expansivefactor} (i.e., \cite[Theorem 1.1]{MR}) for this result in a more general setting. Gelfert-Ruggiero  \cite{GR} showed that $Y$ is a topological $3$-manifold, and that the quotient flow is expansive, topologically mixing and has local product structure. In \cite{Cl2}, the the horocycle flow with Margulis parametrization $h_s^M$ is defined on $Y$. Using Coud\`{e}ne's theorem \cite{Cou}, it is showed in \cite[Proposition 4.2]{Cl2} that the horocycle flow $h_s^M$ on $Y$ is uniquely ergodic, and the unique invariant measure is $\chi_*m$, the projection of Bowen-Margulis measure $m$ onto $Y$.

There is another natural parametrization of the horocycle flow on $SM$ given by arc length of the horocycles, which clearly is well defined everywhere.
It is called the \emph{Lebesgue parametrization} and the \textit{Lebesgue horocycle flow} is denoted by $h_s^L$.
By constructing complete transversals to respective flows, it is showed in \cite[Theorem 5.8]{Cl2} that there is a bijection between finite Borel measures invariant under $h_s^M$ on $Z$ and $h_s^L$ on $SM$ respectively.
Then $h_s^L$ is also uniquely ergodic \cite[Theorem 5.10]{Cl2}. We denote by $w^s$ the unique probability measure invariant under the Lebesgue horocycle flow $h_s^L$.

Let $v\in \R_1$. Recall that 
\begin{equation*}
\begin{aligned}
\F^{0u}(v):=&\{w\in SX: w^-=v^-\}\\
\F^u(v):=&\{w\in \F^{0u}(v): b_{v^-}(\pi w, \pi v)=0\}
\end{aligned}
\end{equation*}
are the weak and strong unstable horocycle manifold of $v$ respectively. 
Put $\F^{0s}(v):=-\F^{0u}(-v)$ and $\F^{s}(v):=-\F^{u}(-v)$.
Then $\F^{0u}(v)$ contains a relatively compact neighborhood $T$ of $v$ in $\R_1$, which is a \textit{complete transversal} to the Lebesgue horocycle flow $h_s^L$ in the sense of \cite[Defintion 5.5]{Cl2}. So locally for a subset $E$ in a regular neighborhood of $v\in \R_1$,
$$w^s(E)=\int_{\F^{0u}(v)}\int_\RR 1_E(h_s^L(u))dsd\mu_{\F^{0u}(v)}(u)$$
where $1_E$ is the characteristic function of $E$, and $\mu_{\F^{0u}(v)}$ is some Borel measure on $T\subset \F^{0u}(v)$ which is in fact independent of the parametrization of the horocycle flow. Note that $\chi$ is a homeomorphism in a neighborhood of $v\in \R_1$.

On the other hand, the unique invariant measure for $h_s^M$ on $Y$ is the projection of Bowen-Margulis measure $m$, which can be expressed as
$$m(E)=\int_\pX\int_\pX \int_\RR 1_E(\xi,\eta,t)e^{h\cdot \beta_{p}(\xi,\eta)}dtd\mu_p(\xi)d\mu_p(\eta)$$
since $E$ contains no generalized strips.
Consider the canonical projection $$P=P_v: \F^u(v)\to \pX, \quad P(w)=w^+,$$
then $$\mu_{\F^{0u}(v)}(A)=\int \int_\RR 1_A(\phi^tu)dt d\mu_{\F^{u}(v)}(u)$$
where $\mu_{\F^{u}(v)}(B):=\int_\pX 1_{B}(P_v^{-1}\eta)e^{-hb_\eta(\pi v, p)}d\mu_p(\eta)$ and $A,B$ are in a regular neighborhood of $v\in \R_1$.

Thus $w^s$ is locally equivalent to the measure $ds \times dt \times d\mu_{\F^u}$. If we disintegrate $w^s$ along $\F^u$ foliation, the factor measure on a section $\F^{0s}(v)$ is equivalent to $ds dt$, i.e., the Lebesgue measure $\text{Vol}$, and the conditional measures on the fiber $\F^u(v)$ is equivalent to $P_v^{-1}\mu_p$.

\subsubsection{Uniqueness of harmonic measure}
Let $\mathcal{G}$ be any foliation on a compact Riemannian manifold $N$ with $C^2$ leaves. A probability measure $\nu$ on $N$ is called \emph{harmonic} with respect to $\mathcal{G}$ if $\int_N \triangle^L fd\nu=0$ for any bounded measurable function $f$ on $N$ which is smooth in the leaf direction, where $\triangle^L$ denotes the Laplacian in the leaf direction.

A \emph{holonomy invariant measure} of the foliation $\mathcal{G}$ is a family of measures defined on each transversal of $\mathcal{G}$, which is invariant under all the canonical homeomorphisms of the holonomy pseudogroup \cite{Pl}. A measure is called \emph{completely invariant} with respect to $\mathcal{G}$ if it disintegrates to a constant function times the Lebesgue measure on the leaf, and the factor measure is a holonomy invariant measure on a transversal. By \cite{Ga}, a completely invariant measure must be a harmonic measure. See \cite{Yue} for more details on the ergodic properties of foliations. 

Let $(M,g)$ be a closed $C^\infty$ uniform visibility manifold without conjugate points and with continuous asymptote, $(X, \tilde g)$ the universal cover of $(M,g)$. Suppose that the geodesic flow has a hyperbolic periodic point. We can identify any weak stable manifold $\F^{0s}(v), v\in SX$ with $X\times \{v^+\}$ and endow $\F^{0s}(v)$ a Riemannian structure induced by $(X,\tilde g)$. Then $\F^{s}(v)$ is identified with the horosphere $H(v)\subset X$. Since $X$ has continuous asymptote, by Proposition \ref{horofoliation} every horosphere $H(v)$ is a $C^2$ submanifold in $X$ (see also the proof of Theorem \ref{marfunction}). So $\F^{s}(v)$ is endowed with a $C^2$ Riemannian structure. Moreover, for any $\varphi\in C(X,\RR)$,
$$\triangle^{0s} \varphi=\bigtriangleup^s\varphi+\ddot{\varphi}-tr U \dot\varphi$$
by \cite[Lemma 5.1]{Yue3}, where $\bigtriangleup^{0s}$ is the Laplacian along $X$ (and hence along $\F^{0s}$-leaves) and $\bigtriangleup^{s}$ is the Laplacian along the horospheres (and hence along $\F^{s}$-leaves). 
\begin{theorem}\label{unique}
Let $M$ be a closed Riemannian surface of genus at least $2$ without conjugate points and with continuous Green bundle. Then there is precisely one harmonic probability measure with respect to the strong stable horocycle foliation.
\end{theorem}
\begin{proof}
If $\dim M=2$, then the leaves of the strong stable horocycle foliation have polynomial volume growth. By \cite{Kai}, any harmonic measure must be completely invariant. From the discussion in the previous subsection, we see that there is a unique completely invariant measure $w^s$. As a consequence, $w^s$ is the unique harmonic measure.
\end{proof}

\subsubsection{Integral formulas for topological entropy}
Recall that $M$ is a closed Riemannian surface of genus at least $2$ without conjugate points and with continuous Green bundle. Using the measure $w^s$ we can establish some formulas for topological entropy $h$ of the geodesic flow.

Let $B^s(v,R)$ denote the ball centered at $v$ of radius $R>0$ inside $\F^s(v)$. In fact, it is just a curve. By the uniqueness of  harmonic measure $w^s$, we have
\begin{lemma} \label{uniform}
For any continuous $\varphi: SM\to \RR$,
$$\frac{1}{\text{Vol}(B^s(v,R))}\int_{B^s(v,R)}\varphi d\text{Vol}(y)\to \int_{SM}\varphi dw^s$$
as $R\to \infty$ uniformly in $v\in SM$.
\end{lemma}
\begin{proof}
The proof is analogous to that of \cite[Theorem 1.2]{Yue} and thus is omitted here.
\end{proof}

For continuous $\varphi: SM\to \RR$, define $\varphi_x: X\to \RR$ by $\varphi_x(y)=\varphi(v(y))$ where $v(y)\in SX$ is the unique vector such that $c_{v(y)}(0)=y$ and $c_{v(y)}(t)=x$ for some $t\ge 0$.

Based on Lemma \ref{uniform}, we get the following proposition. 
\begin{proposition}\label{uniform1}
For any continuous $\varphi: SM\to \RR$,
$$\frac{1}{s_R(x)}\int_{S(x,R)}\varphi_x(y) d\text{Vol}(y)\to \int_{SM}\varphi dw^s$$
as $R\to \infty$ uniformly in $x\in X$.
\end{proposition}
\begin{proof}
The proof is the same as the one before \cite[Proposition 3.1]{Yue} (see also \cite{Kn4, Led}), and hence will be skipped. The basic idea here is that horospheres in $X$ can be approximated by geodesic spheres.
\end{proof}
\begin{theorem}\label{formula}
Let $M$ be a closed Riemannian surface of genus at least $2$ without conjugate points and with continuous Green bundle. Then
\begin{enumerate}
  \item $h=\int_{SM} tr U(v) dw^s(v)$,
  \item $h^2=\int_{SM} -tr \dot{U}(v)+(tr U(v))^2 dw^s(v)$,
  \item $h^3=\int_{SM}tr \ddot{U}-3tr \dot U tr U+(tr U)^3 d w^s$,
\end{enumerate}
where $U(v)$ and $tr U(v)$ are the second fundamental form and the mean curvature of the horocycle $H_{\pi v}(v^+)$ at $\pi v$.
\end{theorem}
\begin{proof}
Consider the following function
$$G_x(R):=\frac{s_R(x)}{e^{hR}}=\frac{1}{e^{hR}}\int_{S(x,R)}d\text{Vol} (y).$$
Taking the derivatives, we have
\begin{equation*}
\begin{aligned}
G_x'(R)=&-hG_x(R)+\frac{1}{e^{hR}}\int_{S(x,R)}tr U_R(y) d\text{Vol} (y),\\
G_x''(R)=&-h^2G_x(R)-2hG_x'(R)+\\
&\frac{1}{e^{hR}}\int_{S(x,R)}-tr \dot U_R(y)+(tr U_R(y))^2 d\text{Vol} (y),\\
G_x'''(R)=&-h^3G_x(R)-3h^2G_x'(R)-3hG_x''(R)+\\
&\frac{1}{e^{hR}}\int_{S(x,R)}tr \ddot{U}_R(y)-3tr \dot U_R(y)tr U_R(y)+(tr U_R(y))^3 d\text{Vol} (y),
\end{aligned}
\end{equation*}
where $U_R(y)$ and $tr U_R(y)$ are the second fundamental form and the mean curvature of $S(x,R)$ at $y$.

Clearly, $tr U_R(y)\to tr U(v(y))$ as $R\to \infty$ uniformly. By Theorem \ref{margulis2},
$$\lim_{R\to \infty}G_x(R)=\lim_{R\to \infty}\frac{s_R(x)}{e^{hR}}=c(x).$$
Combining with Proposition \ref{uniform1}, we have
\begin{equation}\label{e:deri}
\begin{aligned}
\lim_{R\to \infty}G_x'(R)=&-hc(x)+c(x)\int_{SM}tr U dw^s,\\
\lim_{R\to \infty}G_x''(R)=&-h^2c(x)-2h\lim_{R\to \infty}G_x'(R)+c(x)\int_{SM}-tr \dot U+(tr U)^2 dw^s,\\
\lim_{R\to \infty}G_x'''(R)=&-h^3c(x)-3h^2\lim_{R\to \infty}G_x'(R)-3h\lim_{R\to \infty}G_x''(R)+\\
&c(x)\int_{SM}tr \ddot{U}-3tr \dot U tr U+(tr U)^3 d w^s.
\end{aligned}
\end{equation}
Since $\lim_{R\to \infty}G_x'(R)$ exists and $\lim_{R\to \infty}G_x(R)$ is bounded, we have $\lim_{R\to \infty}G_x'(R)=0$. Similarly, considering the second and third derivative, we have
$$\lim_{R\to \infty}G_x''(R)=\lim_{R\to \infty}G_x'''(R)=0.$$
Plugging in \eqref{e:deri}, we have
\begin{enumerate}
  \item $h=\int_{SM} tr U(v) dw^s(v)$,
  \item $h^2=\int_{SM} -tr \dot{U}(v)+(tr U(v))^2 dw^s(v)$,
  \item $h^3=\int_{SM}tr \ddot{U}-3tr \dot U tr U+(tr U)^3 d w^s$.
\end{enumerate}
\end{proof}

\subsubsection{Rigidity}

Recall that $\tilde \mu_x$ is a Borel measure on $S_xX$ (hence descending to $S_xM$) induced by the Patterson-Sullivan measure $\mu_x$ and let us assume that it is normalized, by a slight abuse of the notation. We have the following characterization of $w^s$.
\begin{proposition}\label{measure}
For any continuous $\varphi: SM\to \RR$, we have
$$C\int_{SM} \varphi dw^s=\int_Mc(x)\int_{S_xM}\varphi d\tilde \mu_x(v)d \text{Vol}(x)$$
where $C=\int_M c(x)d\text{Vol}(x)$.
\end{proposition}
\begin{proof}
The idea is to show the right hand side is a harmonic measure up to a normalization. Then the proposition follows from Theorem \ref{unique}. The proof is completely parallel to that of \cite[Proposition 4.1]{Yue} (see also \cite{Yue1,Led}), and hence is omitted.
\end{proof}

\begin{proof}[Proof of Theorem \ref{rigidity}]
By Theorem \ref{formula},
$$h^2=\int_{SM} -tr \dot{U}(v)+(tr U(v))^2 dw^s(v).$$
Recall the Ricatti equation in dimension two:
$$-\dot U+U^2+K=0$$
where $K$ is the Gaussian curvature.
Since now $U$ is just a real number and hence $tr (U^2)=(tr U)^2$. Using Proposition \ref{measure} and Gauss-Bonnet formula we have
\begin{equation*}
\begin{aligned}
h^2=&\int_{SM} -K dw^s=\frac{1}{C}\int_M -c(x)K(x)d \text{Vol}(x)\\
=&\int -Kd \text{Vol}/\text{Vol}(M)=-2\pi E /\text{Vol}(M),
\end{aligned}
\end{equation*}
where $E$ is the Euler characteristic of the surface $M$. By Katok's result \cite[Theorem B]{Ka},
$h^2=-2\pi E /\text{Vol}(M)$ if and only if $M$ has constant negative curvature.
\end{proof}

\subsection{Flip invariance of the Patterson-Sullivan measure}
Let $M$ be a closed $C^\infty$ Riemannian manifold without conjugate points and with continuous Green bundle. Suppose that the geodesic flow has a hyperbolic periodic point. We already see that $y\mapsto b_{\xi}(y,x)$ is $C^2$ in Lemma \ref{Buseregu}. 

For each $x\in X$, denote by $\tilde \mu_x$ both the Borel probability measure on $S_xX$ and $\partial X$ given by the normalized Patterson-Sullivan measure.
Define a measure $w^s$ by
$$C\int_{SM} \varphi dw^s:=\int_Mc(x)\int_{S_xM}\varphi d\tilde \mu_x(v)d \text{Vol}(x)$$
for any continuous $\varphi: SM\to \RR$, where $C=\int_M c(x)d\text{Vol}(x)$.

In view of Proposition \ref{measure}, $w^s$ is a harmonic measure associated to the strong stable foliation $\F^s$, though the uniqueness of harmonic measure is unknown in general. We do not need the uniqueness of harmonic measure in this subsection.

\begin{proposition}\label{2formula}
For $\varphi\in C^1(SM)$, one has
$$\int_{SM}\dot\varphi+(h-tr U)\varphi dw^s=0.$$
\end{proposition}
\begin{proof}
Define a vector field on $M$ by
$$Y(y):=\int_{S_yM}\varphi X(v)d\tilde\mu_y(v)=\int_{S_xM}\varphi X(v)e^{-hb_v(y)}d\tilde\mu_x(v)$$
where $X$ is the geodesic spray. Since $\triangledown b_v=-X$ and $div X=-tr U$, one has
\begin{equation*}
\begin{aligned}
div|_{y=x}Y&=\int_{S_xM}div|_{y=x}\varphi X(v)e^{-hb_v(y)}d\tilde\mu_x(v)\\
&=\int_{S_xM}\dot\varphi+(h-tr U)\varphi d\tilde\mu_x.
\end{aligned}
\end{equation*}
Integrating with respect to $\text{Vol}$ on $M$ and using Green's formula, we have $\int_{SM}\dot\varphi+(h-tr U)\varphi dw^s=0.$
\end{proof}

\begin{proposition}\label{coin}
If $w^s$ is $\phi^t$-invariant, then $M$ is locally symmetric.
\end{proposition}
\begin{proof}
If $w^s$ is $\phi^t$-invariant, by Proposition \ref{2formula}, we have
$$\int_{SM}(h-tr U)\varphi dw^s=0$$
for all  $\varphi\in C^1(SM,\RR)$. It follows that $tr U\equiv h$. Recall that $y\mapsto b_{\xi}(y,x)$ is $C^2$. So $M$ is asymptotically harmonic. By \cite[Theorem 1.2]{Zi}, $M$ is locally symmetric, since $X$ has purely exponential volume growth.
\end{proof}

\begin{lemma}\label{liou}
If for all $x\in M$, $\tilde\mu_x$ is flip invariant, then the Bowen-Margulis measure $m$ coincides with the Liouville measure $\text{Leb}$ on $SM$.
\end{lemma}
\begin{proof}
The proof of \cite[Proposition 2.1]{Yue1} works just as well here. We remark that by uniform visibility, for every pair $\xi\neq \eta\in \pX$, there exists a geodesic $c_{\xi,\eta}$ connecting $\xi$ and $\eta$.
\end{proof}

\begin{lemma}\label{cons}
If for all $x\in M$, $\tilde\mu_x$ is flip invariant, then the Margulis function $c(x)$ is constant.
\end{lemma}
\begin{proof}
Any $\varphi\in C^2(M,\RR)$ can be lifted to a function on $SM$ which we still denote by $\varphi$. Recall that we have identified each weak stable manifold with $X$, so the Laplacian along the weak stable foliation $\bigtriangleup^{0s}$ is exactly the Laplacian $\triangle$ along $X$. By \cite[Lemma 5.1]{Yue3}, $\triangle^{0s} \varphi=\bigtriangleup^s\varphi+\ddot{\varphi}-tr U \dot\varphi$.
Then by definition of $w^s$ and Proposition \ref{2formula},
\begin{equation*}
\begin{aligned}
&\int_M \bigtriangleup \varphi c(x)d\text{Leb}=C\int \bigtriangleup \varphi dw^s\\
=&C\int_{SM}(\bigtriangleup^s\varphi+\ddot{\varphi}-tr U \dot\varphi)dw^s\\
=&C\left(\int_{SM}\bigtriangleup^s\varphi dw^s+\int_{SM}\ddot{\varphi}+(h-tr U) \dot\varphi dw^s-\int_{SM}h\dot\varphi dw^s\right)\\
=&-h\int_{M}c(x)d\text{Leb}(x)\int \dot\varphi(x,\xi) d\tilde\mu_x(\xi).
\end{aligned}
\end{equation*}
Since $d\tilde\mu_x(\xi)=d\tilde\mu_x(-\xi)$ and $\dot\varphi(x,\xi)=-\dot\varphi(x,-\xi)$, we have $$\int_M \bigtriangleup \varphi c(x)d \text{Leb}=0$$
for any $\varphi\in C^2(M,\RR)$. So $c(x)$ must be constant.
\end{proof}

\begin{proof}[Proof of Theorem \ref{flip}]
By the construction, the Bowen-Margulis measure $m$ is flip invariant. By the flip invariance of the partition $\{S_xM\}_{x\in M}$ and the uniqueness of conditional measures, we see that $\bar\mu_x$ is flip invariant for $m\ae x\in M$. It follows that the normalized Patterson-Sullivan measures $\tilde\mu_x$ is flip invariant for $m\ae x\in M$.

We claim that for all $x\in M$, $\tilde\mu_x$ is flip invariant. Indeed, note that for fixed $x$, the density
$$\frac{d\tilde \mu_{y}}{d \tilde\mu_{x}}(\xi)=\frac{c(x)}{c(y)}e^{-h \cdot b_{\xi}(y,x)},$$
is uniformly continuous in $y$. For each continuous function $\varphi: \partial X\to \RR$, its geodesic reflection with respect to $z\in X$ is defined by $\varphi_z(\xi):=\varphi(c_{z,\xi}(-\infty))$. Let $x_k,x\in X$ and $x_k\to x$ as $k\to \infty$. Then by the above continuity,
$$\int_{\partial X}\varphi d\tilde\mu_x=\lim_{k\to \infty} \int_{\partial X}\varphi d\tilde\mu_{x_k}=\lim_{k\to \infty}\int_{\partial X}\varphi_{x_k}d\tilde\mu_{x_k}=\int_{\partial X}\varphi_xd\tilde\mu_{x}.$$
The claim follows.

By Lemma \ref{liou}, the Bowen-Margulis measure $m$ coincides with the Liouville measure, and thus $m$ projects to the Riemannian volume on $M$. By assumption, the conditional measures $\bar\mu_x$ coincides with $\tilde\mu_x$. Moreover, by Lemma \ref{cons}, $c(x)$ is constant. Consequently, we see from definition that $w^s$ coincides with the Bowen-Margulis measure $m$, and hence it is $\phi^t$-invariant. By Proposition \ref{coin}, $M$ is locally symmetric.
\end{proof}

\section{The Hopf-Tsuji-Sullivan dichotomy}
Let $X$ be a simply connected smooth uniform visibility manifold without conjugate points and $\C\subseteq Is(X)$ a discrete group. The \textit{geometric limit set} of $\C$ is defined as $L_\C:=\overline{\C x}\cap \pX$ where $x\in X$ is an arbitrary point. Suppose that $\C\subseteq Is(X)$ is a \textit{non-elementary} discrete group, i.e., the cardinality of $L_\C$ is infinity. We also suppose that $\C$ contains an \textit{expansive isometry} $h\in \C$, i.e., $h$ has an axis $c$ with $\dot c(0)\in \R_0$. We emphasize that in this section $M=X/\C$ is not necessarily compact.

\subsection{The generalized shadow lemma}
For $r>0, c>0$ and $x, y\in X$ we set
\begin{equation*}
\begin{aligned}
&\O_{r,c}^+(x,y):=\{\xi\in \pX: \exists z\in B(x,r) \text{\ such that\ }c_{z,\xi}(\RR_+)\cap B(y, c)\neq \emptyset\},\\
&\O_{r,c}^-(x,y):=\{\xi\in \pX: \forall z\in B(x,r) \text{\ we have \ }c_{z,\xi}(\RR_+)\cap B(y, c)\neq \emptyset\},\\
&\CL_{r,c}(x,y):=\{(\xi,\eta)\in \pX\times \pX: \exists x'\in B(x,r), \exists y'\in B(y,c)\text{\ such that\ }\\
&\quad \quad \quad \quad \quad  \quad \quad \quad \quad \quad   \quad \quad \quad \quad \quad c_{x',y'}(-\infty)=\xi,\ c_{x',y'}(+\infty)=\eta\}.
\end{aligned}
\end{equation*}
It is clear that
\begin{equation*}
\begin{aligned}
&\O_{r,c}^-(x,y)\subset pr_{x}(B(y,c)) \subset \O_{r,c}^+(x,y),\\
&\CL_{r,c}(x,y)\supset \{(\xi,\eta)\in \partial^2X: \eta\in \O_{r,c}^-(x,y), \ \xi\in pr_{\eta}(B(x,r)\}.
\end{aligned}
\end{equation*}

As $\C$ is non-elementary, there exists one element (actually an infinite number) in $\C$ not commuting with $h$ so that conjugation by such an element gives another expansive isometry whose fixed points in $\pX$ are disjoint from those of $h$. By a parallel argument in the proof of \cite[Proposition 2.8]{Bal0}, the geometric limit set of $\C$ is minimal, i.e. $L_\C = \overline{\C\cdot \xi}$ for any $\xi\in L_\C$. This implies that for any open subset $O\subset \pX$ with $O\cap L_\C\neq \emptyset$ there exists a finite set $\L\subset  \C$ depending on $O$ such that 
$$L_\C\subset \bigcup_{\b\in \L}\b O.$$
See the proof of \cite[Lemma 2.8]{Link0} using Proposition \ref{north} which works as well here.
\begin{lemma}\label{finitebeta}
Let $h\in \C$ be the expansive isometry from above, $r_0>0$, and $U, V\subset \overline{X}$ the neighborhoods of $h^-$ and  $h^+$ provided by Proposition \ref{regularneighbor} for $r_0$. Then there
exists a finite set $\L\subset \C$ such that the following holds:
For any $c>0$ there exists $R>1$ large enough such that if $\c\in \C$ satisfies $d(o, \c o) > R$, then for some $\b\in \L$ we have
$$\CL_{r,c}(o, \b\c o)\cap (U\times V)\supset (U\cap \pX)\times \O_{r,c}^-(o, \b\c o).$$
\end{lemma}
\begin{proof}
Th proof uses the cone topology of $\pX$ and is analogous to that of \cite[Proposition 1]{LiP}. The slight difference here is that $h$ is an expansive isometry and hence has width zero.
\end{proof}

Since $\C$ is non-elementary, let $g\in \C$ be another expansive isometry with fixed points in $\pX$ distinct from those of $h$. The following result roughly demonstrates that $\O_{r,c}^-(y,o)$ contains an open neighborhood and hence is rather big, uniformly with respect to $y\in X$ for a suitable large $c$.
\begin{lemma}\label{lowbig}
Given $r>0$ there exist an open neighborhood $O\subset \pX$ of $h^+$, $M\in \NN$ and $c_0 > 0$ such that for all $y\in X$ with $d(y,o)>r+c_0$,
$$h^M(O)\subset \O_{r,c}^-(y,o), \text{\ or \ } g^Mh^M(O)\subset  \O_{r,c}^-(y,o).$$
\end{lemma}
\begin{proof}
The proof is analogous to that of \cite[Proposition 2]{LiP}, which essentially relies on Propositions \ref{regularneighbor} and \ref{north}.
\end{proof}

\begin{lemma}\label{mini}
Let $x\in X$ and $A\subset L_\C$ a $\C$-invariant Borel set. Then $\mu_x(A) > 0$ implies $\mu_x(O\cap A)>0$ for any open set $O\subset \pX$ with $O \cap L_\C \neq \emptyset$.
\end{lemma}
\begin{proof}
See the proof of \cite[Lemma 4.1]{LiP}.
\end{proof}

\begin{proposition}\label{generalizedshadow}
For any $r>0$ there exists a constant $c_0\ge r$ with the following property: If $c\ge c_0$ there exists a constant $D = D(c) > 1$ such that for all $\c\in \C$ with $d(o, \c o)>2c$ we have
\begin{equation*}
\begin{aligned}
\frac{1}{D}e^{-\d_\C d(o,\c o)}\le \mu_o(\O_{r,c}^-(o,\c o))&\le \mu_o(pr_o B(\c o,c))\\
&\le \mu_o(\O_{c,c}^+(o,\c o))\le De^{-\d_\C d(o,\c o)}.
\end{aligned}
\end{equation*}
\end{proposition}
\begin{proof}
Fix $r>0$. According to Lemma \ref{lowbig}, there exist open neighborhood $O\subset \pX$ of $h^+$ and $c_0\ge r$, such that for any $c\ge c_0$, any $\c\in \C$ with $d(\c o,o)>r+c$, there exists $\b\in \{h^M, g^Mh^M\}$ such that
$$\b O\subset \O_{r,c_0}^-(\c o,o)\subset \O_{r,c}^-(\c o,o).$$
Let $q:=\min\{\mu_o(h^MO), \mu_o(g^Mh^MO)\}$, which is positive by Lemma \ref{mini}. Thus for any $\c\in \C$ with $d(o, \c o)>2c$ we have
$$\mu_o( \O_{r,c}^-(\c o,o))\ge \mu_o(\beta O)\ge q>0$$
and
\begin{equation*}
\begin{aligned}
 \mu_o(\O_{r,c}^-(o,\c o))&= \mu_{\c^{-1}o}(\O_{r,c}^-(\c^{-1}o, o))\\
  &\ge e^{-\d_\C d(o,\c o)} \mu_{o}(\O_{r,c}^-(\c^{-1}o, o))\ge e^{-\d_\C d(o,\c o)}q.
\end{aligned}
\end{equation*}

The last inequality in the proposition actually holds for any $\c\in \C$. First, let us notice the following fact: If $x,y\in X$ and $\eta\in \O_{c,c}^+(x,y)$, we have
$$b_\eta(x,y)\ge d(x,y)-4c.$$
Hence if $d(o, \c o)>2c$,
\begin{equation*}
\begin{aligned}
 \mu_o(\O_{c,c}^+(o,\c o))&= \mu_{\c^{-1}o}(\O_{c,c}^+(\c^{-1}o, o))\\
  &\le e^{4c\d_\C}e^{-\d_\C d(o,\c o)} \mu_{o}(\pX).
\end{aligned}
\end{equation*}
If $d(o, \c o)\le 2c$, then $\O_{c,c}^+(o,\c o)=\pX$ and hence
\begin{equation*}
\begin{aligned}
 \mu_o(\O_{c,c}^+(o,\c o))= \mu_{o}(\pX)\le e^{2c\d_\C}e^{-\d_\C d(o,\c o)} \mu_{o}(\pX).
\end{aligned}
\end{equation*}
Let $D=D(c):=\max\{e^{4c\d_\C}\mu_{o}(\pX), 1/q\}$ and we are done.
\end{proof}

\subsection{Properties of the radial limit set}
We define an important subset of geometric limit set. For $c>0$ and $R\gg 1$, set
\begin{equation*}
\begin{aligned}
 L_\C(c,R):=&\bigcup_{\c\in \C, d(o, \c o)>R}pr_o(B(\c o, c)),\\
  L_\C(c):=&\bigcap_{R>0}L_\C(c,R),\\
  L^{\text{rad}}_\C:=&\bigcup_{c>0}L_\C(c).
\end{aligned}
\end{equation*}
We call $L^{\text{rad}}_\C$ the \textit{radial limit set} of $\C$, and it is independent of the choice of the origin $o\in X$.

\begin{lemma}\label{convradial}(\cite[Thereom 5.5]{LLW2})
If $\sum_{\c\in \C}e^{-\d_\C d(o,\c o)}$ converges, then
$$\mu_o(L^{\text{rad}}_\C) = 0.$$
\end{lemma}
\begin{proof}
See the proof of \cite[Thereom 5.5]{LLW2} or the proof of \cite[Lemma 5.1]{LiP} for nonpositive curvature case.
\end{proof}

\begin{proposition}\label{radialmeasure}(\cite[Thereom 5.6]{LLW2})
If $A\subset L^{\text{rad}}_\C$ is a $\C$-invariant Borel subset of $L^{\text{rad}}_\C$, then $\mu_o(A) = 0$ or $\mu_o(A) =\mu_o(\pX)$. 
\end{proposition}
\begin{proof}
See the proof of \cite[Thereom 5.6]{LLW2}. The analogous proof for nonpositive curvature is given in \cite[Proposition 4]{LiP}.
\end{proof}

\begin{proposition}\label{noatom}(\cite[Thereom 5.4]{LLW2})
A radial limit point cannot be a point mass for $m_\C$.
\end{proposition}
\begin{proof}
See the proof of \cite[Thereom 5.4]{LLW2}. In nonpositive curvature, the proof is given in \cite[Proposition 5]{LiP}.
\end{proof}

\subsection{Dynamical properties of the geodesic flow}
The following relation between conservative point and radial limit point can be deduced from the definition.
\begin{lemma}(\cite[Thereom 6.2]{LLW2})\label{conservative}
$\lv\in SM$ is a conservative point, if and only if $c_{v}(+\infty)\in L^{\text{rad}}_\C$ for some (hence any) lift $v\in SX$.
\end{lemma}

\begin{proposition}\label{dic}
Let $M=X/ \C$ be a complete visibility manifold with no conjugate points where $\C$ is a non-elementary discrete subgroup containing an expansive isometry, and $m_\C$ a Bowen-Margulis measure defined on $SM$. We have the following dichotomy.
\begin{enumerate}
    \item $\mu_o(L^{\text{rad}}_\C)=\mu_o(\pX)$, and the geodesic flow is completely conservative with respect t $m_\C$; OR
    \item  $\mu_o(L^{\text{rad}}_\C)=0$, and the geodesic flow is completely dissipative with respect t $m_\C$.
\end{enumerate}
\end{proposition}
\begin{proof}
By Proposition \ref{radialmeasure}, we have the dichotomy that either $\mu_o(L^{\text{rad}}_\C)=\mu_o(\pX)$ or $\mu_o(L^{\text{rad}}_\C)=0$. If $\mu_o(L^{\text{rad}}_\C)=\mu_o(\pX)$, by the local product structure of $m_\C$ and Lemma \ref{conservative}, we have $\Omega_C$ has full $m_\C$-measure in $SM$ and hence $\phi^t$ is completely conservative. Analogous arguments hold for the latter case and $\phi^t$ is completely dissipative.
\end{proof}

The rest of this subsection are devoted to proving the following.
\begin{proposition}\label{divergent}
If $\C$ is nonelementary and contains an expansive isometry, then $\mu_o(L^{\text{rad}}_\C)=0$
implies that $\sum_{\c\in \C}e^{-\d_\C d(o,\c o)}$ converges.
\end{proposition}

We first define for $x\in X$ and $c > 0$ 
\begin{equation}\label{K0}
K^0(x, c) := \{v\in SX: d(x, \pi v) < c \text{\ and\ } x\in H(v)\}.
\end{equation}
We note that in \cite{LiP} for the nonpositive curvature case, $K^0(x, c)$ is defined using a vector $v(x; v^-,v^+)$. More precisely, for a point $x\in X$ and $(\xi, \eta)\in \partial^2X$ we denote $v(x; \xi, \eta)$ the unique element $v\in SX$ whose footpoint is the orthogonal projection of $x$ to the flat strip $(\xi\eta)$ and such that $P(v) = (\xi, \eta)$. This is well defined in nonpositive curvature case. In our setting, due to the possible existence of focal points, $v(x; \xi, \eta)$  is not uniquely defined. Our new definition in \eqref{K0} does not affect anything, in particular the proof of Lemma \ref{Kshadow} below still works.

We need to thicken these sets by the geodesic flow and consider
\begin{equation*}
K(x, c) := \{\phi^s v: v\in K^0(x, c), s\in (-c, c)\}\subset  SB(x, 2c). 
\end{equation*}
Such a set has the important property that the orbit of an arbitrary vector $v\in SX$ under the geodesic flow either does not intersect it at all or spends precisely time $2c$ inside.

\begin{lemma}\label{Kshadow}
    $P(\{K(o,c)\cap \phi^{-t}\c K(\c o,c): t>0\})=\CL_{c,c}(o,\c o)$.
\end{lemma}
\begin{proof}
The proof of \cite[Lemma 6.2]{LiP} still works here under our new definition \eqref{K0}.
\end{proof}

Recall that $h\in \C$ is an expansive isometry with fixed points $h^-, h^+\in \pX$. Without loss of generality, we assume that $o\in Ax(h)$. We fix $\e>0$ and the associated neighborhoods $U=U_\e, V=V_\e\subset \overline{X}$ of $h^-, h^+$ respectively provided by Proposition \ref{regularneighbor}. Let $c_0>r$ be the two constants from the shadow lemma Proposition \ref{generalizedshadow}. Let $\Lambda \subset \C$ be a finite symmetric set provided by Proposition \ref{finitebeta} and
$$\rho:=\{d(o,\b o): \b\in \L\}.$$ 
Fix $c>c_0+\rho+\e>r+\rho+\e$ and from now on consider the set
$$K:=\overline{K(o,c)}.$$

\begin{lemma}\label{divergent1}
There exists a constant $C>0$ such that for $T$ sufficiently large,
$$\int_0^Tdt\int_0^Tds \sum_{\c,\phi\in \C}m(K\cap g^{-t}\c K\cap g^{-t-s}\phi K)\le C\left(\sum_{\c\in \C, d(o, \c o)\le T}e^{-\d_\C d(o, \c o)}\right)^2.$$
\end{lemma}
\begin{proof}
Recall the definition of measure $\lambda_{\xi,\eta}$. It is relatively easier to get a upper bounded for it. If
$$(\xi,\eta)\in P(K\cap g^{-t}\c K\cap g^{-t-s}\phi K)$$
for some $s,t>0$, then
$$\int_0^Tdt\int_0^Tds \lambda_{\xi,\eta}(P^{-1}(\xi,\eta)\cap K\cap g^{-t}\c K\cap g^{-t-s}\phi K)\le (2c)^4.$$

By Lemma \ref{Kshadow}, 
$$P(K\cap g^{-t}\c K\cap g^{-t-s}\phi K)\subset \CL_{2c,2c}(o,\phi o)\subset \pX\times \O^+_{2c,2c}(o, \phi o).$$
Therefor for any $\c, \phi\in \C$,
\begin{equation*}
\begin{aligned}
& \int_0^Tdt\int_0^Tds m(K\cap g^{-t}\c K\cap g^{-t-s}\phi K)\\
\le &16c^4\int_{\CL_{2c,2c}(o,\phi o)}d\mu_o(\xi)d\mu_o(\eta)e^{2\d_\C \b_o(\xi, \eta)}\\
\le &16c^4e^{4c\d_\C}\int_{\CL_{2c,2c}(o,\phi o)}d\mu_o(\xi)d\mu_o(\eta)\\
\le &16c^4e^{4c\d_\C}\int_{\pX}d\mu_o(\xi)\int_{\O^+_{2c,2c}(o, \phi o)}d\mu_o(\eta)\\
\le &16c^4e^{4c\d_\C}\mu_o(\pX)D(2c)e^{-\d_\C d(o,\phi o)}\\
= &C_1e^{-\d_\C d(o,\phi o)}
\end{aligned}
\end{equation*}
where $C_1:=16c^4e^{4c\d_\C}\mu_o(\pX)D(2c)$ only depends on $c$. Then 
\begin{equation*}
\begin{aligned}
& \int_0^Tdt\int_0^Tds \sum_{\c,\phi\in \C} m(K\cap g^{-t}\c K\cap g^{-t-s}\phi K)\\
\le &C_1\sum_{\c,\phi\in \C, d(o,\c o)\le T+4c, d(\c o, \phi o)\le T+4c}e^{-\d_\C (d(o,\c o)+d(\c o,\phi o)-12c)}\\
=&C_1e^{12c\d_\C}\sum_{\c,\phi\in \C, d(o,\c o)\le T+4c, d(o, \a o)\le T+4c}e^{-\d_\C (d(o,\c o)+d(o,\a o))}\\
\le &C_2\left(\sum_{\c\in \C, d(o, \c o)\le T+4c}e^{-\d_\C d(o, \c o)}\right)^2
\end{aligned}
\end{equation*}
where $C_2$ is a constant only depending on $c$. Finally, notice that 
$$\sum_{\c\in \C, T< d(o, \c o)\le T+4c}e^{-\d_\C d(o, \c o)}$$
is uniformly bounded in $T$. This follows from the proof of \cite[Corollary 3.8]{Link0}, which essentially uses the shadow lemma Proposition \ref{generalizedshadow}. We are done.
\end{proof}

\begin{lemma}\label{divergent2}
There exists a constant $C'>0$ such that for $T$ sufficiently large,
$$\int_0^Tdt \sum_{\c\in \C}m(K\cap g^{-t}\c K)\ge C'\sum_{\c\in \C, d(o, \c o)\le T}e^{-\d_\C d(o, \c o)}.$$
\end{lemma}
\begin{proof}
By Lemma \ref{Kshadow} and non-negativity of Gromov product, we have for $\c\in \C$ with $d(o,\c o)>4c$, 
\begin{equation*}
\begin{aligned}
& \int_0^Tdt m(K\cap g^{-t}\c K)\\
=&\int_0^Tdt \int_{\partial^2 X}d\mu_o(\xi)d\mu_o(\eta)e^{2\d_\C \b_o(\xi, \eta)}\lambda_{\xi,\eta}(P^{-1}(\xi,\eta)\cap K\cap g^{-t}\c K)\\
\ge &\int_{\CL_{c,c}(o,\c o)}d\mu_o(\xi)d\mu_o(\eta)\left(\int_0^Tdt \lambda_{\xi,\eta}(P^{-1}(\xi,\eta)\cap K\cap g^{-t}\c K)\right).
\end{aligned}
\end{equation*}

If $(\xi,\eta)\in \CL_{c,c}(o,\c o)\cap (U\times V)$, by Proposition \ref{regularneighbor}, any geodesic connecting $\xi$ and $\eta$ has distance less than $\e$ to $o$, and so is the generalized strip $(\xi\eta)$. Since  $\lambda_{\xi,\eta}$ is supported on  the generalized strip $(\xi\eta)$, we have for $T>d(o,\c o)+4c$ that
\[\int_0^Tdt \lambda_{\xi,\eta}(P^{-1}(\xi,\eta)\cap K\cap g^{-t}\c K)=2c^2.\]
Since $\C$ acts on $SX$ by isometries, the above also holds for $(\xi,\eta)\in \CL_{c,c}(o,\c o)\cap \b(U\times V)$ for any $\b\in \C$. So for any $\b\in \Lambda$,
\begin{equation*}
\begin{aligned}
& \int_0^Tdt m(K\cap g^{-t}\c K)\\
\ge &\int_{\CL_{c,c}(o,\c o)\cap \b(U\times V)}d\mu_o(\xi)d\mu_o(\eta)\left(\int_0^Tdt \lambda_{\xi,\eta}(P^{-1}(\xi,\eta)\cap K\cap g^{-t}\c K)\right)\\
=&2c^2\int_{\CL_{c,c}(o,\c o)\cap \b(U\times V)}d\mu_o(\xi)d\mu_o(\eta).
\end{aligned}
\end{equation*}
By Lemma \ref{finitebeta}, for $\c\in \C$ satisfying $d(o, \c o) > R$, there exists some $\b\in \L$ such that
$$\CL_{r,c}(o, \b^{-1}\c o)\cap (U\times V)\supset (U\cap \pX)\times \O_{r,c}^-(o, \b^{-1}\c o).$$
Since $c>r+\rho+\e>r+\rho$, we also have
$$\CL_{r,c}(o, \b^{-1}\c o)=\b^{-1}\CL_{r,c}(\b o, \c o)\subset \b^{-1}\CL_{r+\rho,c}(o, \c o)\subset \b^{-1}\CL_{c,c}(o, \c o).$$
Thus we have
\begin{equation*}
\begin{aligned}
& \int_0^Tdt m(K\cap g^{-t}\c K)\\
=&2c^2\int_{\b((U\cap \pX)\times \O_{r,c}^-(o, \b^{-1}\c o))}d\mu_o(\xi)d\mu_o(\eta)\\
\ge &2c^2 \mu_o(\b U)\mu_o(\b \O_{r,c}^-(o, \b^{-1}\c o))\\
\ge &2c^2 \mu_o(\b U)e^{-\d_\C d(o, \b^{-1}o)}\mu_o(\O_{r,c}^-(o, \b^{-1}\c o))\\
\ge &2c^2 \mu_o(\b U)e^{-\d_\C d(o, \b^{-1}o)} D(c)e^{-\d_\C d(o, \b^{-1}\c o)}\\
\ge& 2c^2 \min_{\b\in \Lambda}\mu_o(\b U)e^{-2\d_\C \rho} D(c)e^{-\d_\C d(o, \c o)}\\
=& C'e^{-\d_\C d(o, \c o)}.
\end{aligned}
\end{equation*}

Finally, summing over all $\c\in \C$ with $d(o, \c o)\in (R, T-4c)$ we get
$$\int_0^Tdt \sum_{\c\in \C}m(K\cap g^{-t}\c K)\ge C'\sum_{\c\in \C, R<d(o, \c o)\le T-4c}e^{-\d_\C d(o, \c o)}.$$
Notice that both $\sum_{\c\in \C, d(o, \c o)\le R}e^{-\d_\C d(o, \c o)}$ and $\sum_{\c\in \C, T-4c<d(o, \c o)\le T}e^{-\d_\C d(o, \c o)}$ are uniformly bounded. We are done.
\end{proof}

Finally, we can prove Proposition \ref{divergent}.
\begin{proof}[Proof of Proposition \ref{divergent}]
If $\sum_{\c\in \C}e^{-\d_\C d(o, \c o)}$ diverges, by Lemmas \ref{divergent1} and  \ref{divergent2}, we can apply the second Borel-Cantelli lemma and proceeds as in \cite[p. 20]{Rob} to get
\[m_\C(\{v\in SM: \int_0^\infty\chi_{K_\C\cap \phi^{-t}K_\C}(v)dt=\infty\})>0\]
where $K_\C\subset SM$ is the projection of $K$ onto $SM$. It means that $(SM, \phi^t, m_\C)$ is not completely dissipative. However, by Proposition \ref{dic}, $\mu_o(L^{\text{rad}}_\C)=0$ implies that it is completely dissipative. A contradiction and we are done.
\end{proof}

\subsection{Ergodicity for divergent groups}
\begin{theorem}[Hopf’s individual ergodic theorem, cf. \cite{Ho00}]\label{indivi}
Assume that  $(\Omega, \phi^t, \mu)$ is conservative, and let $\rho\in L^1(\mu)$ be a function which is strictly positive $\mu$-almost
everywhere. Then for any function $f\in L^1(\mu)$, the limits
$$f^{\pm}(\omega)=\lim_{T\to \infty}\frac{\int_0^Tf(\phi^{\pm t}(\omega))dt}{\int_0^T\rho(\phi^{\pm t}(\omega))dt}$$
exist and are equal for $\mu \ae \omega\in \Omega$. Moreover,
\begin{enumerate}
    \item  the functions $f^+, f^-$ are measurable and flow invariant,
    \item $\rho\cdot f^+, \rho\cdot f^-\in L^1(\mu)$,
    \item for every bounded measurable flow-invariant function $h$ we have
    $$\int_{\Omega}\rho(\omega)f^{\pm}(\omega)h(\omega)dM(\omega)=\int_{\Omega}f(\omega)h(\omega)d\mu(\omega),$$
    \item $(\Omega, \phi^t, \mu)$ is ergodic if and only if for every $f\in L^1(\mu)$, $f^+$ is constant $\mu$-almost
everywhere. 
\end{enumerate}
\end{theorem}

The following lemma is crucial to apply Hopf argument.
\begin{lemma}\label{close}
Let $v\in \R_0\subset SX$ be a recurrent vector. Then for any $w\in \F^s(v) \subset SX$, there exist a sequence $\{\alpha_{n}\}^{\infty}_{n=1}\subset \Gamma$ and a sequence of time $\{t_{n}\}^{\infty}_{n=1}$ with $t_{n}\rightarrow +\infty$ such that $\lim_{n\to \infty}d\alpha_{n}\phi^{t_n}v=v$ and
$$\lim_{n\to +\infty}d(\phi^{t_{n}}w,\phi^{t_{n}}v)= 0.$$
If moreover $X$ has bounded asymptote, then 
$$\lim_{t\to +\infty}d(\phi^{t}w,\phi^{t}v)= 0.$$
Similar results hold for negatively asymptotic vectors.
\end{lemma}
\begin{proof}
Since $v$ is a recurrent vector in $SX$, there exist a sequence $\{\alpha_{n}\}^{\infty}_{n=1}\subset \Gamma$ and a sequence of time $\{t_{n}\}^{\infty}_{n=1}$ with $t_{n}\rightarrow +\infty$, such that
\begin{equation}\label{recv2}
d\alpha_{n}(\phi^{t_n}v)\to v,~~~n\to +\infty.
\end{equation}

Assume the contrary that there exists $a>0$ such that  
$$\liminf_{ n\to +\infty}d(\phi^{t_{n}}w,\phi^{t_{n}}v)>a.$$
By taking a subsequence, let us assume that $\lim_{ n\to +\infty}d(\phi^{t_{n}}w,\phi^{t_{n}}v)>a.$
Since $w\in \F^s(v)$, $d(\phi^{t}(w), \phi^{t}(v))\le D, \forall t\ge 0$ for some $D>0$. 
Then for each $t_n$ in \eqref{recv2}, we have for any $s \in [-t_{n},+\infty)$,
\begin{equation}\label{e:recc2}
\begin{aligned}
&d(\phi^{t_{n}+s}(w),\phi^{t_{n}+s}(v))= d(d\alpha_{n}\circ\phi^{t_{n}+s}w,d\alpha_{n}\circ\phi^{t_{n}+s}v)\\
= & d(\phi^{s}\circ d\alpha_{n}\circ\phi^{t_{n}}w,\phi^{s}\circ d\alpha_{n}\circ\phi^{t_{n}}v).      
\end{aligned}
\end{equation}
Note that $d(\phi^{t_{n}}(w),\phi^{t_{n}}(v))\leq D$, and $d\alpha_{n}(\phi^{t_{n}}v)\rightarrow v$.
So we know that for any $\epsilon>0$ there is an integer $N>0$ such that for all $n>N$, $$d(v,d\alpha_{n}(\phi^{t_{n}}w))\leq D+\epsilon.$$
This implies that the set $\{d\alpha_{n}(\phi^{t_{n}}w)\}$ has an accumulate point. Without loss of generality, we assume $d\alpha_{n}(\phi^{t_{n}}w)\rightarrow w'$. Letting $n\rightarrow +\infty$ in \eqref{e:recc2}, we get that
$$d(\phi^{s}w',\phi^{s}v) \leq  D, ~~~\forall~s \in \mathbb{R}.$$

On the other hand, 
\begin{equation*}
\begin{aligned}
d(w', v)=\lim_{n\to \infty}d(d\alpha_{n}(\phi^{t_{n}}w),d\alpha_{n}(\phi^{t_{n}}v))=\lim_{n\to \infty}d(\phi^{t_{n}}w,\phi^{t_{n}}v)>a.
\end{aligned}
\end{equation*}
Thus $w'\neq v$. It follows that the geodesics $c_{v}$  and $c_{w'}$ bound a generalized strip,
which contradicts to the assumption that $v\in \R_0$. 

Assume that $X$ has bounded asymptote. Since $w\in \F^s(v)$, there exists $C>0$ such that for any $0<t_n<t$, 
$$d(\phi^{t}w,\phi^{t}v)\le Cd(\phi^{t_{n}}w,\phi^{t_{n}}v).$$
Thus $\lim_{t\to +\infty}d(\phi^{t}w,\phi^{t}v)= 0.$ We are done with the proof.
\end{proof}

\begin{lemma}
The function
$\tilde \rho: SX\to \RR$ by 
$$\tilde \rho(u):=e^{-4\d_\C d(\pi u, \C o)}, \ \forall u\in SX$$
descends to a function $\rho: SM\to \RR$ and belongs to $L^1(m_\C)$. Moreover, for any $u,v\in SX$ with $d(\pi u, \pi v)\le 1$, 
\begin{equation}\label{rho}
|\tilde\rho(u)-\tilde\rho(v)|\le \tilde\rho(u)4\d_\C d(\pi u, \pi v)e^{4\d_\C}.    
\end{equation}
\end{lemma}
\begin{proof}
We only need to verify $\rho \in L^1(m_\C)$ since other statements are straightforward. For any $(\xi, \eta)\in \partial^2X$ and $R\ge 1$,
$$\lambda_{\xi,\eta}(P^{-1}(\xi,\eta)\cap SB(o,R))\le 2R.$$
If $P^{-1}(\xi,\eta)\cap SB(o,R)\neq \emptyset$, then $\beta_{o}(\xi,\eta)\le 2R$. Thus
$$m(SB(o,R))\le 2Re^{2\delta_\C R}.$$
Let $D_\C$ denote the Dirichlet domain for $\C$ with center $o$. Then $\tilde \rho(u)=e^{-4\d_\C d(\pi u, o)}$ for all $u\in S\overline{D_\C}$. Denote 
$$W(R):=SB(o,R)\setminus SB(o,R-1)\cap S\overline{D_\C}.$$ 
Then
\begin{equation*}
\begin{aligned}
\int_{W(R)}\tilde\rho(u)dm(u)\le e^{-4\d_\C(R-1)}\int_{SB(o,R)}dm(u)\le CRe^{-2\d_\C R}\le Ce^{-\d_\C R}
\end{aligned}
\end{equation*}
for some constant $C>0$. Hence $\rho \in L^1(m_\C)$.
\end{proof}
\begin{proposition}\label{ergodicconservative}
Assume that $X$ has bounded asymptote and $\C$ contains a regular isometry $h$. If $(SM, (\phi^t)_{t\in \RR}, m_\C)$ is completely conservative, then it is ergodic.
\end{proposition}
\begin{proof}
Let $f\in C_c(SM, \RR)$ and $\tilde f\in C(SX,\RR)$ be a lift to $SX$. Let $K=\supp (f)\subset SM$. Since $(SM, (\phi^t)_{t\in \RR}, m_\C)$ is completely conservative, Hopf's individual ergodic theorem states that for $m \ae u\in SX$, the limits
$$\tilde{f}^{\pm}(u)=\lim_{T\to \infty}\frac{\int_0^T\tilde{f}(\phi^{\pm t}(u))dt}{\int_0^T\tilde{\rho}(\phi^{\pm t}(u))dt}$$
exist and are equal.

Let $h^+,h^-$ denote the attractive and repulsive fixed point of the regular isometry $h\in \C$. By Proposition \ref{crucial}, there exist open neighborhoods $U$ of $h^-$ and $V$ of $h^+$ in $X$ such that each $\xi\in U$ and $\eta\in V$ can be joined by a unique regular geodesic $c_{\xi,\eta}$. We define
\begin{equation*}
\begin{aligned}
\Omega(U,V):=&\{w\in SX: w^-\in U, w^+\in V\},\\
\Omega^{\text{rec}}(U,V):=&\{w\in \Omega(U,V): w \text{\ is recurrent}\},\\
\Omega'(U,V):=&\{w\in \Omega^{\text{rec}}(U,V): \tilde{f}^{+}(u)=\tilde{f}^{-}(u)\}.
\end{aligned}
\end{equation*}
\begin{lemma}\label{equalvalue}
Let $v\in \Omega^{\text{rec}}(U,V)$ and $w\in \F^s(v)$. Then $\tilde{f}^{+}(w)=\tilde{f}^{+}(v)$. Analogous statement holds for $\F^u$ and $\tilde{f}^{-}$.
\end{lemma}
\begin{proof}
We have $v\in \R_1$ by  Proposition \ref{crucial}. In fact by Lemma \ref{recreg}, $w\in \R_1$. By Lemma \ref{close},
$$\lim_{t\to +\infty}d(\phi^{t}w,\phi^{t}v)= 0.$$
Then the conclusion of the lemma follows from the continuity of $f$ and \eqref{rho} applied to $v$ and $w$.   
\end{proof}

By the local product structure of $m$, there exists a vector $w\in \Omega'(U,V)$ with $w^-=\xi$ such that 
$$G_\xi:=\{\eta\in V: \exists u\in \Omega'(U,V) \text{\ with\ }u^-=\xi, u^+=\eta\}$$
has full $\mu_o$-measure in $V$. Let us fix such $w$ and $\xi$ in the following discussion.

\begin{lemma}\label{equa}
$\tilde f^+$ is constant $m \ae$ on $\Omega'(U,V)$.
\end{lemma}
\begin{proof}
For $m \ae v\in \Omega'(U,V)$, we have $v^+\in G_\xi$. Let $u\in \Omega'(U,V)$ be such that $u^-=\xi=w^-$ and $u^+=v^+$. Then there exist $s_1,s_2\in \RR$ such that $\phi^{s_1}u\in \F^s(v)$ and $\phi^{s_2}u\in \F^u(w)$. Then by Lemma \ref{equalvalue},
\[\tilde f^+(v)=\tilde f^+(\phi^{s_1}u)=\tilde f^+(\phi^{s_2}u)=\tilde f^-(\phi^{s_2}u)=\tilde f^-(w)=\tilde f^+(w).\]
\end{proof}
Let $V':=\{\eta\in V: \exists u\in \Omega'(U,V) \text{\ with\ }u^+=\eta\}$. Using Lemma \ref{equalvalue} again and Lemma \ref{equa}, we see that $\tilde f^+$ is constant $m \ae $ on $$\{v\in SX: v^+\in V'\}.$$
Consider the set $Y:=\bigcup_{\c\in \C}\c V'$. Let $Z:=\{v\in SX: v^+\in Y\}$.
\begin{lemma}
$Z$ has full $m$-measure in $SX$.
\end{lemma}
\begin{proof}
We have $\mu_o(Y\cap V)=\mu_o(V)$. By the conformality, we have $\mu_x(Y\cap V)=\mu_x(V)$ for every $x\in X$. For any $\eta\in L_\C^{\text{rad}}$, there exists an open neighborhood $O(\eta)\subset \pX$ which can be mapped by an element of $\C$ to $V$. Thus  $\mu_o(Y\cap O(\eta))=\mu_o(O(\eta))$ for each $\eta\in L_\C^{\text{rad}}$. So $\mu_o(Y)=\mu_o(L_\C^{\text{rad}})$. It follows that $Z$ has full $m$-measure.
\end{proof}
It is clear that $\tilde f^+$ is $m \ae$ on $Z$, so it is constant $m \ae$ on $SX$.

Finally, notice that $C_c(SM,\RR)$ is dense in $L^1(m_\C)$. So $\phi^t$ is ergodic with respect $m_\C$ by Hopf’s individual ergodic Theorem \ref{indivi}, which completes the proof of the proposition.
\end{proof}

\begin{proof}[Proof of Theorem \ref{HTS}]
By Proposition \ref{dic}, we have the dichotomy. In each case, (i)$\Leftrightarrow$(ii) follows from Lemma \ref{convradial} and Proposition \ref{divergent}. (ii)$\Leftrightarrow$(iii) follows from Proposition \ref{dic}.

If furthermore $X$ has bounded asymptote and $\C$ contains a regular isometry, in the second case, (iii)$\Rightarrow$(iv) follows from Proposition \ref{ergodicconservative} and (iv)$\Rightarrow$(iii) is straightforward since $m_\C$ is not supported on a single orbit.
\end{proof}

\ \
\\[-2mm]
\textbf{Acknowledgement.}
This work is supported by National Key R\&D Program of China No. 2022YFA1007800 and NSFC No. 12071474.



\end{document}